%% file: manifold_langevin_arxiv_20.tex
\documentclass[11pt]{article}
\usepackage{url,ifthen}
\usepackage[round]{natbib}
\usepackage{multirow}
\usepackage[margin=1.1in]{geometry}
\usepackage{nicefrac}
\usepackage{xspace}
\usepackage{graphicx}
\usepackage{enumitem}
\usepackage[usenames]{color}
\usepackage{fullpage}
\usepackage{xspace}
\definecolor{DarkGreen}{rgb}{0.1,0.5,0.1}
\definecolor{DarkRed}{rgb}{0.5,0.1,0.1}
\definecolor{DarkBlue}{rgb}{0.1,0.1,0.5}
\usepackage[small]{caption}
\usepackage{float}
\usepackage{balance}
\usepackage{amsmath}
\usepackage{amsfonts}
\usepackage{amssymb}
\usepackage{amsthm}
\usepackage{mathtools}
\usepackage{blkarray}

\usepackage{tikz}
\newcommand*\circled[1]{\tikz[baseline=(char.base)]{
    \node[shape=circle,draw,inner sep=1pt] (char) {#1};}}

\newtheorem{theorem}{Theorem}
\newtheorem*{namedtheorem}{\theoremname}
\newcommand{\theoremname}{testing}

\newtheorem{thm}[theorem]{Theorem}

\newtheorem{lem}{Lemma}

\newtheorem{prop}[theorem]{Proposition}

\newtheorem*{question*}{Question}

\newtheorem{conds}{Condition}
\theoremstyle{definition}

\newtheorem{defn}[theorem]{Definition}

\newtheorem*{remark*}{Remark}

\theoremstyle{plain}
\newtheorem{Alg}{Algorithm}

\definecolor{DarkGreen}{rgb}{0.1,0.5,0.1}
\definecolor{DarkRed}{rgb}{0.5,0.1,0.1}
\definecolor{DarkBlue}{rgb}{0.1,0.1,0.5}
\usepackage[pdftex]{hyperref}
\hypersetup{
    unicode=false,          
    pdftoolbar=true,        
    pdfmenubar=true,        
    pdffitwindow=false,      
    pdfnewwindow=true,      
    colorlinks=true,       
    linkcolor=DarkGreen,          
    citecolor=DarkGreen,        
    filecolor=DarkGreen,      
    urlcolor=DarkBlue,          
}

\newcommand{\ignore}[1]{}

\renewcommand{\Pr}{\mathop{\bf Pr\/}}                    
\newcommand{\E}{\mathop{\bf E\/}}

\newcommand{\set}[1]{\left\{ \#1 \right\}}

\makeatletter
\floatstyle{ruled}
\newfloat{fragment}{H}{lop}
\floatname{fragment}{Algorithm}
\renewcommand{\floatc@ruled}[2]{\vspace{2pt}{\@fs@cfont \#1.\:} \#2 \par
 \vspace{1pt}}
\makeatother

\title{Fast Convergence for Langevin Diffusion\\ with Manifold Structure}
\author{Ankur Moitra \thanks{Department of Mathematics, Massachusetts Institute of Technology. Email: {\tt moitra@mit.edu}. This work was supported in part by NSF CAREER Award CCF-1453261, NSF Large CCF-1565235, a David and Lucile Packard Fellowship, an Alfred P. Sloan Fellowship and an ONR Young Investigator Award.} \and Andrej Risteski \thanks{Machine Learning Department, Carnegie Mellon University. Email: {\tt aristesk@andrew.cmu.edu}}}
\date{\today}

\include{defs}
\def\shownotes{0}  
\ifnum\shownotes=1
\newcommand{\authnote}[2]{{$\ll$\textsf{\footnotesize #1 notes: #2}$\gg$}}
\else
\newcommand{\authnote}[2]{}
\fi
\newcommand{\Anote}[1]{{\color{magenta}\authnote{Andrej}{{#1}}}}

\begin{document}

\maketitle
\abstract{}
In this paper, we study the problem of sampling from distributions of the form $p(x) = e^{-\beta f(x)}/Z$ where $Z$ is the normalizing constant and $\beta$ is the inverse temperature, for some  function $f$ whose values and gradients we can query.  This mode of access to $f$ is natural in the scenarios in which such problems arise, for instance sampling from posteriors in parametric Bayesian models and energy-based generative models. Classical results \citep{bakry1985diffusions} show that a natural Markov process, Langevin diffusion, mixes rapidly when $f$ is convex. Unfortunately, even in simple examples, the applications listed above will entail working with functions $f$ that are nonconvex \---- for which sampling from $p$ may in general require an exponential number of queries \citep{risteski2018}.

In this paper, we focus on an aspect of nonconvexity relevant for modern machine learning applications: existence of invariances (symmetries) in the function $f$, as a result of which the distribution $p$ will have \emph{manifolds} of points with equal probability.  First, we give a recipe for proving mixing time bounds for Langevin diffusion as a function of the geometry of these manifolds. Second, we specialize this recipe to functions exhibiting matrix manifold structure: we give mixing time bounds for classic matrix factorization-like Bayesian inference problems where we get noisy measurements $\mathcal{A}(XX^T), X \in \mathbb{R}^{d \times k}$ of a low-rank matrix, for a linear ``measurements'' operator $\mathcal{A}$---thus $f(X) = \|\mathcal{A}(XX^T) - b\|^2_2, X \in \mathbb{R}^{d \times k}$, and $\beta$ the inverse of the standard deviation of the noise. 

This setting has compelling statistical motivations: sampling posteriors of distributions induced by noisy measurements. Additionally, such functions $f$ are invariant under orthogonal transformations: arguably the simplest family of symmetries relevant for practice. The problems we tackle include matrix factorization ($\mathcal{A}$ is the identity map), matrix sensing ($\mathcal{A}$ collects the measurements), matrix completion ($\mathcal{A}$ is the projection operator to the visible entries). Finally, apart from sampling, Langevin dynamics is a popular toy model for studying stochastic gradient descent. Along these lines, we believe that our work is an important first step towards understanding how SGD behaves when there is a high degree of symmetry in the space of parameters the produce the same output.  

\newcommand{\distance}{\textsf{dist}}
\newcommand{\cosT}{\cos \theta}
\newcommand{\Us}{U^{*}}
\newcommand{\Q}{\mathbb{Q}}
\newcommand{\Gl}[1]{\mathcal{GL}(#1)}
\newcommand{\Real}{\mathbb{R}}
\newcommand{\Integer}{\mathbb{Z}}
\newcommand{\dd}{\rm{d}}
\newcommand{\PP}{\rm{P}}
\newcommand{\maxdist}{100 \frac{\sqrt{N} \log d}{\beta \sigma^2_{\min} }}
\newcommand{\mainpc}{\mathbf{C}_{pc}}
\newcommand{\damp}{\frac{1}{16}\sigma^2_{\min}}
\newcommand{\argmin}{\operatorname{argmin}}
\newcommand{\numm}{L}
\newcommand{\miss}{\epsilon}
\newcommand{\lbdmf}{k^8 \kappa^8 \left(\frac{1}{\sigma_{\min}}\right)^6 (d \log d \log (1/\miss))^3}

\newcommand{\lbdms}{k^8 \kappa^8 \left(\frac{1}{\sigma_{\min}}\right)^6 (d \log \numm \log (1/\miss))^{3}}

\newcommand{\lbdmc}{\left(d k^3 \log d \log(1/\miss)\right)^{3}\frac{\kappa^{18}}{\sigma^2_{\min} p^6} }

\newcommand{\ripconst}{\frac{1}{20}}

\newcommand{\DM}{\textsf{DecideMove}}
\newcommand{\Grid}{\textsf{Grid}}
\newcommand{\coverRatio}{k}
\newcommand{\distanceRatio}{\gamma}
\newcommand{\1}{\mathbb{1}}
\newcommand{\loss}{\textsf{loss}}
\newcommand{\tO}{\tilde{O}}
\newcommand{\tf}{\tilde{f}}
\newcommand{\tF}{\tilde{F}}
\newcommand{\tR}{\tilde{R}}
\newcommand{\tx}{\tilde{x}}
\newcommand{\Regret}{\text{Regret}}
\newcommand{\valMedian}{{val}}
\newcommand{\convex}{\textsf{conv}}
\newcommand{\cube}{\mathcal{Q}}
\newcommand{\oracle}{\mathcal{O}}
\newcommand{\toracle}{\tilde{\mathcal{O}}}
\newcommand{\valu}{\textsf{Val}}
\newcommand{\xint}{x_{\textsf{m}}}
\newcommand{\an}[1]{\textbf{\textcolor{blue}{AN: #1}}}
\newcommand{\yz}[1]{\textbf{\textcolor{red}{YZ: #1}}}

\newcommand{\LRA}{\set{A}}
\newcommand{\vLRA}{v}
\newcommand{\varLRA}{\sigma}
\newcommand{\true}{\textsf{True}}
\newcommand{\epo}{\tau}
\newcommand{\Epo}{\Gamma}
\newcommand{\errorEstimation}{\sigma}
\newcommand{\NE}{\textsf{ShrinkSet}}
\newcommand{\LCE}{\textsf{LCE}}
\newcommand{\MCVE}{\textsf{MCVE}}
\newcommand{\FLCE}{F_{\LCE}}
\newcommand{\tfLCE}{\tf_{\LCE}}
\newcommand{\blowUp}{\alpha}
\newcommand{\ellBlowUp}{\kappa}
\newcommand{\centerShrink}{\beta}
\newcommand{\levelSetValue}{\ell}
\newcommand{\goto}{\textsf{goto}}
\newcommand{\convexSet}{\mathcal{K}}
\newcommand{\hconvexSet}{\hat{\convexSet}}
\newcommand{\gridScale}{\alpha}
\newcommand{\centerPoint}{x}
\newcommand{\radiusSearch}{r}
\newcommand{\extendRatio}{\gamma}
\newcommand{\twoExtendRatio}{2 \gamma}
\newcommand{\numEpoch}{d \log T}
\newcommand{\dnumEpoch}{d^2 \log T}
\newcommand{\twoNumEpoch}{2 d \log T}
\newcommand{\errnoise}{\varepsilon_{\textsf{noise}}}
\newcommand{\err}{\varepsilon_{\textsf{error}}}
\newcommand{\Ball}{\mathbb{B}}
\newcommand{\rY}{\mathbb{Y}}
\newcommand{\rS}{\mathbb{S}}
\newcommand{\gl}{\gamma_{l}}
\newcommand{\gu}{\gamma_{u}}
\newcommand{\blow}{\tau}
\newcommand{\tPi}{\tilde{\Pi}}
\newcommand{\diag}{\textsf{diag}}
\newcommand{\diam}{\textsf{diam}}
\newcommand{\PO}{\textsf{Projection Oracle }}
\newcommand{\GWone}{\mathcal{G}}
\newcommand{\GWtwo}{\mathcal{W}}
\newcommand{\sff}{\mathrm{I\!I}}
\newcommand{\lopt}{\mathbf{E}}
\newcommand{\so}{\mathrm{SO}(k)}
\newcommand{\mintemp}{\Omega\left((dk)^{5/2} \kappa^{40}\right)}
\newcommand{\minnoise}{O\left(\min\left( \beta \frac{\sigma^4_{\min}}{d^{5/2} \log d k^{11/2} \sigma^2_{\max}}, \frac{\sigma_{\min}^2}{\sqrt{k}}\right)\right)}

\newcommand{\symm}{\mbox{\emph{Sym}}}
\newcommand{\skewm}{\mbox{\emph{Skew}}}
\newcommand{\vol}{\textsc{Vol}}
\newcommand{\bdpc}{\frac{1}{k\sigma^2_{\min}}}

\newpage
\input{introduction3_neurips}
\input{overview_rescaled} 
\input{notation_extrinsic_rescaled}

\input{heuristic_vaguer_rescaled}

\input{torus}

\input{matrix_fact}
\input{concentration_noiseless_rescaled}
\input{functionlevel_1_rescaled}

\input{curvature_level}  
\input{gradtoval}

\input{pf-overview_rescaled}  
\input{discretization}

\section{Conclusion}
We considered the problem of sampling from a distribution using Langevin dynamics, in cases where the distribution is not log-concave, and the distribution $p$ encodes has symmetries. 
We draw out the interaction between the geometry of the manifold and the mixing time, via tools that span stochastic differential equations and differential geometry. We hope that this will inspire researchers to take a closer look at the algorithmic relevance of curvature.
\bibliographystyle{plainnat}
\bibliography{langevin}
\appendix 
\input{helper_noiseless_2}
\end{document}

%% file: introduction3_neurips.tex
\vspace{-0.4cm}
\section{Introduction}


In this paper, we study the problem of sampling from a distribution
$p(X) = \frac{e^{-\beta f(X)}}{Z}$ where $Z$ is the normalizing constant, for some particular families of functions $f(X)$ that are {\em nonconvex}, and we can access $f$ through a value and gradient oracle. 
This problem is the sampling equivalent to the classical setup of minimizing a function $f$, given access to the same oracles, which is the usual sandbox in which query complexity of optimization can be quantified precisely. 

Mirroring what happens for optimization, when $f(X)$ is convex (i.e. $p(X)$ is logconcave), there are a variety of algorithms for efficiently sampling from $p(X)$. Beyond that, however, the problem is in general hard: \cite{risteski2018} prove an exponential lower bound on the number of queries required. Nevertheless, the non-logconcave case is relevant in practice because of its wide-ranging applications:

\begin{enumerate}[leftmargin=*]

\item {\bf Bayesian inference:} In instances when we have a prior on a random variable $X$, of which we get noisy observations, the posterior distribution. $\beta$ is called the inverse temperature and depends on the level of noise in the model: when $\beta$ is large, the distribution places more weight on the $X$'s close to the observation as measured by $f(X)$; when $\beta$ is small, it samples from a larger entropy distribution around the observation. 

We will consider natural instances in this paper, where we get ``measurements'' $\mathcal{A}(XX^T)$ of a low rank matrix, perturbed by Gaussian noise---subsuming problems like noisy matrix factorization, matrix sensing, matrix completion, but where our goal is to sample from the posterior rather than merely obtaining a single low-error estimate. 


\item {\bf Sampling in energy-based models:} Many recent state-of-the-art generative models (under a variety of performance metrics), especially for the domain of vision \citep{song2019generative, du2019implicit} are based on the \emph{energy-model} paradigm: they are parametrized as $p(X) \propto e^{-f(X)}$, for a parametric function $f$ (e.g. a neural network). Samples in such models are drawn by running Langevin dynamics, and manifolds of equiprobable points are a very natural structural assumption: image distributions have a rich group of symmetries (e.g. rotations/translations). 

The distributions we will provide guarantees for in this paper all have invariance under orthogonal transformations---arguably the simplest group of symmetries of relevance to practice. Furthermore, our general recipe will elucidate how the geometry of the distribution governs mixing time.

\end{enumerate}


Towards exploring the landscape of tractable distributions we can sample from, for which $f(X)$ is nonconvex, we ask:

\begin{question*}
Are there statistically and practically meaningful families of nonconvex functions $f(X)$ where we can provably sample from $p(X)$ in polynomial time?
\end{question*}

\noindent The aspect of $f(X)$ we wish to capture in this paper is the existence of \emph{symmetries}, motivated by applications above. Taking inspiration from the literature on nonconvex optimization, we consider the case when $f$ is the objective corresponding to relatives of {\em noisy low rank matrix factorization}, which is invariant under \emph{orthogonal transforms}---e.g. {\em matrix completion} and {\em matrix sensing}. 

When we can query the values and gradients of $f(X)$, there is a natural algorithm for sampling from $p(X)$ called Langevin dynamics. In its continuous form, it is described by the following stochastic differential equation 
$dX_t = -\beta \nabla f(X_t) dt + \sqrt{2} dB_t $
where $B_t$ is Brownian motion of the appropriate dimension. It is well known that under mild conditions on $f(X)$, the stationary distribution is indeed $p(X)$. When $p(X)$ is log-concave Langevin dynamics mixes quickly \citep{bakry1985diffusions}. 

We remark that in order to actually run this algorithm, we need a version that takes discrete rather than continuous steps: $X_{t+1} = X_t - \eta \beta \nabla f(X_t) + \sqrt{2 \eta} \xi$, where $\xi \sim N(0,I_d)$  
As we take the limit of $\eta \to 0$, we recover the Langevin dynamics. This is yet another motivation for studying Langevin dynamics beyond log-concavity: it is often used as a representative model for studying the behavior of stochastic gradient descent \citep{zhang2017hitting, shwartz2017opening}. Thus, studying Langevin dynamics when $p(X)$ is not logconcave can reveal what types of solutions stochastic gradient descent spends time close to when $f(X)$ is nonconvex. 

\vspace{-0.3cm}
\section{Overview of Results}
\label{l:setupresults}
\vspace{-0.1cm}
Our first contribution is to formulate a general recipe for bounding the mixing time of Langevin diffusion in the presence of manifold structure. Precisely, we study the general problem of sampling from the conditional distribution of $e^{-\beta f(X)}$, restricted to be close to a manifold $\mathbf{M}$ that is a level set of $f(X)$ and has the property that all of its points are local minima \--- i.e. 
$$\forall X \in \mathbf{M}: \nabla f(X) = 0, \nabla^2 f(X) \succeq 0, f(X) = s_0$$

Towards stating the result somewhat informally at least for now, consider an arbitrary point $X_0 \in \mathbf{M}$, and denote the ``norm-bounded'' normal space at $X_0$: $\mathbf{B} = \{\Delta: \Delta \in N_{X_0}(\mathbf{M}), \|\Delta\|_2 \leq s\}$. 
Furthermore, we assume that $\forall X \in \mathbf{M}$, there is a differentiable bijection
$\phi_X: \mathbf{B} \to \{\Delta: \Delta \in N_{X}(\mathbf{M}), \|\Delta\|_2 \leq s \}$
that ``transports'' the normal space at $X_0$ to the normal space at $X$. With this in mind, it's natural to consider the ``level set'' corresponding to $\Delta$, namely $\mathbf{M}^{\Delta}:= \{X + \phi_X(\Delta): X \in \mathbf{M}\}$. 
Finally let $\tilde{p}^{\Delta}(X)$ denote the restriction of $p(X)$ to $\mathbf{M}^{\Delta}$ (with a suitable change of measure correction that comes from the coarea formula) and let $q(\Delta)$ denote the total weight that $p(X)$ places on each $\mathbf{M}^{\Delta}$ (with the same correction, again coming from the coarea formula). We show the following:

\begin{thm}[Informal] 
Suppose the following conditions hold:
\begin{enumerate}

\item[(1)] (Nearness to the manifold): When initialized close to $\mathbf{M}$, the Langevin dynamics stay in some neighborhood $\mathcal{D} = \{X: \min_{X' \in \mathbf{M}} \|X - X'\|_2 \leq s\}$ of $\mathbf{M}$ up to time $T$ with high probability. 

\item[(2)] (Poincar\'e inequality along level sets): The distributions $\tilde{p}^{\Delta}$ for all $\Delta \in \mathbf{B}$ have a Poincar\'e constant bounded by $C_{\mbox{level}}$

\item[(3)] (Poincar\'e inequality across level sets): The distribution $q$ has a Poincar\'e constant bounded by $C_{\mbox{across}}$.

\item[(4)] (Bounded change of manifold probability): If we denote by $G_{\Delta}: \mathbf{M} \to \mathbf{M}^{\Delta}$ the map $G_{\Delta}(X) = X + \phi_X(\Delta)$, for all $X \in \mathbf{M}$ and $\Delta \in \mathbf{B}$, the relative change (with respect to $\Delta$) in the manifold density is bounded\footnote{Note, the gradient is for a function defined on the manifold $\mathbf{B}$. See Definition \ref{d:deronman}}: 
$$\left\|\frac{\nabla_{\mathbf{B}} \left(p^{\Delta}(X + \phi_X(\Delta)) \mbox{det}\left((dG_{\Delta})_X\right)\right)}{p^{\Delta}(X + \phi_X(\Delta)) \mbox{det}\left((dG_{\Delta})_X\right)}\right\|_2 \leq C_{\mbox{change}}$$  
\end{enumerate}

Then Langevin dynamics run for time $O\left(\max\left(1, C_{\mbox{level}}\right) \max\left(1,C_{\mbox{across}}\right)\max\left(1, C^2_{\mbox{change}}\right)\right)$
outputs a sample from a distribution that is close in total variation distance to the conditional distribution of $p(X)$ restricted to $\mathcal{D}$ with high probability. 
\label{t:inform1}
\end{thm}

\noindent For a formal statement, see Section \ref{s:recipe} and in particular Theorem \ref{l:abstract}. 

Our second contribution is to specialize the recipe to a natural family of distributions $p(X)$ for which $f(X) = \|\mathcal{A}(X X^T) - b\|_2^2$ and actually prove the conditions above rather than assuming them.  
Here $X$ is a $d \times k$ matrix, $\mathcal{A}$ is a linear measurements operator, s.t. 
\begin{equation} \forall i \in [\numm], M \in \Real^{d \times d}, \mathcal{A}(M)_i = \mbox{Tr}(A_i^T M), A_i \in \mathbb{R}^{d \times d} \label{eq:measurements}\end{equation}
and $b_i$ are noisy measurements of some ground-truth matrix, namely  
\begin{equation} \forall i \in [\numm], b_i = \mbox{Tr}(A_i^T M^*) + n_i \end{equation}
where $M^*  = X^* (X^*)^T \in \mathbb{R}^{d \times d}$ is of rank $k$ with $\sigma_{\max},\sigma_{\min}$ denoting the largest and smallest singular values of $X^*$ respectively, and let $\kappa = \frac{\sigma_{\max}}{\sigma_{\min}}$ denote the condition number. Furthermore, $n_i \sim N(0, \frac{1}{\beta})$---i.e. Gaussian noise with variance $\frac{1}{\beta}$. We will consider three instances of $\mathcal{A}$: 
\begin{enumerate} [leftmargin=*]
\item {\bf Noisy matrix factorization}: $\mathcal{A}$ is simply the identity operator, i.e.  
$\mathcal{A}(XX^T) = \mbox{vec}(XX^T)$. 
\item {\bf Matrix sensing} with measurements satisfying \emph{restricted isometry (RIP)}: $\mathcal{A}$ satisfies 
$\left(1 - \ripconst\right) \|M\|^2_F \leq \|\mathcal{A}(M)\|^2_2 \leq \left(1 + \ripconst\right) \|M\|^2_F$, for all $M \in \Real^{d \times d}$ of rank at most $2k$. 
\item {\bf Matrix completion}: $\mathcal{A}$ is a projection to a set of randomly chosen entries $\Omega \subseteq [d] \times [d]$, namely $\mathcal{A} = P_{\Omega}$, where $P_{\Omega}(Z)_{i,j} = P_{i,j} Z_{i,j}$, with $P_{i,j} = 1$ if $(i,j) \in \Omega$ and $0$ otherwise. Furthermore, the probability of sampling an entry is $\displaystyle p = \Omega\left(\max(\mu^6 \kappa^{16}k^4, \mu^4 \kappa^{4}k^6)\frac{\log^2 d}{d}\right)$, where $\mu$ is an upper bound on the incoherence of $M^*$, that is the singular value decomposition $M^* = U \Sigma V^T$ satisfies $\max_{i \in [d]} \|e_i^T U\| \leq \sqrt{\mu \frac{k}{d}}$. 
 
\end{enumerate} 

This problem has a very natural statistical motivation: $p(X) \propto e^{-\beta f(X)}$ is exactly the posterior distribution over $X$, with an appropriate prior (proof included in Section~\ref{s:posterior}): 
\begin{prop}[Posterior under appropriate prior]\label{prop:post} Let $\mathbb{B}_{\alpha} = \{X: \|X\|_F \leq \alpha\}$ and 
let $X$ be sampled uniformly from $\mathbb{B}_{\alpha}$. Let $b = \mathcal{A}(XX^T) + n, n \sim N(0, I_{\numm})$ and $f(X) = \|\mathcal{A}(X X^T) - b\|_2^2$, where 
$\mathcal{A}$ is as specified by one of the three settings above. 
Then, denoting $\tilde{p}: \Real^{d \times d} \to \Real$, s.t. $\tilde{p}(X) \propto e^{-\beta f(X)}$, and 
$p(\cdot | b)$ the posterior distribution of $X$ given $b$, we have $\lim_{\alpha \to \infty} \mbox{TV}(p(\cdot|b) || \tilde{p}) = 0 $. 
\label{p:posterior} 
\end{prop}

We note that in each corresponding context, the structure we are imposing on the operator $\mathcal{A}$ are the standard ones in the literature on non-convex optimization -- so in light of Proposition~\ref{prop:post} our results can be viewed as sampling analogues of classic results in non-convex optimization. We furthermore note that we chose the Gaussian noise setting in order to have a sampling problem from a natural posterior. However, from an algorithmic point of view, even the setting where $b = \mathcal{A}(X^* (X^*)^T)$, and we wish to sample from the corresponding $p$ is equally hard/interesting, as the distribution is not log-concave, and satisfies the same manifold structure.  


We will prove that Langevin dynamics mixes in polynomial time when $\beta$ is at least a fixed polynomial in $d$, $k$ and the condition number of $M$. Our analysis is geometric in nature, involving various differential geometric estimates of the curvatures of the level sets of the distribution, as well as the distribution of volume along these sets. These estimates are combined to prove a Poincar\'e inequality. 

Towards stating the results, again at least informally for now,  
the set of global minimizers for the function $f$ in each of the above settings will in fact take the form $$\lopt_1 = \{X_0 R, R \in O(k), \mbox{det}(R) = 1\} \mbox{ and } \lopt_2 = \{X_0 R, R \in O(k), \mbox{det}(R) = -1\}$$ where $X_0$ is any fixed minimum of $f(X)$ and $O(k)$ is the group of orthogonal matrices of dimension $k$. In general, it will take exponentially long for Langevin diffusion to transition from one manifold to the other. However we show that it successfully discovers one of them and samples from $p(X)$ restricted to a neighborhood around it. 

\begin{thm}[Informal] Let $\mathcal{A}$ correspond to matrix factorization, sensing or completion under the assumptions in Section~\ref{l:setupresults} and $\beta=\Omega(\mbox{poly}(d))$. If initialized close to one of $\lopt_i, i \in \{1,2\}$, after a polynomial number of steps the discretized Langevin dynamics will converge to a distribution that is close in total variation distance to $p(X)$ when restricted to a neighborhood of $\lopt_i$.
\label{t:inform2}
\end{thm}
For a formal statement of the theorem, see Theorem~\ref{t:maincompletion}. 

By way of remarks: In the most interesting setting, when sampling from the posterior is statistically meaningful and not just pure noise, some dependence on $d$ is also necessary: previous work by \cite{perry2018optimality} (and a precursor by \cite{peche2006largest}) show that for natural families of priors over $X$ (a particularly natural one is where $X = vv^T$ where $v$ is a random vector with $\pm 1/\sqrt{d}$ entries), when $\beta < \frac{1}{d}$, no statistical test can distinguish the ``planted'' distribution from Gaussian nose with probability $1-o(1)$.   






An important feature of our algorithms is that they are {\bf not} given an explicit description of the manifold around which they want to sample. Rather, the manifold is implicitly defined through $f(X)$ and our algorithms only use query access to its value and gradients. Nevertheless Langevin dynamics is able to discover this manifold on its own regardless of how it is embedded. 
\vspace{-0.2cm}
\section{Prior work} 
{\bf Differential Geometry:} Our work can be thought of as building on classic works that expose the connection between Ricci curvature and Poincar\'e inequalities for Brownian motion on manifolds \citep{kendall1986nonnegative}. In particular, Kendall showed that two Brownian motions on a compact manifold with nonnegative Ricci curvature couple in finite time. Later works established quantitative mixing time bounds using Bakry-Emery theory including showing that the Poincar\'e constant of a strongly log-concave measure and the Ricci curvature combine in an additive way. From a technical standpoint, our work can be thought of as a robust version of these results. When $\beta$ is large but finite then $p(X)$ is concentrated near a manifold of nonnegative Ricci curvature. Our analysis involves getting a handle on the Ricci curvature of level sets of the distance function from $\mathbf{M}$, as well as their interaction with $f(X)$, rather than just its global minimizers, which helps us show that the Langevin dynamics mixes quickly along and across level sets. 

{\bf Langevin diffusion:} For basic Langevin diffusion (in $\Real^N$), understanding the mixing time of the continuous dynamics for log-concave distributions is a standard result: \cite{bakry1985diffusions, bakry2008simple} show that log-concave distributions satisfy a Poincar\'e inequality, which characterizes the rate of convergence in $\chi^2$. Since algorithmically, we can only run a discretized version of the chain, recent line of work has focused on analyzing the necessary precision of discretization in the log-concave case: \cite{dalalyan2016theoretical, durmus2016high,dalalyan2017further} obtained an algorithm for sampling from a log-concave distribution over $\mathbb{R}^d$. \cite{bubeck2015sampling} gave a algorithm to sample from a log-concave distribution restricted to a convex set by incorporating a projection step. Finally, \cite{raginsky2017non} give a nonasymptotic analysis of Langevin dynamics for arbitrary non-log-concave distributions with certain regularity conditions and decay properties. Of course, the mixing time is exponential in general when the spectral gap of the chain is small. Related results are obtained by \cite{ma2019sampling, cheng2019quantitative}, albeit under slightly different conditions (essentially, the distributions they consider are log-concave outside a ball of radius $R$, but the mixing time exponentially depends on $R$).   

{\bf Beyond log-concavity:} In recent work, \cite{risteski2018} explored some preliminary {\em beyond log-concave} settings. Namely, they considered the case when the distribution $p(X)$ is a mixture of shifts of a log-concave distribution with unknown means. In this case, they were able to show that Langevin diffusion when combined with {\em simulated tempering} can sample from a distribution close to $p(X)$ in time polynomial in the ambient dimension and number of components. (We emphasize that without something like simulated tempering for exploring multiple deep modes, this is hopeless, as standard results in metastability \citep{ventsel1970small} show that the escape time from one of the peaks is exponential.)
We note that bounding the Poincar\'e constant by a decomposition was also employed in \cite{risteski2018}, albeit with much fewer measure theoretic complications. 

%% file: overview_rescaled.tex
\section{Formal results and technical overview}
\label{s:overview}

Our general strategy will involve decomposing the distribution along level sets of the function and leveraging various tools from differential geometry to get a handle on their curvature, their volume and various restricted Poincar\'e inequalities. From these estimates, we will be able to deduce an overall Poincar\'e inequality. The basis of our decomposition is a measure-theoretic version of the law of total probability, derived from the \emph{co-area} formula (Theorem~\ref{thm:coarea}) which we will introduce later after giving the necessary background. 

In this section, we will formally state our main results. This involves making precise the assumptions that we previously introduced informally, such as in what sense we need the Langevin dynamics to remain close to to the manifold, and how the decomposition into level sets works at a technical level. While each of these assumptions are natural, and formulating a recipe based on them that gives mixing time bounds is an important contribution of our work, we emphasize that in the particular case of matrix factorization, matrix sensing and matrix completion we are able to rigorously complete the steps in this meta-plan so that we get unconditional bounds. 

\subsection{The general decomposition recipe}  
\label{s:recipe}

First we lay out formally the conditions for the general setup: Suppose $\mathbf{M}$ is a manifold 
consisting of local minima of a twice-differentiable function $f: \Real^N \to \Real$ and is a level set of $f$. In particular for all $X \in \mathbf{M}$ we have that
$$ \nabla f(X) = 0, \nabla^2 f(X) \succeq 0, \mbox{ and } f(X) = s_0$$
Our first assumption is that $X_t$ stays close to the manifold, which is natural when $\mathbf{M}$ corresponds to a deep mode of the distribution. 
\begin{conds}[Nearness to the manifold]
For a parameter $T$ and function $s(\beta)$, Langevin dynamics $X_t$ stays in $\mathcal{D} = \{X: \min_{X' \in \mathbf{M}} \|X - X'\|_2 \leq s(\beta)\}, \forall 0 \leq t \leq T$ with probability at least $1-\epsilon$. Furthermore, let the projection $\Pi_{\mathbf{M}}(X) := \mbox{argmin}_{X' \in \mathbf{M}} \|X-X'\|_2$ be uniquely defined, $\forall X \in \mathcal{D}$.  
\label{c:nearness}
\end{conds}

\begin{remark*} To understand why this condition is natural, consider the $\beta \to \infty$ limit of the walk: the ODE $\frac{dX_t}{dt} = -\nabla f(X_t)$ will converge to a local minimum \citep{lee2016gradient} almost surely when the initial point is chosen randomly. If such points form a manifold, at large but finite $\beta$, the walk ought to take a long time to escape.
\end{remark*} 

Next we will formally state the decomposition of $p$ that we will be relying on. Let $\tilde{p}$ denote the restriction of $p$ to the region $\mathcal{D}$, renormalized so that it is also a distribution. Let us choose an arbitrary point $X_0 \in \mathbf{M}$, and denote the ``norm-bounded'' normal space\footnote{For formal definition, see Definition \ref{d:submanifold}}  
\begin{equation} \mathbf{B} = \{\Delta: \Delta \in N_{X_0}(\mathbf{M}), \|\Delta\|_2 \leq s(\beta)\} \label{eq:normalball}\end{equation}
Furthermore, $\forall X \in \mathbf{M}$, let us assume the existence of a diffeomorphism (i.e. differentiable bijection)   
\begin{equation} \phi_X: \mathbf{B} \to \{\Delta: \Delta \in N_{X}(\mathbf{M}), \|\Delta\|_2 \leq s(\beta)\} \label{eq:proxyexp}\end{equation}
One should think of this function as a way to map the normal space at any point in $\mathbf{M}$ to the normal space at $X_0$\footnote{One way this could be done is the exponential map, if globally defined, but we will never require this.}  
Given this, let us define a manifold for every $\Delta \in \mathbf{B} $: 
$$\forall \Delta \in \mathbf{B}: \mathbf{M}^{\Delta}:= \{X + \phi_X(\Delta): X \in \mathbf{M}\}$$
This can be viewed as a ``part'' of the level-set of the distance function specified by $\Delta$: the disjoint union of the manifolds $\mathbf{M}^{\Delta}$, s.t. $\|\Delta\|_2 = s$ gives the set of all points at distance at most $s$ from $\mathbf{M}$. 

Now we define a family of distributions $\tilde{p}^{\Delta}$ that come from restricting $\tilde{p}$ to $\mathbf{M}^{\Delta}$. Towards this end, let us denote by $F: \mathcal{D} \to N_{X_0}(\mathbf{M})$ the function s.t. 
$F(Y) = \Delta$, where $\Delta \in N_{X_0}(\mathbf{M})$ is the unique vector s.t. $Y = X + \phi_X(\Delta), X \in \mathbf{M}$ (the uniqueness follows from Condition \ref{c:nearness}). Let $\bar{dF}$ be the restriction of the differential map $dF$ to subspace $\mbox{ker}(dF)^{\perp}$--- that is, the orthogonal subspace of the kernel of $dF$, and let $\mbox{det}(\bar{dF})$ be the determinant of this map\footnote{For the reader unfamiliar with differentials, refer to Definition \ref{d:jac}}.  
We then denote
\begin{equation}  \tilde{p}^{\Delta}(X) \propto \frac{\tilde{p}(X)}{\mbox{det}(\bar{dF}(X))} \label{eq:tildepdefg} \end{equation}
And finally let $q$ be the distribution that captures how $\tilde{p}$ is spread out across the manifolds $\mathbf{M}^{\Delta}$. In particular let $q: \mathbf{B} \to \Real$ be 
$$q(\Delta) \propto \int_{X \in \mathbf{M}^{\Delta}} \frac{\tilde{p}(X)}{\mbox{det}(\bar{dF}(X))} d\mathbf{M}^{\Delta}(X)$$ 
where $d\mathbf{M}^{\Delta}(X)$ denotes the volume element of the manifold $\mathbf{M}^{\Delta}$. (See Definition \ref{d:localvol}.) This is a decomposition of $\tilde{p}$ in the following sense:
\begin{lem}[Decomposing distribution] 
Let $\chi: \mathcal{D} \to \Real$ be any measurable function. 
Then 
$$\E_{X \sim \tilde{p}} \chi(X) = \E_{\Delta \sim q} \E_{X \sim \tilde{p}^{\Delta}} \chi(X)$$ 
\label{l:general} 
\end{lem}
This follows from the coarea formula and is a key ingredient in our proof.  
With this decomposition in hand, we will need bounds on various restricted Poincar\'e constants. In particular, we assume:
\begin{conds}[Poincar\'e constant along level sets]
$\forall \Delta \in \mathbf{B}$: the distribution $\tilde{p}^{\Delta}$ has a Poincar\'e constant bounded by $C_{\mbox{level}}$.
\label{c:along}
\end{conds}

\begin{remark*} In our settings of interest, $\mathbf{M}$ will be a matrix manifold that has nonnegative Ricci curvature. It is well-known that a lower bound on the Ricci curvature translates to an upper bound on the Poincar\'e constant (Lemma~\ref{d:bcp}). However when $\beta$ is large but finite the Langevin dynamics will merely be near $\mathbf{M}$ and so $\mathbf{M}^{\Delta}$ could be expected to be ``similar'' to $\mathbf{M}$. Note, however, this is very subtle as curvature is a local quantity---we wish to take the functions $\phi_X$ such that $\mathbf{M}^{\Delta}$ behave like ``translates'' of $\mathbf{M}$ in the sense of non-negativity of the Ricci curvature---which is quite fragile. 


\end{remark*} 

Furthermore, we will assume:
\begin{conds}[Poincar\'e constant across level sets]
$q$ has a Poincar\'e constant that is at most $C_{\mbox{across}}$. 
\label{c:across}
\end{conds}
\begin{remark*}
\vspace{-0.1cm} 
To understand why this condition is natural, note that $q$ is supported over $\mathbf{B}$, which is in fact a ball, hence a convex set. If the function $f$ were exactly the distance from $\mathbf{M}$, $q$ would have the form $q(\Delta) \propto e^{-\beta \|\Delta\|^2_F}$---which in fact log-concave. Since log-concave functions supported over convex sets have good Poincar\'e constants (Lemma \ref{l:constrained}), the assumption above would follow. 
In the matrix setup we consider, we will show that something like this approximately happens---namely, we will show that $q$ will approximately have the form $q(\Delta) \propto e^{-\Delta^T \Sigma \Delta}$ for a PSD matrix $\Sigma$. 


\end{remark*} 

\begin{conds}[Bounded change of manifold probability]
Let us define by $G_{\Delta}: \mathbf{M} \to \mathbf{M}^{\Delta}$ the map $G_{\Delta}(X) = X + \phi_X(\Delta)$. Then, 
$$\forall \Delta \in \mathbf{B}, X \in \mathbf{M}: \left\|\frac{\nabla_{\mathbf{B}} \left(p^{\Delta}(X + \phi_X(\Delta)) \mbox{det}\left((dG_{\Delta})_X\right)\right)}{p^{\Delta}(X + \phi_X(\Delta)) \mbox{det}\left((dG_{\Delta})_X\right)}\right\|_2 \leq C_{\mbox{change}}$$

\label{c:deltadensity}
\end{conds}

\begin{remark*} 
It is intuitively easy to understand the quantity above: the denominator is the ``measure'' 
on the manifold $\mathbf{M}^{\Delta}$ implied by $\tilde{p}^{\Delta}$ and the volume form of $\mathbf{M}^{\Delta}$, and the numerator is the ``change'' in this measure -- what we require is that the relative magnitude of this change is small.  
\end{remark*}

With the above setup in place, the first theorem we will prove is the following: 
\begin{thm}[Main, generic framework]\label{thm:framework}
Let $p_T$ be the solution (i.e. a distribution) to the stochastic differential equation
$d X_t = -\beta \nabla f(X_t) dt + \sqrt{2} dB_t $
at time $T$ when initialized according to $p_0$ which is absolutely continuous with respect to the Lebesgue measure. If Conditions~\ref{c:nearness},~\ref{c:along} and~\ref{c:across} hold, we have that
$$d_{\mbox{TV}}(p_t,\tilde{p}) \leq \epsilon + \sqrt{\chi^2(p_0, \tilde{p})} e^{-\frac{t}{2\mainpc}}$$
for all $t \leq T$ where 
$\mainpc = O\left(\max\left(1, C_{\mbox{level}}\right) \max\left(1,C_{\mbox{across}}\right)\max\left(1, C^2_{\mbox{change}}\right)\right)$  
\label{l:abstract}
\end{thm}

The main idea is to show that $\tilde{p}$ satisfies a Poincar\'e inequality. In particular we want to show that
$\mbox{Var}_{\tilde{p}}(g) \leq \mainpc \E_{\tilde{p}} \|\nabla g \|^2$ for appropriately restricted functionals $g: \mathbb{R}^N \rightarrow \mathbb{R} $. Now by applying Lemma~\ref{l:general} 
and invoking the law of total variance, we have
$ \mbox{Var}_{\tilde{p}}(g) = \E_{\Delta \sim q} \mbox{Var}_{X \sim \tilde{p}^\Delta}(g) + \mbox{Var}_{\Delta \sim q} (\E_{X \sim  p^\Delta} g) $. The Poincar\'e inequality 
will follow by using Condition~\ref{c:along} and Condition~\ref{c:across} to bound each term separately, namely
$\E_{\Delta \sim q} \mbox{Var}_{X \sim \tilde{p}^\Delta}(g)  \leq  C_{\mbox{level}}\E_{\tilde{p}} \|\nabla g\|^2 \label{eq:part1}$ and $\mbox{Var}_{\Delta \sim q} (\E_{X \sim  \tilde{p}^\Delta} g)  \leq 2 C_{\mbox{across}} \left(C_{\mbox{level}} + C_{\mbox{level}} C^2_{\mbox{change}} \right) \E_{\tilde{p}} \|\nabla g\|^2$.

One can intuitively think of $C_{\mbox{level}}$ and $C_{\mbox{across}}$ as capturing the expansion/conductance properties of the level sets, and the conditional distribution over the level sets. (The latter has a somewhat technical correction factor, which appears due to an application of the chain rule.) We need Condition~\ref{c:nearness} to ensure that Langevin dynamics stays in $\mathcal{D}$ long enough to mix -- see Section \ref{l:gensetup} for details.  

\subsection{Implementing the recipe for matrix factorization objectives}  
\label{s:matrixoverview}

While the general recipe we gave was simple and intuitive, proving that Conditions~\ref{c:nearness},~\ref{c:along} and~\ref{c:across} hold can be rather technically challenging. (To help the reader get some intuition, we provide a simpler toy example in Section~\ref{s:warmuptorus} of a function which has tori as level sets.)



Let us state the results formally first. 
Let $p^1(X)$ be the proportional to $p(X)$ if $\|X - \Pi_{\lopt_1}(X)\|_F < \|X - \Pi_{\lopt_{2}}(X)\|_F$ and zero otherwise. Define $p^2(X)$ analogously with $\lopt_1$ and $\lopt_2$ interchanged. 




We then have our second main result: 

\begin{thm}[Main, matrix objectives] Let $\mathcal{A}$ correspond to matrix factorization, sensing or completion, with the restrictions on the RIP constant, incoherence and observations as in Section \ref{l:setupresults}, and let $f$ be the corresponding loss. Finally, for any $\epsilon > 0$, let 
$$\beta \gtrsim \left\{\lbdmf, \lbdms, \lbdmc\right\}$$
for matrix factorization, sensing and completion respectively. 
Then, for $\mainpc = O\left(\bdpc\right)$, the following holds: 
\begin{enumerate}[leftmargin=*] 
\item[(1)] {\bf Continuous process:}  Let $p_T$ be the solution (i.e. a distribution) of the Langevin diffusion chain $dX_t = -\beta \nabla f(X_t) dt + \sqrt{2} dB_t$ at time $T$, 
where $dB_t$ is the standard $dk$-dimensional Brownian motion, with $p_0(X)$ absolutely continuous with respect to the Lebesgue measure and supported on points $X_0$, s.t. for some $i \in \{1,2\}$, 
\begin{equation} \|X_0 - \Pi_{\lopt_i}(X)\|_F \leq 40 \left\{\frac{k \frac{\kappa}{\sigma_{\min}} \sqrt{d \log d \log (1/\miss)}}{\sqrt{\beta}},  \frac{\sqrt{d k \log \numm \log (1/\miss)} \frac{\kappa}{\sigma_{\min}}}{\sqrt{\beta}}, 
\frac{\sqrt{d k^3 \log d \log (1/\miss)}\kappa^3/\sigma_{\min}}{p \sqrt{\beta}}\right\} \label{eq:bdsonnbrs}\end{equation} 
for factorization, sensing and completion respectively. Then, for any $t > 0$, 
$$d_{\mbox{TV}}(p_t(X),p^i(X)) \leq \miss + \sqrt{\chi^2(p_0(X), p^i(X))} e^{-\frac{t}{2\mainpc}} $$ 
\item[(2)] {\bf Discretized process:} A point $X_0$ satisfying \eqref{eq:bdsonnbrs} can be found in polynomial time. \footnote{In fact, by performing gradient descent on the corresponding $f$ from a random starting point, with an appropriate regularizer for the matrix completion case.}
Furthermore, for a step size $h > 0$, let $\hat{t}:=t/h$, let the sequence of random variables $\hat{X}_i, i \in [0, \hat{t}]$ be defined as $\displaystyle \hat{X}_{i+1} = -\beta \nabla f(\hat{X}_i) h + \sqrt{2h} \xi, \xi \sim N(0,I), \hat{X}_0 = X_0$. Then, 
$$d_{TV}(\hat{p}_{\hat{t}},p^i) \leq \sqrt{\beta \mbox{poly}(d,\sigma_{\max}) th} + \miss + \sqrt{\chi^2(p_0, p^i)} e^{-\frac{t}{2\mainpc}} $$ 
Hence, if $h = O\left(\frac{\miss^2}{t \beta  \mbox{poly}(d,\sigma_{\max})}\right)$ we have 
$d_{TV}(\hat{p}_{\hat{t}},\tilde{p}) \lesssim \miss + \sqrt{\chi^2(p_0, p^i)} e^{-\frac{t}{2\mainpc}}$.  
\end{enumerate}
\label{t:maincompletion}  
\end{thm} 

The main task is to verify Conditions~\ref{c:nearness},~\ref{c:along},~\ref{c:across} and \ref{c:deltadensity}
in the setup of Theorem~\ref{thm:framework}. Next we describe the main technical ingredients in establishing each of these conditions. 

{\bf Establishing Condition \ref{c:nearness}:} This step turns out to be non-trivial despite how intuitive the statement is. At least one reason for this is that standard tools giving large deviation bounds for SDEs, such as Freidlin-Wentzell theory (\cite{ventsel1970small}) do not apply in a black-box manner: typically, one assumes in these settings that the minima of the function are isolated. 
Instead, we will derive an SDE that tracks the distance to the manifold. We will then use the Cox-Ingersoll-Ross process \citep{cox2005theory} formalism and its characterization as the square of an Ornstein-Uhlenbeck process along with comparison theorems for SDEs to obtain concentration bounds. This is in fact the only part where the usual intuition of local convexity from the optimization variant of these problems carries over -- the reason the random process stays close to the manifold is that the gradient term has significant correlation with the direction of the projection to the manifold. See Section \ref{s:concentration}. 

{\bf Establishing Condition \ref{c:along}:} The strategy is to decompose the space near $\lopt_i$ according to vectors $\Delta \in N_{X_0}(\lopt_i)$ --- the main part of which is designing the map $\phi_X$ (see \eqref{eq:phimx}). Under our choice of $\phi_X$, the manifolds $M^{\Delta}$ will have the form $M^{\Delta} = \{YU: U \in \mbox{SO}(k)\}$ for some matrix $Y$. 

We will show they have non-negative Ricci curvature which will allow us to derive a Poincar\'e inequality. The primary tool for this is a classic estimate due to \cite{milnor1976curvatures} which gives an exact formula for the Ricci and sectional curvatures of Lie groups equipped with a left-invariant metric. It turns out we cannot directly apply this formula because the metric we need comes from the ambient space and is not left invariant---however we can relate the Poincar\'e inequalities under these two metrics. To handle the weighting by $\tilde{p}^{\Delta}(X)$ and $\mbox{det}(\bar{dF})$, we will show that in fact they are both constant over $\mathbf{M}^{\Delta}$. See Section~\ref{s:levelsetp}.


{\bf Establishing Condition \ref{c:across}:} Following the intuition we gave when we introduced Condition \ref{c:across}, our proof will argue that $q$ is approximately log-concave with support over a convex set. 
The strategy will be to Taylor expand $f$, and prove that it is up to low-order terms log-concave, when the support of $q$ is appropriately parametrized. See Section~\ref{s:poincarr}.

{\bf Establishing Condition \ref{c:deltadensity}:} Given that (as part of proving Condition \ref{c:along}) 
we show that $p^{\Delta}$ is uniform over $\mathbf{M}^{\Delta}$ and $M^{\Delta}$ is the image of $\mbox{SO}(k)$ under a linear map, we can explicitly calculate $p^{\Delta}(X + \phi_X(\Delta)) \mbox{det}\left((dG_{\Delta})_X\right)$---and we in fact show it's independent of $\Delta$. 
See Section~\ref{s:gradienttoval}.

 



{\bf Remarks on statements and proofs:} The proof of Conditions \ref{c:along} and \ref{c:deltadensity} 
in fact does not depend on the operator $\mathcal{A}$ at all---we will mostly repeatedly use the orthogonal invariance of the objective, which attains for any $\mathcal{A}$. Condition \ref{c:nearness} is mostly where the specific operator properties come in play: namely, we will use the well-known property that the gradient of the matrix completion and sensing objectives is correlated with the projection towards the manifold of optima. This will ensure that in both of these cases, if we start close to one of the manifolds of optima, we will remain close to it. 

We also note that the initialization condition can be attained for matrix factorization and sensing by just running variants of gradient descent that avoid saddle points \citep{ge2017no}, or just gradient descent with appropriate initialization. In the case of matrix completion, some regularization has to be added to ensure the algorithm stays in the region of incoherent matrices. It's entirely plausible in the former two cases (factorization and sensing), that Langevin dynamics converges to a point $X_0$ satisfying the initialization conditions (as the saddle-point avoidance algorithms are essentially gradient descent with noise). We leave this for future work.

%% file: notation_extrinsic_rescaled.tex
\section{Crash course in differential geometry and diffusion processes} 

In this section, we introduce several key definitions and tools from differential geometry and diffusion processes. Most of these are standard, and can be found in classical references on differential geometry (e.g. \cite{do2016differential}) -- for the less standard ones, we will provide separate references. 

\subsection{Basic differential geometric notions}
 
First we will define basic notions in differential geometry like a submanifold, a tangent space, a normal space, etc. Whenever possible, we will specialize the definitions to only what we will need. For example, we will only need the notion of a submanifold embedded in $\Real^d$ because that is the space in which we will be working. 

\begin{defn} [Submanifold] A manifold $\mathbf{M}$ is a smooth (differentiable) $m$-dimensional \emph{submanifold} of $\Real^d$, if $\mathbf{M} \subseteq \Real^d$ and $\forall x \in \mathbf{M}$, there exists a local \emph{chart}: a pair $(U, F_x)$, s.t. $U \subseteq \mathbf{M}, x \in U$ and $F_x: U \to V$ is a diffeomorphism for some open $V \subseteq \Real^m$.  
A submanifold is called a \emph{hypersurface} if it is of dimension $d-1$ (i.e. of co-dimension 1). An \emph{atlas} of $\mathbf{M}$ is a collection $(U_{\alpha}, F_{\alpha} | \alpha \in A)$ indexed by a set $A$, s.t. $\cup_{\alpha \in A} U_{\alpha} = \mathbf{M}$

The \emph{tangent space} of a submanifold $\mathbf{M}$ at a point $x \in \mathbf{M}$, denoted $T_x(\mathbf{M})$, is the 
vector space of tangent vectors to curves through $x$ in $\mathbf{M}$. In other words, 
$$T_x(\mathbf{M}) = \{\phi'(0): \phi: (-1,1) \to \mathbf{M}, \phi(0) = x\}$$
When clear from context, we will drop the manifold explicitly, and just refer to $T_x$. 
The \emph{normal space} of a submanifold $\mathbf{M}$ at a point $x \in \mathbf{M}$, denoted by $N_x(\mathbf{M})$, is the 
orthogonal space to $T_x(\mathbf{M})$.  

We say the manifold is \emph{equipped} (or \emph{endowed}) with a metric $\gamma$, if 
$$\gamma_x: T_x(\mathbf{M}) \times T_x(\mathbf{M}) \to \Real, \hspace{0.5cm} x \in \mathbf{M}$$
is a \emph{smoothly varying} inner product: namely for any pair of $C^{\infty}$ vector fields $V,W$ on $\mathbf{M}$, 
$x \to \langle V(x), W(x) \rangle_{\gamma_x}$
is a $C^{\infty}$ function. 

For the majority of this paper, we will work with the standard Euclidean metric. (Most of the calculations involving alternate metrics will be in Section \ref{s:levelsetp}, where we will extensively work with Lie groups and left-invariant metrics.) 

To reduce clutter in the notation, when the metric $\gamma$ is not specified and clear from context, we will assume it is the standard Euclidean metric.  
\label{d:submanifold}
\end{defn}

As is conventional, it will be convenient to collect either the tangent or normal space along with the manifold into what is called a bundle: 
\begin{defn}[Tangent bundle] The tangent bundle $T \mathbf{M}$ of a manifold $\mathbf{M}$ is the set $T \mathbf{M} := \{(x,v): x \in \mathbf{M}, v \in T_x(\mathbf{M})\}$.  
\label{d:tangentbundle}
\end{defn}
\begin{defn} [Normal bundle] The normal bundle $N \mathbf{M}$ of a manifold $\mathbf{M}$ is the set $N \mathbf{M} := \{(x,v): x \in \mathbf{M}, v \in N_x(\mathbf{M})\}$. 
\label{d:normalbundle}
\end{defn} 

We will often need to work with projections, particularly onto a manifold of global optima to reason about how the diffusion is mixing both on and off of the manifold. 
\begin{defn} [Projection] Given a point $x \in \Real^n$, the projection of $x$ to a submanifold $\mathbf{M}$, denoted $\Pi_{\mathbf{M}}(x)$, is defined as 
$$\Pi_{\mathbf{M}}(x) = \mbox{argmin}_{x' \in \mathbf{M}} \|x-x'\|_2 $$ 
When the minimizer is non unique, we choose among them arbitrarily. 
\label{d:projection}
\end{defn} 

\begin{defn} [Differential (pushforward) of function] Let $F: \mathbf{M} \to \mathbf{N}$ be a differential function between two smooth submanifolds. The \emph{differential} of $F$ at $x \in \mathbf{M}$ is the function $dF_x: T_x(\mathbf{M}) \to T_{\phi(x)}(\mathbf{N})$, s.t. if $\phi:(-1,1) \to \mathbf{M}$ is a curve with $\phi(0) = x$ and $\phi'(0) = v$, then  
$$d F_x(v) = (F \circ \phi)'(0)$$    
\label{d:jac}
\end{defn} 

As a special case, we will characterize the derivative of a function on a manifold: 
\begin{prop} [Derivative of function on manifold] Let $\mathbf{M} \subseteq \Real^d$ be a smooth submanifold, endowed with the standard Euclidean metric. Let $f: \mathbf{M} \to \Real$ be a differentiable function. Then, the derivative of $f$ is  
$$\nabla_{\mathbf{M}} f(x) = \Pi_{T_x(\mathbf{M})} \nabla f(x)$$
where we use the notation to distinguish with the usual gradient. 
\label{d:deronman} 
\end{prop} 
 
We will also need the notion of \emph{normal determinant}, which is a slight generalization of the usual determinant: 
\begin{defn}[Normal determinant] 
Let $\mathbf{M}$ and $\mathbf{N}$ be submanifolds and let $F: \mathbf{M} \to \mathbf{N}$ be a differentiable map, s.t. $\forall x \in \mathbf{M}$, the differential $dF_x: T \mathbf{M} \to T \mathbf{N}$ is surjective.

Then, the restriction of $dF_x$ to the orthogonal complement of its kernel is a linear isomorphism. The absolute value of the determinant of this map, which we denote as $|\mbox{det}(\bar{dF_x})|$, is called the \emph{normal  determinant}.    
\end{defn} 

Finally, we will need a few concepts relating to volume of submanifolds. First, we recall the notion of a differential form somewhat abstractly (we will quickly make it substantially more concrete): 

\begin{defn}[Differential form on a manifold] A differential $k$-form $\omega$ on a manifold $\mathbf{M}$ is an alternating multilinear function on the tangent bundle of $\mathbf{M}$: namely $\forall x \in \mathbf{M}$, we have an alternating multilinear function $ \omega(x): T^{\otimes k}_x(\mathbf{M}) \to \Real$. 
(Recall, a function $f: V^{\otimes k} \to \Real$ is alternating multilinear if $f(v_1, v_2, \dots, v_k) = (-1)^{\mbox{sign}(\sigma)} f(v_{\sigma(1)}, v_{\sigma(2)}, \dots, v_{\sigma(k)})$for any permutation $\sigma$.)   

The explicit notation for differential forms is in terms of \emph{wedge products}: if $\mathbf{M} \subseteq \Real^d$ is locally parametrized by a chart  $(U, (x_1, x_2, \dots, x_m))$, s.t. $U \subseteq \mathbf{M}$ and $(x_1, x_2, \dots, x_m): U \to V$ is a diffeomorphism for some open $V \subseteq \Real^m$, a $k$-form $\omega$ can be written as $\omega := \sum_{I \subseteq [d]: |I| = k} f_I \wedge_{i \in I} d x_i$ for scalars $f_I$, where $d x_i$ is the differential of the function $x_i$, and the wedge product of functions $f:  V^{\otimes k} \to \Real, g:  V^{\otimes l} \to \Real$ is defined as 
\begin{align*} &(f \wedge g): V^{\otimes (k+l)} \to \Real, f(v_1, v_2, \dots, v_{k+l}) := \\
&\frac{1}{k! l!} \sum_{\sigma \in S_{k+l}} (-1)^{\mbox{sign}(\sigma)} f(v_{\sigma(1)}, v_{\sigma(2)}, \dots, v_{\sigma(k)}) g(v_{\sigma(1)}, v_{\sigma(2)}, \dots, v_{\sigma(l)}) \end{align*} 
where $S_{k+l}$ is the set of permutations on $k+l$ elements.
\end{defn}

We will also introduce the \emph{volume form}:
\begin{defn}[Volume form on a manifold] A $k$-dimensional submanifold $\mathbf{M}$ is orientable, if it admits an atlas $(U_{\alpha}, F_{\alpha} | \alpha \in A)$, s.t. the determinants $\mbox{det}(dF_{\alpha}), \forall \alpha \in A$ are everywhere positive.

An orientable $k$-dimensional submanifold $\mathbf{M}$ equipped with a metric $\gamma$ defines a differential $k$-form, called the \emph{volume form} of $\mathbf{M}$ and denoted as $d \mathbf{M}$. If $\mathbf{M}$ is locally parametrized by a chart 

 $(x_1, x_2, \dots, x_m): x_i: U \subseteq \mathbf{M} \to \Real$, the volume form locally can be written as $\omega := \sqrt{|\mbox{det}(g)|}\wedge_{i=1}^m  d x_i$ where $g$ is the matrix representation of $\gamma$ in the basis $x$, namely the matrix $g \in \mathbb{R}^{m \times m}: g_{i,j} = \langle \frac{\partial}{\partial x_i}, \frac{\partial}{\partial x_j}\rangle_{\gamma}$. 

\label{d:volform}
\end{defn} 


As a straightforward consequence of the above definition, we can define the volume of a manifold: 

\begin{defn}[Volume of parametrized manifold] Let $\mathbf{M}$ be a submanifold of $\Real^d$ equipped with a metric $\gamma$ and let $\phi: U \subseteq \Real^{m} \to \mathbf{M}$ be a diffeomorphism. Then, we will denote by $d\mathbf{M}(x)$ the volume form corresponding to $\mathbf{M}(x)$, and 
$$\vol(\mathbf{M}):= \int_{x \in \mathbf{M}} d\mathbf{M}(x) := \int_{U} \sqrt{|\mbox{det}(g(u))|} du $$ 
where $g(u) \in \mathbb{R}^{m \times m}$ is defined as $g(u)_{i,j} = \langle \frac{\partial \phi(u)}{\partial u_i}, \frac{\partial \phi(u)}{\partial u_j}\rangle_{\gamma}$

We remark that this definition is independent of the choice of parametrization, up to sign.   
\label{d:localvol}
\end{defn} 
Note, the parametrization above is \emph{global}, as the range of the $\phi$ is $\mathbf{M}$. In our definition of submanifold (Definition \ref{d:submanifold}), note that we only required the manifold to be coverable by \emph{local} maps $\phi$. We note that if there is no global parametrization of the manifold, the notion of volume can be easily extended, by using partitions of unity. 

\begin{defn}[Partition of unity] Let $S \subseteq \Real^d$ be compact. Let $(U_{\alpha} | \alpha \in A), U_{\alpha} \subseteq \Real^d$ be a collection of open sets, s.t. 
$S \subseteq \cup_{\alpha \in A} U_{\alpha}$. The collection of functions $(\rho_{\alpha} | \alpha \in A)$ is called a \emph{partition of unity subordinate to} $(U_{\alpha} | \alpha \in A)$ if: \\
(1) $\forall x \in S$, there is a neighborhood of $x$ where all but a finite number of the functions $\rho_{\alpha}$ are 0. \\
(2) $\forall x \in S$, $\sum_{\alpha \in A} \rho_{\alpha}(x) = 1$. \\
(3) $\forall \alpha \in A: \mbox{supp}(\rho_{\alpha}) \subseteq U_{\alpha}$ \\
\end{defn}  
The existence of partitions of unity is a standard result. With this in mind, we can define the volume of a manifold that doesn't have a global parametrization: 
\begin{defn}[Volume of manifold] Let $\mathbf{M}$ be a submanifold of $\Real^d$ equipped with a metric $\gamma$ and let $(U_{\alpha}, F_{\alpha} | \alpha \in A)$ be an atlas for $\mathbf{M}$. Let $(\rho_{\alpha} | \alpha \in A)$ be a subordinate partition of unity. Then, 

$$\vol(\mathbf{M}):= \int_{x \in \mathbf{M}} d\mathbf{M}(x) := \sum_{\alpha \in A} \int_{u \in F_{\alpha}(U_{\alpha})} \sqrt{|\mbox{det}(g(u))|} du $$
where  $g(u) \in \mathbb{R}^{m \times m}$ is defined as $g(u)_{i,j} = \langle \frac{\partial F^{-1}_{\alpha}(u)}{\partial u_i}, \frac{\partial F^{-1}_{\alpha}(u)}{\partial u_j}\rangle_{\gamma}$.
We remark that this definition is independent of the choice of parametrization, up to sign.   
\label{d:globalvol}
\end{defn} 

Finally, given the definition of a volume form, we can also define distributions with a density over a manifold: 
\begin{defn}[Distribution over manifold] Let $\mathbf{M}$ be a submanifold with volume form $d \mathbf{M}$. Then, a distribution over $\mathbf{M}$ with density $p$ is a function $p:\mathbf{M} \to \mathbb{R}^+$, s.t. 
$$\int_{x \in \mathbf{M}} p(x) d \mathbf{M}(x) = 1$$ 
\label{d:distrm}
\end{defn} 

We also need the following standard measure-theoretic theorem, called the co-area formula: 
\begin{thm}[Co-area formula, \cite{burgisser2013condition}]\label{thm:coarea}
Let $\mathbf{M}$ and $\mathbf{N}$ be manifolds and let $F: \mathbf{M} \to \mathbf{N}$ be a differentiable map, s.t. $\forall x \in \mathbf{M}$, the differential $d F_x: T \mathbf{M} \to T \mathbf{N}$ is surjective.

We then have:  
$$ \int_{x \in \mathbf{M}} \phi(x) d\mathbf{M}(x) = \int_{y \in \mathbf{N}} \int_{x \in F^{-1}(y)} \frac{\phi(x)}{\mbox{det}(\bar{dF}_x)} \left(dF^{-1}(y)\right)(x) d\mathbf{N}(y) $$
where $dF^{-1}(y)$ denotes the volume form on the manifold $F^{-1}(y)$ and $\bar{J}_F$ is the normal Jacobian determinant. 
\label{l:coarea}
\end{thm} 

Now with these definitions in hand, we can introduce the notion of curvature, estimates of which will play a key role in our proof. 

\subsection{Notions of curvature}

We will use multiple notions of curvature. They all give us various sorts of control on Poincar\'e inequalities and on the mixing time of a diffusion, but at some junctures some of them will be more convenient to work with than others. To help the reader who is unfamiliar with these, we offer intuition for how to interpret them geometrically.  We remark that the usual exposition proceeds in an \emph{intrinsic} manner, by defining first the notion of a connection, and then defining the Riemannian curvature tensor through the Levi-Civita connection. We will follow an extrinsic approach because it will be easier to perform explicit calculations and it comes with less technical baggage for audiences who are unfamiliar with Riemannian geometry.

First, we define the second fundamental form, which is most easily understood in the case of hypersurfaces: it captures the rate of change of the normal along the surface. 

\begin{defn}[Second fundamental form on a surface] Let $\mathbf{M} \subseteq \Real^N$ be a surface. The second fundamental form $\sff_x$ at $x \in \mathbf{M}$ is a quadratic form $\sff_x: T_x(\mathbf{M}) \times T_x(\mathbf{M}) \to N_x(\mathbf{M})$ s.t.
$$ \sff_x(v,w) = \langle v, (\nabla n) w \rangle_{\gamma_x} n$$  
where $n$ is the vector field of unit normals to $\mathbf{M}$.
\label{l:surfacesff}
\end{defn} 
We will abuse notation and treat $\sff_x$ as a map $\sff_x: T_x(\mathbf{M}) \times T_x(\mathbf{M}) \to \Real$ by intepreting it as $\sff_x(v,w) = \langle v, (\nabla n) w \rangle_{\gamma_x}$ (Note, for a surface, the unit normal is uniquely defined up to orientation.)



The second fundamental form matches the intuition that the second-order behavior of a surface (i.e. curvature) should be described by the Hessian, if the surface is given as the graph of a function. 
Namely, we have the following lemma:

\begin{lem}[\cite{do2016differential} Second fundamental form of a hypersurface] 
Let $\mathbf{M}$ be a hypersurface in $\mathbb{R}^N$ which is defined as the set $\{x: f(x) = 0\}$ for a twice differentiable $f(x)$ and endowed with the Euclidean metric. Then, if $\forall x \in \mathbf{M}, \nabla f(x) \neq 0$, we have:
\begin{enumerate}
\item[(1)] The unit normal at $x \in \mathbf{M}$ is $\frac{\nabla f(x)}{\|\nabla f(x)\|}$. 
\item[(2)] The second fundamental form at $x \in \mathbf{M}$ is given by $\sff_x = \frac{\nabla^2 f(x)}{\|\nabla f(x)\|}$. 
\end{enumerate}
\label{eq:surfacesff}
\end{lem} 

Analogous notions can be defined for co-dimension $>1$ submanifolds:

\begin{defn}[Second fundamental form on a submanifold] Let $\mathbf{M}$ be a submanifold. The second fundamental form $\sff_x$ at $x \in \mathbf{M}$ is a quadratic form $\sff_x: T_x(\mathbf{M}) \times T_x(\mathbf{M}) \to N_x(\mathbf{M})$ s.t. for a direction $n_0 \in N_x$, and a smooth vector field of normals, s.t. $n(x) = n_0$, we have
$$\langle \sff_x (v,w), n_0 \rangle_{\gamma_x} = \langle v, (\nabla n) w \rangle_{\gamma_x} $$  
\end{defn}  

Then, a similar statement to Lemma \ref{eq:surfacesff} for a co-dimension $>1$ submanifold attains: 
\begin{lem}[\cite{do2016differential} Second fundamental form of a submanifold] 
Let a submanifold $\mathbf{M}$ in $\mathbb{R}^N$ be parametrized around $x \in \mathbf{M}$ as $\phi: T_x(\mathbf{M}) \to T_x(\mathbf{M}) \times N_x(\mathbf{M})$ and endowed with the Euclidean metric, s.t. 
$$\phi(z) = x(0) + (z, f(z))$$ 
for a twice-differentiable function $f: T_x(\mathbf{M}) \to N_x(\mathbf{M})$, s.t. $f(0) = 0$. 
Then, $\sff_{x} = \nabla^2 f$, viewed as a quadratic map from $T_x(\mathbf{M}) \times T_x(\mathbf{M}) \to N_x(\mathbf{M})$. 
\footnote{ In other words, $\sff_{x}$ is the best local quadratic approximation $\mathbf{M}$.}
\label{l:sffsubmanifold}
\end{lem}


With these definitions in place, we will see a few notions of curvature we will use extensively.  

\begin{defn}[Principal curvatures] Let $\mathbf{M}$ be a hypersurface. The principal curvatures at a point $x$ are the eigenvalues of the 
quadratic form 
$$ \sff_x: T_x(\mathbf{M}) \times T_x(\mathbf{M}) \to \Real$$ 
\end{defn} 

\begin{defn}[Sectional curvature] Let $\mathbf{M}$ be a hypersurface, and let $u,v$ be linearly independent vectors in $T_{\mathbf{M}}(x)$. 
The sectional curvature in the plane\footnote{It may not be obvious from the definition that this quantity only depends on the span, but this is indeed the case.}spanned by $u,v$ is defined as
$$\kappa(u,v) = \frac{\sff_x(u,u) \sff_x(v,v) - \sff_x(u,v)^2}{\langle u, u\rangle_{\gamma_x} \langle v, v\rangle_{\gamma_x} - \langle u, v\rangle_{\gamma_x}^2} $$  
\label{l:sectdef}
\end{defn} 

For the readers more familiar with intrinsic definitions, this definition of sectional curvature can be derived from the usual one by using the Gauss-Codazzi equations.  


Finally, we move on to the Ricci curvature, which is in a sense an average of sectional curvatures, and hence is a coarser measure of curvature. 
\begin{defn}[Ricci Curvature] The Ricci curvature of a manifold $\mathbf{M}$ at a point $x \in \mathbf{M}$ in a direction $v$ is defined as 
$$\mbox{Ric}(v) = \sum_{i=1}^m \langle \sff_x(u,u), \sff_x(e_i,e_i) \rangle_{\gamma_x} - \|\sff_x(u,e_i)\|^2_{\gamma_x} $$  
for any orthonormal basis $\{e_i\}_{i=1}^m$ of $T_{\mathbf{M}}(x)$.
\label{l:ricciabstract}
\end{defn} 

Though the notion of Ricci curvature may appear somewhat abstract, it can be geometrically understood as controlling the evolution of volume under geodesic flow. 
More precisely, given a point $x \in \mathbf{M}$ and a tangent direction $v \in T_{\mathbf{M}}(x)$, consider any small neighborhood $C$ (of any shape) of $x$. 
Let $C_t$ be the {\em evolved} form of $C$ in the direction of $v$: Namely let $C_t = \{\psi_t(x): x \in C\}$, where $\psi_t(x)$ is the point on the geodesic that passes through $x$ 
in the direction of $v$ at time $t$. Then, we have (see e.g. \citep{ollivier2010survey}): 
$$\mbox{vol}(C_t) = \mbox{vol}(C)\left(1 - \frac{t^2}{2} \mbox{Ric}(v) + o(t^2)\right) $$ 
Some helpful canonical examples to keep in mind: a sphere has positive Ricci curvature and hyperbolic space has negative Ricci curvature.

\subsection{Lie group manifolds with invariant metrics}

Finally, we will also need a few classic results regarding sectional and Ricci curvatures on manifolds coming from Lie groups with an invariant metric. 
In the interest of keeping the notation and background light, we will take a somewhat unorthodox approach and will not define Lie brackets/algebras from scratch, and 
will instead define all relevant notions through the lense of matrix Lie groups (see below). 

First, the definition of a Lie group:
\begin{defn}[Lie group] A Lie group is a set $G$ which has both manifold and group structure, with group operation $\star$. Furthermore, the map 
$$\rho: \rho(p,q) = p \star q^{-1}, p, q \in G $$
is $C^{\infty}$-smooth.  
\end{defn} 



A particularly relevant kind of Lie group is a subgroup of $GL_n(\Real)$: 
\begin{defn}[Matrix Lie group] A manifold $G$ which is a subgroup of $GL_n(\Real)$ with the induced matrix multiplication group operation is called a \emph{matrix Lie group}.  
\end{defn}

There are two reasons why matrix Lie groups are particularly convenient. On the one had, certain calculations on matrix Lie group are often very easy to do (in particular, the Lie bracket has an explicit expression). Furthermore, many interesting groups can be embedded as a matrix Lie group. In particular, we have the following: 

\begin{defn}[Lie bracket] Let $G$ be a Lie group, and $\phi: G \to GL_n(\Real)$ be a homomorphism, namely 
$$\phi(p \star q) = \phi(p) \phi(q) $$
Furthermore, let $\phi_*$, the pushforward of $\phi$, be a bijection at $e \in G$, the identity element.  
The \emph{Lie bracket} $[\cdot, \cdot]$ on $T_e(G)$ is a bilinear form, s.t. 
$$[U, V]_G = \phi_*^{-1}\left(\phi_*(U) \phi_*(V) - \phi_*(V) \phi_*(U)\right)$$
\label{d:lbkt}
\end{defn} 
We note that for those acquainted with Lie groups -- the above theorem is actually a consequence of the infinitesimal Lie group representation theorem, though stating this theorem properly requires defining the Lie bracket through the differentiation view of vector fields on manifolds, so we refer the reader to \cite{varadarajan2013lie}.

\begin{defn}[Left invariant metric] Let $G$ be a Lie group, and let the translation $L_g: G \to G$ be defined as $L_g(u) = g u$ for $g, u \in G$. If a metric satisfies, 
$$ \forall g:  \langle u, v \rangle_x  = \left\langle (L_{g})_* u, (L_{g})_* v  \right\rangle_x , \forall u, v \in T_x G $$
where $(L_g)_*$ is the pushforward of the map $L_g$, the metric is called left-invariant.  
\label{d:lint}
\end{defn}

In a classic result, Milnor gave the following simple expressions for the Riemannian tensor and the sectional curvature:
\begin{thm}[Curvature of Lie group manifold, \citep{milnor1976curvatures, anderson2010introduction}] Let $G$ be a Lie group with Lie bracket $[\cdot, \cdot]_G$ with left-invariant metric $\gamma$. Then,   
\begin{enumerate}
\item[(a)] $\mbox{Ric}(v) = \left\langle \frac{1}{4} \sum_{i=1}^m [[v,e_i]_G, e_i]_G, v \right\rangle_{\gamma}$
where $\{e_i\}^m_{i=1}$ is an orthonormal basis of the $T_e(G)$.  
\item[(b)] If $G = \mbox{SO}(k)$ equipped with the left-invariant metric $\langle A,B\rangle_{\gamma} = \mbox{Tr}(A^T B)$, we have 
$$\frac{1}{4} \sum_{i=1}^m \left\langle [[v,e_i]_G, e_i]_G, v \right\rangle_{\gamma} = \frac{k-2}{4} \|v\|_{\gamma}$$ 
and hence 
$\mbox{Ric}(v) = \frac{k-2}{4} \|v\|_{\gamma}$
\end{enumerate}
\label{l:curvlie}
\end{thm}

%


\subsection{Diffusion processes and mixing time bounds}

In this section, we introduce the key definitions related to continuous Markov chains and diffusion processes: 

\begin{defn}[Markov semigroup]
We say that a family of functions $\{P_t(x, y)\}_{t \geq 0}$ on a state space $\Omega$ is a Markov semigroup if $P_t(x, \cdot)$ is a distribution on $\Omega$ and 
$$\Pr_{t+s}(x, y) = \int_{\Omega} P_t(x, z) P_s(z, y) dz$$
for all $x, y \in \Omega$ and $s, t \geq 0$. 
\end{defn}

\begin{defn}[Continuous time Markov processes]
A continuous time Markov process $(X_t)_{t\ge 0}$ on state space $\Omega$ is defined by a Markov semigroup $\{P_t(x, y)\}_{t \geq 0}$ as follows. For any measurable $A \subseteq \Omega$
$$
\Pr(X_{s+t} \in A) = \int_A P_t(x,y) dy := P_t(x,A) 
$$
Moreover $P_t$ can be thought of as acting on a function $g$ as
\begin{align*}
(P_t g)(x) &= \E_{P_t(x,\cdot)} [g(y)] = \int_{\Omega}  g(y) P_t(x,y) dy
\end{align*}
Finally we say that $p(x)$ is a stationary distribution if $X_0\sim p$ implies that $X_t\sim p$ for all $t$. 
\end{defn}


\begin{defn}
The generator $\mathcal{L}$ of the Markov Process is defined (for appropriately restricted functionals $g$) as 
\begin{align*}
\mathcal{L} g &= \lim_{t \to 0} \frac{P_t g - g}{t}.
\end{align*}
Moreover if $p$ is the unique stationary distribution, the Dirichlet form and the variance are
$$\mathcal{E}_M(g,h) = -\E_p\langle g,  \mathcal{L} h \rangle \mbox{ and } \mbox{Var}_p(g) = \E_p(g-\E_p g)^2$$
respectively. We will use the shorthand $\mathcal{E}(g):=\mathcal{E}(g,g)$. 
\end{defn}

Next, we define the Poincar\'e constant, which captures the spectral expansion properties of the process:  
\begin{defn} [Poincar\'e inequality]
A continuous Markov process satisfies a Poincar\'e inequality with constant $C$ if for all functions $g$ such that $\mathcal{E}_M(g)$ is defined (finite),\footnote{We will implicitly assume this condition whenever we discuss Poincar\'e inequalities. } 
\begin{align*}
\mathcal{E}_M(g) \ge \frac{1}{C} \mbox{Var}_p(g).
\end{align*}
We will abuse notation, and for a Markov process with stationary distribution $p$, denote by $C_P(p)$ the \emph{Poincar\'e constant of $p$}, the smallest $C$ such that above Poincar\'e inequality is satisfied.   
\end{defn}

Finally, we introduce a particular Markov process, the Langevin diffusion: 
\begin{defn} [Langevin diffusion] The Langevin diffusion is the following stochastic process: 
\begin{equation} d X_t = -  \nabla f(X_t) dt + \sqrt{2} d B_t \label{eq:lang}\end{equation} 
where $f: \Real^N \to \Real$, $d B_t$ is Brownian motion in $\Real^N$ with covariance matrix $I$. 
Under mild regularity conditions on $f$, the stationary distribution of this process is $p(X): \Real^N \to \Real$, s.t. $p(X) \propto e^{-f(X)}$. 
\label{def:Langevin}
\end{defn} 

We will also need the following reflected Langevin diffusion process, which has a stationary measure a restriction of the usual Langevin distribution to a region $\mathcal{D}$. 
\begin{defn} [Restricted Langevin diffusion, \cite{lions1984stochastic, saisho1987stochastic}] For a sufficiently regular region $\mathcal{D}$, 
there exists a measure $L(x)$ supported on $\mathcal{D}$, s.t.  the stochastic differential process 
\begin{equation} d \tilde{X}_t = -\nabla f(\tilde{X}_t) dt + \sqrt{2} dB_t + \nu_t L(\tilde{X}_t) dt 
 \end{equation} 
where $f: \Real^N \to \Real$, $d B_t$ is Brownian motion in $\Real^N$ with covariance matrix $I$ 
and $\nu_t$ is an outer normal unit vector to $\mathcal{D}$ has as stationary measure $p(X): \mathcal{D} \to \Real$, s.t. $p(X) \propto e^{-f(X)}$.  
\label{def:restrictedlangevin}
\end{defn} 

The generator of the (either restricted, or unrestricted) Langevin diffusion is $\mathcal{L}$, s.t. 
$$\mathcal{L} g = - \langle \nabla f, g \rangle + \Delta g$$    
For the restricted Langevin diffusion, we understand the generator to be defined with a Neumann condition (hence the absence of the boundary term): namely it's to be understood as acting on functionals $g$, s.t. $\nabla_n g = 0$, where $n$ is the vector field of inward-pointing normals to $\mathcal{D}$. Hence, 
$\mathcal{E}_M(g) = E_{p}\| \nabla g\|^2$.
Since this depends in a natural way on $p$, we will also write this as $\mathcal{E}_p(g)$. \\
A Poincar\'e inequality for Langevin diffusion thus takes the form
\begin{equation}
\E_p\|\nabla g\|^2 \ge \frac{1}{C} \mbox{Var}_p(g)
\label{eq:poincarelangevin}
\end{equation}

The above definitions were defined over Euclidean space, but they have natural analogues over manifolds as well. 
More concretely, we will say: 
\begin{defn}[Poincar\'e inequality over manifold] The distribution $p(x) = \frac{e^{-f(x)}}{Z}$ over a submanifold $\mathbf{M} \subseteq \Real^{N}$ equipped with a metric $\gamma$ satisfies a Poincar\'e inequality with constant $C$ if for all differentiable $g: \mathbf{M} \to \Real$, we have
$$ \E_p\|\nabla_{\mathbf{M}} g\|_{\gamma}^2 \ge \frac{1}{C} \mbox{Var}_p(g)$$ 
where the norm $\gamma$ is induced by the manifold metric, and $\nabla_{\mathbf{M}} g$ is the gradient with respect to the manifold $\mathbf{M}$.
\label{eq:poincaremanifold} 
\end{defn}

We note, above we mean a distribution in the sense of Definition \ref{d:distrm}. The variance on the right is of course understood by integrating with respect to the volume form of the metric $\gamma$ as in Definition \ref{d:globalvol}. (This will be particularly important in Lemma \ref{l:pconstr}.) Finally, we note $\nabla_{\mathbf{M}} g(x)$ may not equal $\Pi_{T_x(\mathbf{M})} \nabla g(x) $ if the metric is not the standard Euclidean metric. 

We will crucially use the following interplay between the Poincar\'e constant of a distribution over a manifold $\mathbf{M}$ and the Ricci curvature of the manifold $\mathbf{M}$:   
\begin{lem}[Ricci and Poincar\'e, \cite{hsu2002stochastic, bakry1985diffusions}] Suppose a distribution $p(x) = \frac{e^{-f(x)}}{Z}$ over a compact submanifold $\mathbf{M}$ equipped with metric $\gamma$ satisfies
$$\forall x \in \mathbf{M}, v \in T_x(\mathbf{M}), \|v\|_{\gamma} = 1: \hspace{0.5 cm} \nabla^2 f(v,v) + \mbox{Ric}(v) \geq \lambda $$  
for $\lambda > 0$, where $\nabla ^2 f(v,v)$\footnote{Note, $\nabla^2 f$ is the Hessian with respect to the manifold, and in general will not agree with the standard Euclidean Hessian.} is defined as  
$$\nabla^2 f(v,v):= \langle v, \nabla^2 f(x) v \rangle_{\gamma}  $$ 
Then, the Poincar\'e constant of $p$ satisfies
$C_P(p) \leq \frac{2}{\lambda}$.
\label{d:bcp}
\end{lem}

Finally, we also need the following well-known result about measures over convex subsets of $\mathbb{R}^d$:  
\begin{lem}[Log-concave measure over convex set, \cite{bebendorf2003note}] Suppose a measure $p:\mathbb{R}^d \to \mathbb{R}$ of the form $p(x) = \frac{e^{-f(x)}}{Z}$ is supported over $S \subseteq \Real^d$ which is convex, and $\forall x \in S, \nabla^2 f \gtrsim 0$. 
Then, the Poincar\'e constant of $p$ satisfies
$C_P(p) \leq \frac{\mbox{diam}(S)}{\pi}$.
\label{l:constrained}
\end{lem} 

We will also several times use the following perturbation lemma on the Poincar\'e constant of a distribution:

\begin{lem}[Holley-Stroock perturbation] 
\label{l:holleystroock} 
Let $q: \Omega \to \mathbb{R}^{+}, q(x) \propto e^{f(x)}$ be a probability distribution over a domain $\Omega$, and let $\psi: \Omega \to \mathbb{R}$ be a bounded function. 
Then, if $\tilde{q}: \Omega \to \mathbb{R}^{+}$ is defined as $\tilde{q}(x) \propto e^{f(x) + \psi(x)}$, 
$$C_P(\tilde{q}) \leq C_P(q) e^{\mbox{osc}(\psi)}$$ 
where $\mbox{osc}(\psi) = \max_{x \in \Omega} \psi(x) - \min_{x \in \Omega} \psi(x) $.  
\end{lem} 
We note that $\mbox{osc}$ is of course tied to the domain of $\psi$. In particular, we will, for a function $\psi$, use the notation $\psi_{|A}$ to denote the restriction of $\psi$ to set $A$.

Finally, we will also need the following well-known lemmas about distances between distributions: 
\begin{lem}[Coupling Lemma] Let $p,q: \Omega \to \mathbb{R}$ be two distributions, and $c: \Omega^{\otimes 2} \to \mathbb{R}$ be any coupling of $p,q$. 
Then, if $(X,X')$ are random variables following the distribution $c$, we have 
$$d_{\mbox{TV}}(p,q) \leq 2 \Pr[X \neq X'] $$ 
\label{l:coupling}
\end{lem}

\begin{lem} [Inequality between TV and $\chi^2$] Let $p,q$ be probability measures, s.t. $p$ is absolutely continuous with respect to $q$. We then have:  
\begin{align*} 
\mbox{TV}(p,q) &\leq \frac{1}{2} \sqrt{\chi^2(p,q)} 
\end{align*} 
\label{l:ineqsprob}
\end{lem} 


%

%
%




%% file: heuristic_vaguer_rescaled.tex
\section{Decomposition recipe: proof of Theorem \ref{l:abstract}}
\label{l:gensetup}

In this section, give the formal proof of Theorem \ref{l:abstract}

Recalling that the measure $\tilde{p}$ is the stationary measure of the SDE 
$$d \tilde{X}_t = -\beta \nabla f(\tilde{X}_t) dt + \sqrt{2} dB_t + \nu_t L(\tilde{X}_t) dt $$
for $L(\tilde{X}_t)$ a measure supported on $\mathcal{D}$, it satisfies a Poincare inequality with constant $\mainpc $, if 
\begin{equation} \mbox{Var}_{\tilde{p}}(g) \leq \mainpc \E_{\tilde{p}} \|\nabla g \|^2  \label{eq:poincare}\end{equation}
for appropriately restricted functionals $g: \mathbb{R}^N \rightarrow \mathbb{R} $. 

Towards decomposing the left hand side of \eqref{eq:poincare}, we will use the law of total variance and the co-area formula. The co-area formula manifests through Lemma \ref{l:general}, the proof of which is by the definitions of $\tilde{p}^{\Delta}$ and $q$ and Theorem \ref{l:coarea}. We note that similar decomposition theorems, modulo the measure-theoretic elements have appeared before (e.g. \cite{lelievre2009general}, Theorem D.3 in \cite{ge2018simulated}, Lemma 1 in \cite{mou2019sampling}). 


Given this Lemma, we will extract a Poincar\'e constant on $\tilde{p}$:  

\begin{lem}[Poincar\'e inequality for $\tilde{p}$] Under assumptions (1),(2) and (3), the distribution $\tilde{p}$ satisfies a Poincar\'e inequality with Poincare constant 
$$C_P(\tilde{p}) = O\left(\max\left(1, C_{\mbox{level}}\right) \max\left(1,C_{\mbox{across}}\right)\max\left(1, C^2_{\mbox{change}}\right)\right)$$ 
\label{l:poincaretilde}
\end{lem} 
\begin{proof} 
We wish to show that for any functional $g$, we have 
$$\mbox{Var}_{\tilde{p}}(g) \leq \left(C_{\mbox{level}} + C_{\mbox{across}}\right) \E_{\tilde{p}} \|\nabla g\|^2  $$ 
Without loss of generality, it suffices to consider $\E_{\tilde{p}}(g) = 0$ 

By Lemma \ref{l:general}, we have 
$$ \mbox{Var}_{\tilde{p}}(g) = \E_{\Delta \sim q} \mbox{Var}_{X \sim \tilde{p}^{\Delta}}(g) + \mbox{Var}_{\Delta \sim q} (\E_{X \sim \tilde{p}^{\Delta}} g) $$ 

We will upper bound each of these terms: namely we will show
\begin{equation}  \E_{\Delta \sim q} \mbox{Var}_{X \sim \tilde{p}^{\Delta}}(g)  \leq C_{\mbox{level}} \E_{\tilde{p}} \left\|\nabla g\right\|^2  \label{eq:poincareclaim}\end{equation}
and  
\begin{equation} \mbox{Var}_{\Delta \sim q} (\E_{X \sim \tilde{p}^{\Delta}} g)  \leq 2 C_{\mbox{across}} \left(C_{\mbox{level}} + C_{\mbox{level}} C^2_{\mbox{change}}\right) \E_{\tilde{p}} \|\nabla g\|^2  \label{eq:tubeclaim}\end{equation} 

By Condition \ref{c:along}, the distribution $\tilde{p}^{\Delta}$ satisfies a Poincar\'e inequality with Poincar\'e constant $C_{\mbox{level}}$. Hence, 
\begin{align*} \E_{\Delta \sim q} \mbox{Var}_{X \sim \tilde{p}^{\Delta}}(g) &\leq \E_{\Delta \sim q} C_{\mbox{level}} \E_{X \sim \tilde{p}^{\Delta}} \|\nabla_{\mathbf{M}^{\Delta}} g\|^2 \\
&\leq  C_{\mbox{level}} \E_{X \sim \tilde{p}}\|\nabla g\|^2
\end{align*}  
where the last inequality follows since $\|\nabla_{\mathbf{M}^{\Delta}} g\|^2 \leq \|\nabla g\|^2$ by Proposition \ref{d:deronman}. Thus, \eqref{eq:poincareclaim} follows.

By Condition \ref{c:across}, we have
\begin{equation} \E_{\Delta \sim q}\left\|\nabla_{\mathbf{B}} \E_{\tilde{p}^{\Delta}} g\right\|^2 \geq \frac{1}{C_{\mbox{across}}}  \mbox{Var}_{\Delta \sim q} \left(\E_{\tilde{p}^{\Delta}} g\right) \label{eq:secondt}\end{equation}
We will analyze the left-hand side more carefully. Towards that, let us define by $G_{\Delta}: \mathbf{M} \to \mathbf{M}^{\Delta}$ the map $G_{\Delta}(X) = X + \phi_X(\Delta)$. 
Expanding out the expectation in terms of the definition of $\mathbf{M}^{\Delta}$, we have
\begin{align*} \E_{\tilde{p}^{\Delta}} g &= \int_{Y \in \mathbf{M}^{\Delta}} g(Y) p^{\Delta}(Y) d\mathbf{M}^{\Delta}(Y) \\  
&= \int_{X \in \mathbf{M}} g\left(X + \phi_X(\Delta)\right) p^{\Delta}(X + \phi_X(\Delta)) \mbox{det}\left((dG_{\Delta})_X\right) d\mathbf{M}(X)
\end{align*}
where the last line follows from Definition \ref{d:globalvol}.
\Anote{Maybe cite lemmas about differentiating under integral and product rule for manifolds?} 
Differentiating under the integral and using the product rule, we have 
\begin{align*} &\nabla_{\mathbf{B}}\left(\int_{X \in \mathbf{M}} g(X + \phi_X(\Delta)) p^{\Delta}(X + \phi_X(\Delta)) \mbox{det}\left((dG_{\Delta})_X\right) d\mathbf{M}(X)\right) \\ 
&= \underbrace{\int_{X \in \mathbf{M}} \nabla_{\mathbf{B}} g(X + \phi_X(\Delta)) p^{\Delta}(X + \phi_X(\Delta)) \mbox{det}\left((dG_{\Delta})_X\right) d\mathbf{M}(X)}_{\mbox{I}} \\
&+ \underbrace{\int_{X \in \mathbf{M}} g(X + \phi_X(\Delta)) \nabla_{\mathbf{B}}\left(p^{\Delta}(X + \phi_X(\Delta)) \mbox{det}\left((dG_{\Delta})_X\right)\right) d\mathbf{M}(X)}_{\mbox{II}}
\end{align*}

From $\|a+b\|^2 \leq 2(\|a\|^2 + \|b\|^2)$, we have 
\begin{equation} \left\|\nabla_{\mathbf{B}} \E_{\tilde{p}^{\Delta}} g\right\|^2 \leq 2\left(\|\mbox{I}\|^2 + \|\mbox{II}\|^2\right)\label{eq:bbound}\end{equation}

We consider each of the terms I and II individually. 

Proceeding to I, we will show that 
\begin{equation}  \|\mbox{I}\|^2 \leq \E_{\tilde{p}^{\Delta}} \|\nabla g\|^2 \label{eq:ibound} \end{equation} 
We have:
\begin{align*}
0 &\leq \int_{X \in \mathbf{M}} \|\nabla_{\mathbf{B}} g(X + \phi_X(\Delta)) - \mbox{I}\|^2_2 p^{\Delta}(X + \phi_X(\Delta)) \mbox{det}\left((dG_{\Delta})_X\right) d\mathbf{M}(X) \\ 
&= \int_{X \in \mathbf{M}} \left(\|\nabla_{\mathbf{B}} g(X + \phi_X(\Delta))\|^2 - 2 \left\langle \nabla_{\mathbf{B}} g(X + \phi_X(\Delta)), \mbox{I}\right\rangle + \|\mbox{I}\|^2_2\right) p^{\Delta}(X + \phi_X(\Delta)) \mbox{det}\left((dG_{\Delta})_X\right) d\mathbf{M}(X)\\ 
&\stackrel{\mathclap{\circled{1}}}{=} \int_{X \in \mathbf{M}} \|\nabla_{\mathbf{B}} g(X + \phi_X(\Delta))\|^2 p^{\Delta}(X + \phi_X(\Delta)) \mbox{det}\left((dG_{\Delta})_X\right) d\mathbf{M}(X) - \|\mbox{I}\|^2_2 \\
&\stackrel{\mathclap{\circled{2}}}{\leq} \int_{X \in \mathbf{M}} \|\nabla g(X + \phi_X(\Delta))\|^2 p^{\Delta}(X + \phi_X(\Delta)) \mbox{det}\left((dG_{\Delta})_X\right) d\mathbf{M}(X) - \|\mbox{I}\|^2_2 \\
&= \E_{\tilde{p}^{\Delta}} \|\nabla g\|^2 - \|\mbox{I}\|^2_2
\end{align*} 
where $\circled{1}$ follows since 
$$ \int_{X \in \mathbf{M}} \langle \nabla_{\mathbf{B}} g(X + \phi_X(\Delta)), \mbox{I} \rangle p^{\Delta}(X + \phi_X(\Delta)) \mbox{det}\left((dG_{\Delta})_X\right) d\mathbf{M}(X) = -\|\mbox{I}\|^2_2$$ 
and $\circled{2}$ follows from Proposition \ref{d:deronman}. 

Proceeding to II, we will show 
\begin{equation}  \|\mbox{II}\|^2 \leq C^2_{\mbox{change}} \E_{\tilde{p}^{\Delta}} (g^2) \label{eq:iibound} \end{equation} 

\begin{align*} 
&\|\mbox{II}\|^2 = \left\|\int_{X \in \mathbf{M}} g(X + \phi_X(\Delta)) \nabla_{\mathbf{B}}\left(p^{\Delta}(X + \phi_X(\Delta)) \mbox{det}\left((dG_{\Delta})_X\right)\right) d\mathbf{M}(X)\right\|^2 \\ 
&\stackrel{\mathclap{\circled{1}}}{\leq} \left\|\int_{X \in \mathbf{M}} \left(g(X + \phi_X(\Delta)) - \E_{\tilde{p}^{\Delta}}(g)\right) \nabla_{\mathbf{B}}\left(p^{\Delta}(X + \phi_X(\Delta)) \mbox{det}\left((dG_{\Delta})_X\right)\right) d\mathbf{M}(X)\right\|^2 \\
&= \left\|\int_{X \in \mathbf{M}} \left(g(X + \phi_X(\Delta)) - \E_{\tilde{p}^{\Delta}}(g)\right) \frac{\nabla_{\mathbf{B}} \left(p^{\Delta}(X + \phi_X(\Delta)) \mbox{det}\left((dG_{\Delta})_X\right)\right)}{p^{\Delta}(X + \phi_X(\Delta)) \mbox{det}\left((dG_{\Delta})_X\right)}p^{\Delta}(X + \phi_X(\Delta)) \mbox{det}\left((dG_{\Delta})_X\right) d\mathbf{M}(X)\right\|^2 \\ 
&\stackrel{\mathclap{\circled{2}}}{\leq} \int_{X \in \mathbf{M}} \left\|\left(g(X + \phi_X(\Delta)) - \E_{\tilde{p}^{\Delta}}(g)\right) \frac{\nabla_{\mathbf{B}} \left(p^{\Delta}(X + \phi_X(\Delta)) \mbox{det}\left((dG_{\Delta})_X\right)\right)}{p^{\Delta}(X + \phi_X(\Delta)) \mbox{det}\left((dG_{\Delta})_X\right)}\right\|^2 p^{\Delta}(X + \phi_X(\Delta)) \mbox{det}\left((dG_{\Delta})_X\right) d\mathbf{M}(X)\\
&\leq C^2_{\mbox{change}} \int_{X \in \mathbf{M}} \left(g(X + \phi_X(\Delta)) - \E_{\tilde{p}^{\Delta}}(g)\right)^2 p^{\Delta}(X + \phi_X(\Delta)) \mbox{det}\left((dG_{\Delta})_X\right) d\mathbf{M}(X)\\
&= C^2_{\mbox{change}} \mbox{Var}_{\tilde{p}^{\Delta}} (g) \\
\end{align*} 
where $\circled{1}$ follows since 
\begin{align*} &\int_{X \in \mathbf{M}} \nabla_{\mathbf{B}}\left(p^{\Delta}(X + \phi_X(\Delta)) \mbox{det}\left((dG_{\Delta})_X\right)\right) d\mathbf{M}(X) \\ 
&= \nabla_{\mathbf{B}} \left(\int_{X \in \mathbf{M}}\left(p^{\Delta}(X + \phi_X(\Delta)) \mbox{det}\left((dG_{\Delta})_X\right)\right) d\mathbf{M}(X) \right) \\
&= \nabla_{\mathbf{B}}(1) \\ 
&= 0
\end{align*}
and $\circled{2}$ follows by Jensen's inequality. 

Plugging \eqref{eq:ibound} and \eqref{eq:iibound} in \eqref{eq:bbound} and subsequently in \eqref{eq:secondt}, we have 
\begin{align*} 
\mbox{Var}_{\Delta \sim q} \left(\E_{\tilde{p}^{\Delta}} g\right) &\leq C_{\mbox{across}} \E_{\Delta \sim q}\left\|\nabla_{\mathbf{B}} \E_{\tilde{p}^{\Delta}} g\right\|^2 \\
&\leq 2 C_{\mbox{across}} \left(C^2_{\mbox{change}} \E_{\Delta \sim q} \mbox{Var}_{\tilde{p}^{\Delta}} (g) + \E_{\Delta \sim q} \E_{\tilde{p}^{\Delta}} \|\nabla g\|^2\right) \\ 
&\leq 2 C_{\mbox{across}} \left(C^2_{\mbox{change}} C_{\mbox{level}} \E_{\tilde{p}} \|\nabla g\|^2 + \E_{\tilde{p}} \|\nabla g\|^2\right) \\
&= 2 C_{\mbox{across}} \left(C^2_{\mbox{change}} C_{\mbox{level}} +1\right)\E_{\tilde{p}} \|\nabla g\|^2
\end{align*} 

Putting these two inequalities together, and using Lemma \ref{l:general} we have 
\begin{align*} 
\mbox{Var}_{s \sim q} (\E_{\tilde{p}^s} g) &\leq 2 C_{\mbox{across}} \left(1 + C_{\mbox{level}} + C_{\mbox{level}} C^2_{\mbox{change}} \right) \E_{\tilde{p}} \|\nabla g\|^2 
\end{align*} 

Hence, \eqref{eq:tubeclaim} attains, which finishes the proof of the lemma. 
\end{proof}   

With this in hand we proceed to proving mixing bounds for $\tilde{p}$. Note that it is fairly standard that the Poincar\'e inequality implies fast mixing in $\chi^2$, but we repeat it here for completeness. 
(Note, this bound is for the \emph{restricted} Langevin diffusion process! We will relate it to the unrestricted diffusion in the following lemma.)
 
Precisely, we show:  

\begin{lem}[Mixing in $\chi^2$ from Poincar\'e] Let $\tilde{X}_t$ follow the SDE 
$$d \tilde{X}_t = -\beta \nabla f(\tilde{X}_t) dt + \sqrt{2} dB_t + \nu_t L(\tilde{X}_t) dt $$
where $L(\tilde{X}_t)$ is a measure supported on $\partial\mathcal{D}$ such that
the stationary measure of $\tilde{X}_t$ is $\tilde{p}$.
Let $\tilde{p}_0$ be absolutely continuous with respect to the Lebesgue measure, $\tilde{p}_t$ be the pdf of $\tilde{X}_t$, and $\mainpc$ the Poincar\'e constant of $\tilde{p}$. Then: 
\begin{enumerate}
\item[(1)] If $\tilde{p}_0$ is supported on $\mathcal{D}$, $\tilde{p}_t$ is supported on $\mathcal{D}, \forall t > 0$.
\item[(2)] $\chi^2(\tilde{p}_t, \tilde{p}) \leq  e^{-t/\mainpc} \chi^2(\tilde{p}_0 , \tilde{p})$ 
\end{enumerate}
\label{l:chi-poincare}
\end{lem}
\begin{proof}
Condition (1) follows from the properties of the drift $L$.

Condition (2) is a consequence of a Poincar\'e inequality. We include the proof here for completeness:
The Poincar\'e inequality implies for every $\langle g, \nabla_n g \rangle = 0$, we have 
$$\E_{\tilde{p}} (P_t g - \E_{\tilde{p}} g)^2 \leq e^{-t/\mainpc} \E_{\tilde{p}} (g - \E_{\tilde{p}} g)^2  $$ 
Consider the functional $g = \frac{\tilde{p}_0}{\tilde{p}}$, which is in the domain of $\mathcal{L}$: indeed, since the support of $\tilde{p}_0$ is $\mathbf{D}$,
the support of $g$ is $\mathbf{D}$, and $\langle g, \nabla_n g \rangle = 0$. Hence, we have by the Poincar\'e inequality  
$$\E_{\tilde{p}} \left(P_t  \frac{\tilde{p}_0}{\tilde{p}} - \E_{\tilde{p}}  \frac{\tilde{p}_0}{\tilde{p}}\right)^2 \leq e^{-t/\mainpc} \E_{\tilde{p}} \left( \frac{\tilde{p}_0}{\tilde{p}} - \E_{\tilde{p}}  \frac{\tilde{p}_0}{ \tilde{p}}\right)^2  $$ 
Since $P_t  \tilde{p}_0 = \tilde{p}_t$, and $\E_{\tilde{p}}  \frac{\tilde{p}_0}{\tilde{p}} = 1$, we have 
$$\E_{\tilde{p}} \left( \frac{\tilde{p}_t}{\tilde{p}} -1 \right)^2 \leq e^{-t/\mainpc} \E_{\tilde{p}} \left( \frac{\tilde{p}_0}{\tilde{p}} - 1\right)^2  $$ 
By the definition of $\chi^2$, we have 
$\chi^2(\tilde{p}_t, \tilde{p}) \leq e^{-t/\mainpc} \chi^2(\tilde{p}_0, \tilde{p})   $
which completes the proof. 
\end{proof} 

Next, using Assumption (1), we can prove that these two Langevin processes track each other fairly well. Namely, we show: 

\begin{lem}[Comparing restricted vs normal chain] 
Let $\tilde{X}_t$ follow the stochastic differential equation 
$$d \tilde{X}_t = -\nabla f(\tilde{x}_t) dt + \sqrt{2} dB_t + \nu_t L(\tilde{x}_t) dt $$ where
$L(\tilde{X}_t)$ is a measure supported on $\{t \geq 0: \tilde{X}_t \in \partial \mathcal{D}\}$, s.t.  
the stationary measure of $\tilde{X}_t$ is $\tilde{p}$. 

Let $\tilde{p}_t$ be the pdf of $\tilde{X}_t$ and let $\tilde{p}_0$ be absolutely continuous with respect to the Lebesgue measure. Then, if $p_t$ is the pdf of 
$$d X_t = -\nabla f(\tilde{X}_t) dt + \sqrt{2} dB_t $$ 
it holds that
$d_{\mbox{TV}}(p_t,\tilde{p}) \leq \epsilon + \sqrt{\chi^2(p_0, \tilde{p})} e^{-t/2\mainpc} $
for $t \leq T$. 
\label{l:restricted}
\end{lem}
\begin{proof} 

Consider the coupling of $X_t, \tilde{X}_t$, s.t. the Brownian motion $dB_t$ is the same for $X_t, \tilde{X}_t$.  
By Lemma~\ref{l:coupling} and Lemma~\ref{l:concentrationwrapper}, we have
\begin{equation} d_{\mbox{TV}}(p_t, \tilde{p}_t) \leq \Pr[X_t \neq \tilde{X}_t] \leq \Pr[\exists s \in [0,t], X_t \notin \mathcal{D}] \leq \epsilon \label{eq:dtv}\end{equation}
where the last inequality follows by Assumption (1). 

Then, consider the total variation distance between $p_t$ and $\tilde{p}$: we have 
\begin{align*} d_{\mbox{TV}}(p_t, \tilde{p}) &\leq d_{\mbox{TV}}(p_t, \tilde{p}_t) + d_{\mbox{TV}}(\tilde{p}_t, \tilde{p}) \\
&\leq d_{\mbox{TV}}(p_t, \tilde{p}_t) + \sqrt{\chi^2(\tilde{p}_t, \tilde{p})}  
\end{align*} 
where the first inequality follows by the triangle inequality, and the second by Lemma \ref{l:ineqsprob}.

By Lemma~\ref{l:chi-poincare}, we have $\chi^2(\tilde{p}_t, \tilde{p}) \leq \chi^2(\tilde{p}_0,\tilde{p}) e^{-t/\mainpc}$, which together with \eqref{eq:dtv} finishes the proof of the Lemma. 
\end{proof}

Putting Lemmas \ref{l:chi-poincare} and \ref{l:restricted} together, Theorem~\ref{l:abstract} immediately follows.

%% file: torus.tex
\section{Warmup: proving the theorem for a torus} 
\label{s:warmuptorus} 

In order to provide some intuition, we will first consider a very simple setting: the manifold of optima in consideration will be a circle $\mathcal{C}$ embedded in $\Real^3$, namely
$$\mathcal{C} = \{(x,y,z): x^2 + y^2 = 1, z = 0\} $$ 
We will set $f$ to be the distance from the circle: namely $f(x) = \|x - \Pi_{\mathcal{C}}(x)\|^2_2$.  

\subsection{Instantiating the decomposition framework}
 
With this in mind, we will implement the framework described in Section \ref{s:overview}. 
To set up notation, notice that the set of points s.t. $\{x: f(x) = s^2\}$ form a torus, which can be described 
in spherical-like coordinates as  
$$T: [0,2\pi)^2 \to \mathbb{R}^3, \mbox{s.t. } T(u,v) = ((1 + s \cos v) \cos u, (1 + s \cos v) \sin u, s \sin v) $$ 
Let us denote $F: \mathcal{D} \to  [0, s_{\max}] \times [0, 2\pi)$ be the mapping s.t. $F(x) = (s,v)$. 
We will partition $\mathcal{D}$ according to the pairs $(s,v)$ -- in other words, to instantiate the framework, we can choose $x_0 = (1,0,0)$, in which case $\mathbf{B} = \{\alpha (\cos v, 0, \sin v): v \in [0, 2\pi), \alpha \in [0, s_{\max})\}$. Furthermore, we choose 
$$\phi_{(\cos u, \sin u, 0)}\left(\alpha (\cos v, 0, \sin v)\right) = \alpha (\cos v \cos u, \cos v \sin u, \sin v)$$ 

The set of points with $(s,v)$ constant form a circle, which we denote $\mathbf{M}^{(s,v)}$ in accordance with the notation in Section \ref{s:overview}. We instantiate Theorem \ref{l:general} as   
$$\E_{x \sim \tilde{p}} \chi(x) = \E_{(s,v) \sim q} \E_{x \sim \tilde{p}^{(s,v)}} \chi(x)$$ 
where 

\begin{equation} q: [0, s_{\max}] \times [0, 2\pi) \to \Real, \mbox{ s.t. } q(s,v) \propto \int_{x \in \mathbf{M}^{(s,v)}} e^{-\beta^2 f(x)}  \frac{1}{|\mbox{det}(\bar{dF}_x)|} d\mathbf{M}^{(s,v)}(x) \label{l:rtorus}\end{equation}
where $\mbox{det}(\bar{dF}_x)$ is the normal determinant of $F$ and by $\tilde{p}^{(s,v)}$ the distribution
\begin{equation} \tilde{p}^{(s,v)}: \mathbf{M}^{(s,v)} \to \Real, \mbox{ s.t. } \tilde{p}^{(s,v)}(x) \propto e^{-\beta^2 f(x)} \frac{1}{|\mbox{det}(\bar{dF}_x)|} \label{l:svtorus}\end{equation}  

\subsection{Bounding $C_{\mbox{level}}$}

First, we proceed to show that $C_{\mbox{level}} \lesssim 1$. The strategy will be rather simple: we will show that $\tilde{p}^{(s,v)}$ is the uniform distribution over the circle $\mathbf{M}_{(s,v)}$. 

Note that $f(x)$ is constant over $\mathbf{M}^{(s,v)}$, so it will suffice to show that $\mbox{det}(\bar{dF}_x)$ is constant as well. 

Towards that, we will choose a particularly convenient basis for $\bar{dF}_x$. 
Keeping in mind the diffeomorphism  
$$X: [0,s_{\max}] \otimes [0,2\pi)^2 \to \mathcal{D}: X(s,u,v) = ((1 + s \cos v) \cos u, (1 + s \cos v) \sin u, s \sin v)$$ 
we have that the set of partial derivatives of $X$ forms a basis, namely:  
$$\left\{\begin{pmatrix}
\cos v \cos u \\
\cos v \sin u \\
\sin v \\
\end{pmatrix}, \begin{pmatrix}
-(1+s \cos v) \sin u  \\
(1+s \cos v) \cos u  \\
0 \\
\end{pmatrix}, \begin{pmatrix}
-s \sin v \cos u  \\
-s \sin v \sin u  \\
s \cos v \\
\end{pmatrix}   \right\} $$ 
In fact, it's easy to check that this basis is orthogonal. Furthermore, we claim that the kernel of $dF$ is spanned by the first vector. Indeed, for a curve parametrized as $\phi(t): (-1,1) \to \Real^3$, by the chain rule, we have 
$$\frac{\partial}{\partial t} F(\phi(t)) = dF_{\phi(t)}(\phi'(t))$$  
Consider the curve $\phi(t) = T(s, u + t, v)$. By the definition of $F$, since $s, v$ do not change along $\phi$, we have $\frac{\partial}{\partial t} F(\phi(t)) = 0$, which implies that  
$$ dF_x \left(\begin{pmatrix}
-(1+s \cos v) \sin u  \\
(1+s \cos v) \cos u  \\
0 \\
\end{pmatrix}\right) = 0$$ 
This implies that $\mbox{ker}(dF_x)^{\perp}$ is spanned by $\left\{\begin{pmatrix}
\cos v \cos u \\
\cos v \sin u \\
\sin v \\
\end{pmatrix}, \begin{pmatrix}
-s \sin v \cos u  \\
-s \sin v \sin u  \\
s \cos v \\
\end{pmatrix}   \right\}$ 

Furthermore, we claim the action of $dF_x$ in this basis can be easily described, considering the curves 
$\phi(t) = T(s+t, u, v)$ and $\phi(t) = T(s, u, v+t)$. Since $F(T(s+t, u, v)) = (s+t, v)$ and $F(T(s, u, v+t)) = (s, v)$, we have 
$$dF_x \left(\begin{pmatrix}
\cos v \cos u \\
\cos v \sin u \\
\sin v \\
\end{pmatrix}\right) = (1,0), \hspace{2 cm} dF_x \left(\begin{pmatrix}
-s \sin v \cos u  \\
-s \sin v \sin u  \\
s \cos v \\
\end{pmatrix}\right) = (0,1) $$ 

By linearity of the map $dF_x$, this implies that 
$$dF_x\left(\begin{pmatrix}
\cos v \cos u \\
\cos v \sin u \\
\sin v \\
\end{pmatrix}\right) = (1,0), \hspace{2 cm} dF(x) \left(\begin{pmatrix}
\sin v \cos u  \\
\sin v \sin u  \\
\cos v \\
\end{pmatrix}\right) = \frac{1}{s} (0,1) $$ 
which immediately implies that $\mbox{det}(\bar{dF}_x) = \frac{1}{s}$, from which we have that $\tilde{p}^{(s,v)}$ is the uniform distribution over the circle $\mathbf{M}^{(s,v)}$. Since the circle has Ricci curvature equal to the radius of the circle, by Lemma \ref{d:bcp} we have $C_{\mbox{level}} \lesssim 1$. 

\subsection{Bounding $C_{\mbox{across}}$}

This part is immediate: $r$ is supported on a convex set, since $(s,v) \in [0, s_{\max}] \otimes [0, 2\pi)$, and $e^{-\beta^2 s^2}s$ is a log-concave function of $(s,v)$. Hence, by Lemma \ref{l:constrained}, $C_{\mbox{across}} \lesssim 1$.

\subsection{Bounding $C_{\mbox{change}}$}  

Finally, we show $C_{\mbox{change}} = 0$. Since we showed that $\tilde{p}^{(s,v)}$ is the uniform distribution over $\mathbf{M}^{(s,v)}$, we have 
$\tilde{p}^{(s,v)}(x) = \frac{1}{2 \pi (1+s \cos v)}$. On the other hand, following the notation in Section \ref{s:overview}, and denoting 
$$G_{(s,v)}((\cos u, \sin u,0)) = ((1 + s \cos v) \cos u, (1 + s \cos v) \sin u, s \sin v) $$
We can calculate $\mbox{det}(dG_{(s,v)})$ as $\mbox{det}(dG_{(s,v)}) = \sqrt{\mbox{det}(J_G^T J_G)}$, where $J_G \in \Real^{3 \times 1}$ is the Jacobian of $G$. A simple calculation shows $\sqrt{\mbox{det}(J_G^T J_G)} = 1 + s \cos v$, so 
Hence, $\tilde{p}^{(s,v)}(x)\mbox{det}(dG_{(s,v)}) = \frac{1}{2 \pi}$ -- i.e. is independent of $(s,v)$, which implies that $C_{\mbox{change}}$ = 0.

%% file: matrix_fact.tex
\section{Matrix objectives: proofs of Theorem \ref{t:maincompletion} }\label{sec:setup}

In this section, we will provide the proof of Theorem 
\ref{t:maincompletion}.


\paragraph{Notation} 

In addition to the notation introduced in Section \ref{s:matrixoverview}, we will set $f(X) = \|\mathcal{A}(XX^{\top}) - b\|^2_2$ -- we will specify which linear operator $\mathcal{A}$ is in question, when the statement of a Lemma or Theorem depends on $\mathcal{A}$. 

We also set $N = d \times k$ and $m = \binom{k}{2}$ which are the ambient dimension and intrinsic dimension of the manifolds $\lopt_i$ respectively. We will often move from a matrix to a vector representation. To do so, $\mbox{vec}(X): \Real^{m \times n} \to \Real^{mn}$ will be defined as 
$$\mbox{vec}(X) = (X_{1,1}, X_{2,1}, \dots, X_{m,1}, \dots, X_{1,n}, X_{2,n}, \dots, X_{m,n})^T$$
Finally, we will denote $\mbox{Sym}^k$ the set of symmetric matrices in $\mathbb{R}^{k \times k}$.

The proof of Theorem 
\ref{t:maincompletion} will follow the recipe from Section \ref{s:recipe}, and we will establish each ingredient in a separate section. Namely, Section \ref{s:concentration} will establish nearness (Condition \ref{c:nearness}), Section \ref{s:levelsetp} a bound on $C_{\mbox{along}}$ (Condition \ref{c:along}), Section \ref{s:poincarr} a bound on $C_{\mbox{across}}$ (Condition \ref{c:across}) and Section \ref{s:gradienttoval} a bound on $C_{\mbox{change}}$. 
 

%% file: concentration_noiseless_rescaled.tex
\subsection{Maintaining Nearness to Manifold} 
\label{s:concentration}

In this section, we prove the concentration of the diffusion close to one of the manifolds $\lopt_i$. Recall that $N = dk$ is the ambient dimension. 
For notational convenience, we define the following neighborhoods:  
\begin{align} 
\mathcal{D}^{\mbox{mf}}_i &= \left\{X \in \mathbb{R}^{d \times k}: \|X - \Pi_{\lopt_i}(X)\|_F \leq 100 \frac{k \kappa/\sigma_{\min} \sqrt{d \log d \log (1/\miss)}}{\sqrt{\beta}}\right\}, i \in \{1,2\} \label{l:regionmf}\\
\mathcal{D}^{\mbox{ms}}_i &= \left\{X \in \mathbb{R}^{d \times k}: \|X - \Pi_{\lopt_i}(X)\|_F \leq 100 \frac{\sqrt{ d k \log \numm \log(1/\miss)} \kappa/\sigma_{\min}}{\sqrt{\beta}}\right\}, i \in \{1,2\} \label{l:regionms}\\
\mathcal{D}^{\mbox{mc}}_i &= \left\{X \in \mathbb{R}^{d \times k}: \|X - \Pi_{\lopt_i}(X)\|_F \leq 100 \frac{\sqrt{d k^3 \log d \log (1/\miss)}\kappa^3/\sigma_{\min}}{p \sqrt{\beta}}\right\}, i \in \{1,2\} \label{l:regionmc}  
\end{align} 

Our main result is that if the chain starts in  $\mathcal{D}_i $ it is likely to stay there. 

\begin{lem}\label{l:concentrationwrapper} 
The linear operators $\mathcal{A}$ of interest satisfy the following: 
\begin{enumerate}
\item For $\mathcal{A}$ corresponding to matrix factorization, let $X_0$ satisfy $\|X_0 - \Pi_{\lopt_i}(X)\|_F \leq 40 \frac{k \kappa/\sigma_{\min} \sqrt{d \log d \log (1/\miss)}}{\sqrt{\beta}}, i \in \{1,2\}$. Then, with probability $1 - \miss$, we have that $\forall t \in [0,T], \|X_t - \Pi_{\lopt_i}(X_t)\|_F \in \mathcal{D}^{\mbox{mf}}_i$. 
\item For $\mathcal{A}$ corresponding to matrix sensing, let $X_0$ satisfy $\|X_0 - \Pi_{\lopt_i}(X)\|_F \leq 40 \frac{\sqrt{ d k \log \numm \log (1/\miss)} \kappa/\sigma_{\min}}{\sqrt{\beta}}, i \in \{1,2\}$. Then, with probability $1 - \miss$, we have that $\forall t \in [0,T], \|X_t - \Pi_{\lopt_i}(X_t)\|_F \in \mathcal{D}^{\mbox{ms}}_i$.  
\item For $\mathcal{A}$ corresponding to matrix completion, let $X_0$ satisfy $\|X_0 - \Pi_{\lopt_i}(X)\|_F \leq 40 \frac{\sqrt{d k^3 \log d \log (1/\miss)}\kappa^3/\sigma_{\min}}{p \sqrt{\beta}}, i \in \{1,2\}$. Then, with probability $1 - \miss$, we have that $\forall t \in [0,T], \|X_t - \Pi_{\lopt_i}(X_t)\|_F \in \mathcal{D}^{\mbox{mc}}_i$.   
\end{enumerate}
\end{lem}

First, we will derive a stochastic differential equation for tracking the distance from the manifold:

\begin{lem} [Change of projection, worst-case noise] Let $\eta(X) = \|X - \Pi_{\lopt_i}(X)\|^2_F$. Then, if $X \in \mathbf{D}_i$ and $X$ follows the Langevin diffusion \eqref{eq:lang}, we have: \\
\begin{enumerate}
\item For $\mathcal{A}$ corresponding to matrix factorization,  then   
$$d \eta(X) \leq -\beta \damp \eta(X) dt + 500k^2 \kappa^2 d \log d dt + \sqrt{2\eta(X)} dB_t$$ 
\item For $\mathcal{A}$ corresponding to matrix sensing, 
$$d \eta(X) \leq -\beta \damp \eta(X) dt + 500 d k \kappa^2 \log \numm dt + \sqrt{2\eta(X)} dB_t$$ 
\item For $\mathcal{A}$ corresponding to matrix completion, 
$$d \eta(X) \leq -\beta \frac{p \sigma^2_{\min}}{16\kappa^4} \eta(X) dt + 500 \frac{d k^3 \kappa^2 \log d}{p} dt + \sqrt{2\eta(X)} dB_t$$ 
\end{enumerate}
\label{l:changeproject}
\end{lem} 

We can think of this expression as an ``attraction''
term $-\alpha \beta \eta(X)$, and a diffusion term $\sqrt{\eta(X)} dB_t$ along with a bias $\tilde{N} $ for appropriate $\alpha$ and $\tilde{N}$. The ``attraction'' term comes from the fact that near the manifold, $f(X)$ is locally convex so the walk is attracted towards the manifold. The diffusion term comes from the Brownian motion in the Langevin diffusion, and finally $\tilde{N}dt $ is a second-order effect that comes from the $dB^2_t = dt$ term in It\'o Lemma, and the fact that the Hessian of $\eta$ can be appropriately bounded.  

\begin{proof} 
Using the definition of Langevin diffusion (Definition \ref{def:Langevin}) and It\'o's Lemma, we can compute
\begin{equation} d \eta(X) = -\beta \langle \nabla \eta(X), \nabla f(X)\rangle dt + \frac{1}{2} \Delta \eta(X) dt + \langle \nabla \eta(X), dB_t\rangle  \label{eq:maindiffeta}\end{equation}

We will upper bound each of the terms in turn. For ease of notation, let us shorthand $\Pi_{\mathbf{E}_i}$ as $\Pi$. 

We proceed to the first term -- which in fact will be the only difference between the different $\mathcal{A}$ operators. First, we will show that
\begin{equation} \nabla \eta(X) = 2(X - \Pi(X))\label{eq:etader}\end{equation}
Note that it suffices to show $\nabla(\sqrt{\eta(X)}) = \frac{X - \Pi(X)}{\|X - \Pi(X)\|_F}$: from this we have 
$$\nabla \eta(X) = 2 \sqrt{\eta(X)}\nabla(\sqrt{\eta(X)}) = 2(X - \Pi(X))$$ 

Towards that, by Lemma~\ref{l:tubular}, we have $\frac{d}{dt} \sqrt{\eta(\gamma(t))} = -1$. On the other hand, we have $\frac{d}{dt} \sqrt{\eta(\gamma(t))} = \langle \gamma'(0), \nabla \sqrt{\eta(X)} \rangle \geq -1$, by the chain rule and using the fact that $\sqrt{\eta}$ is a 1-Lipschitz function. Thus, $\nabla(\sqrt{\eta(X)}) = \gamma'(0) = \frac{X - \Pi(X)}{\|X - \Pi(X)\|_F}$. 

From this, the bounds for each of the operators $\mathcal{A}$ follow from Lemma \ref{l:gradcorr}. Namely, we have: 
\begin{enumerate} 
\item For $\mathcal{A}$ corresponding to matrix factorization, 
$$\langle \nabla f(X), X - \Pi(X)\rangle \geq \frac{1}{16}\beta\sigma^2_{\min} \|X-\Pi(X)\|^2_F - 16 k^2 \kappa^2 d \log d$$   
\item For $\mathcal{A}$ corresponding to matrix sensing, 
$$\langle \nabla f(X), X - \Pi(X)\rangle \geq \frac{1}{16}\beta\sigma^2_{\min} \|X-\Pi(X)\|^2_F - 200 d k \kappa^2 \log \numm$$  
\item For $\mathcal{A}$ corresponding to matrix completion, 
$$\langle \nabla f(X), X - \Pi(X)\rangle \geq \beta\frac{p \sigma^2_{\min}}{16\kappa^4}  \|X-\Pi(X)\|^2_F - 400 \frac{d k^3 \kappa^2 \log d}{p}$$ 
\end{enumerate}


Moving on to the second term of \eqref{eq:maindiffeta}, by Theorem 2.2 in \cite{ambrosio1998curvature}, the eigenvalues of $\nabla^2 \eta(X)$ are bounded by 1, so $\Delta \eta(X) \leq N$.   
The proof of this is not very complicated, though calculational, and is based on the identity 
$\|\nabla \eta(X)\|^2 = 2 \eta(X)$
and repeated differentiations of it. 

Finally, for the third term of \eqref{eq:maindiffeta}, since $\|\nabla \eta(x)\| = \sqrt{2 \eta(X)}$, we have 
$\langle \nabla \eta(X) ,dB_t\rangle = \sqrt{2 \eta(X)} dB_t$. 

Putting these bounds together, we get the statement of the Lemma. 

\end{proof} 

Our goal is to prove that the above process stays near the origin for long periods of time: the difficulty is due to the the fact that the Brownian motion-like term depends on the current value of $\eta(X)$. This precludes general purpose tools for concentration of diffusions like Freidlin-Wentzell and related tools. Instead, we note that the above process is an instantiation of a Cox-Ingersoll-Ross process, which has a representation as the square of an Ornstein-Uhlenbeck process.\footnote{These processes have applications in financial mathematics. Originally, the reason for their study was the fact that normal Brownian motion is not guaranteed to be non-negative.}

\begin{lem} [Cox-Ingersoll-Ross process estimates]
\label{l:CIR}
Consider the SDE 
$$dY_t = -\gamma Y_t + \sqrt{Y_t} dB_t + \tilde{N} $$ 
for $\tilde{N} \in 2 \mathbb{N}$ and $\gamma > 0$. Then, 
$$\forall T > 0, \Pr\left[\exists t \in [0,T], \mbox{ s.t. } Y_t \geq 4\sqrt{Y_0^2 + \tilde{N} \frac{\log(1/\miss)}{\gamma}}\right] \leq \miss$$
\end{lem} 
\begin{proof} 
The stochastic differential equation describes a Cox-Ingersoll-Ross process of dimension $\frac{\tilde{N}}{2}$ (\cite{jeanblanc2010mathematical}, Chapter 6), which equals in distribution
$$\sum_{i=0}^\frac{\tilde{N}}{2} (Z_i(t))^2 $$ 
where $Z_i$ follow the Ornstein-Uhlenbeck equation $dZ_i = -\frac{\gamma}{2} Z_i dt + \frac{1}{2} dB_t$, and $Z_i(0) = \frac{Y_0}{\sqrt{\tilde{N}/2}}$. Indeed, applying It\'o's Lemma, 
we have  
$$d\left(\sum_{i=0}^\frac{N}{2} (Z_i(t))^2\right) = -\gamma \sum_{i=0}^\frac{\tilde{N}}{2} (Z_i(t))^2 dt  + \tilde{N} dt + \sum_{i=0}^{\frac{\tilde{N}}{2}} Z_i(t) dB_t$$  
Notice that $\sum_{i=0}^{\frac{\tilde{N}}{2}} Z_i(t) dB_t$ equals in distribution to $\sqrt{\sum_{i=0}^{\frac{\tilde{N}}{2}} Z^2_i(t)} dB_t$ (they are both Brownian motions, with matching variance) from which the claim follows.

This SDE has an explicit solution: namely, since each $Z_i$ is an Ornstein-Uhlenbeck process, we have
$$Z_i(t) = Z_0 e^{- \frac{\gamma}{2} t} + \frac{1}{2} \int_{0}^t e^{-\frac{\gamma}{2}(t-s)} dB_s $$

By the reflection principle, we have $\forall r > 0$,   
\begin{align*} &\Pr\left[\exists t \leq T, \frac{1}{2} \int_{0}^t e^{-\gamma/2(t-s)} dB_s \geq r \frac{2}{\sqrt{\alpha \beta}} (1-e^{-\gamma T}) \right] \\
&= 2 \Pr\left[\frac{1}{2} \int_{0}^T e^{-\gamma/2(T-s)} dB_s \geq r \frac{2}{\sqrt{\gamma}} (1-e^{-\gamma T})\right] \\ 
&\leq 2 e^{-r^2}  \end{align*}

Hence, with probability $1-\miss$, we have 
$$\sup_{t \in [0,T], i \in [\frac{\tilde{N}}{2}]} Z_i(t) \leq \frac{2}{\sqrt{\gamma}} (1-\exp(-\gamma T)) \sqrt{\log(2/\miss)}$$
and correspondingly, with probability  $1-\miss$
$$\sum_{i=0}^\frac{N}{2} (Z_i(t))^2 \leq 4\sqrt{Y_0^2 + \tilde{N} \frac{\log(1/\miss)}{\gamma}}$$ 
as we need. 

\end{proof} 

Finally, we need the following comparison theorem for diffusions with same diffusion coefficients, but different drifts, one of which dominates the other: 

\begin{lem}[Comparison theorem, \cite{ikeda1977comparison}]  
\label{l:comparison}
Let $Y_t$, $Z_t$ be two SDEs satisfying 
$$dY_t = f(Y_t) dt + \sigma(Y_t) dB_t $$
and 
$$dZ_t = g(Z_y) dt + \sigma(Z_t) dB_t $$
driven by the same Brownian motion, at least one of which has a pathwise unique solution.\footnote{Recall, an SDE $dY_t = f(Y_t) dt + \sigma(Y_t) dB_t$ has a pathwise unique solution, if for any two solutions $y(t), \bar{y}(t)$, $\Pr[y(t) = \bar{y}(t), \forall t \geq 0] = 1$. } 
Let furthermore, $f(Y_t) \leq g(Y_t)$, and $Y_0 = Z_0$. Then, with probability 1,  
$$Z_t \geq Y_t, \forall t \geq 0$$ 
\end{lem}

With these in place, the proof of Lemma~\ref{l:concentrationwrapper} follows:  
\begin{proof}[Proof of Lemma~\ref{l:concentrationwrapper}]
Consider the SDE for $\eta$ 
$$d \eta(X_t) = -\beta \langle \nabla \eta(X_t), \nabla f(X_t)\rangle dt + \frac{1}{2} \Delta \eta(X_t) dt +  \langle \nabla \eta(X_t) ,dB_t\rangle  $$ 
and the SDE 
$$dY_t = -\alpha \beta Y_t + \sqrt{Y_t} dB_t + \tilde{N} $$ 
such that $Y_0 = \eta(X_0)$ and $(\alpha, \tilde{N}) = (\damp, 500k^2 \kappa^2 d \log d)$ for matrix factorization, 
$(\alpha, \tilde{N}) = (\damp, 500 d k \kappa^2 \log \numm)$ for matrix sensing, 
$(\alpha, \tilde{N}) = (\frac{p \sigma^2_{\min}}{16\kappa^4} , 500 \frac{d k^3 \kappa^2 \log d}{p})$ for matrix completion. 

By Lemma~\ref{l:CIR}, with probability $1 - \miss$, we have
$$\forall t \in [0,T], Y_t \leq 2\sqrt{Y_0^2 + \tilde{N} \frac{\log (1/\miss)}{\beta \alpha}}$$ 
On the other hand, by Lemma~\ref{l:comparison}, conditioned on the event $\forall t \in [0,T], Y_t \leq 2\sqrt{Y_0^2 + \tilde{N} \frac{\log(1/\miss)}{\beta \alpha}}$, we have $\eta(X_t) \leq Y_t, \forall t \in [0,T]$. 

After plugging in the relevant values for $\alpha, \tilde{N}$ and $Y_0$, the statement of the lemma follows. 

\end{proof} 

%% file: functionlevel_1_rescaled.tex
\subsection{Setting up the decomposition framework}

In line with the notation in Section \ref{l:gensetup} we define the distributions $\tilde{p}^i, i \in \{1,2\}$, 
s.t.  
$$\tilde{p}^i(X)  \propto \begin{cases} 
p(X), \mbox{ if } x \in \mathcal{D}_i^{j} \\
0, \mbox{ otherwise} \end{cases} $$ 
where $j \in \{\mbox{mf},\mbox{ms},\mbox{mc}\}$, as per definitions \eqref{l:regionmf}, \eqref{l:regionms}, \eqref{l:regionmc} 
for each of the operators $\mathcal{A}$ corresponding to matrix factorization, sensing and completion respectively. Similarly, we define 
\begin{align*} 
s^{\mbox{mf}} = 100 \frac{k \kappa/\sigma_{\min} \sqrt{d \log d \log (1/\miss)}}{\sqrt{\beta}}, \hspace{0.1cm} s^{\mbox{ms}} = 100 \frac{\sqrt{ d k \log \numm \log (1/\miss)} \kappa/\sigma_{\min}}{\sqrt{\beta}}, \hspace{0.1cm} s^{\mbox{mc}} = 100 \frac{\sqrt{d k^3} \kappa^3 \log d}{p \sqrt{\beta}} 
\end{align*}  
For ease of notation, we will drop the index $j$, as it will be clear from the context which objective we are considering. 

Also, we will take $i=1$ without loss of generality, and consequently drop the index $i$ too, again, for ease of notation. The case $i=2$ is identical. 



\noindent Following Section \ref{l:gensetup}, we need to define the map $\phi_X$ -- which in fact will be the same for all $\mathcal{A}$. Let's denote by $X_0$ an arbitrary fixed matrix $X_0 \in \mathbf{E}$, so that the set of matrices in $\mathbf{E}$ have the form $X_0 U: U \in \mbox{SO}(k)$. Then, the ``norm-bounded'' normal space at $X_0$ is diffeomorphic to 
$$\mathbf{B} = \{(S,Y): S \in \mbox{Sym}^k, Y \in \Real^{d \times k}, X_0^T Y = 0,  \|Y\|^2_F + \|X_0 (X_0^T X_0)^{-1} S\|^2_F \leq s^2\}$$
This reparametrization is a very slighy deviation from our recipe and will be slightly more convenient. 
Then, we define 
\begin{equation}\phi_{X_0 U}: \mathbf{B} \to \{\Delta \in N_{X_0 U}(\mathbf{E}), \|\Delta\|_F \leq s\}, \hspace{1cm} \phi_{X_0 U} (S,Y) = X_0 (X_0^T X_0)^{-1} S U + Y U\label{eq:phimx}\end{equation}

We show that $\phi_X$ is also a diffeomorphism: 
\begin{lem}[Parametrization of $\mathcal{D}$] 
For all $U \in \mbox{SO}(k)$, the map
$$\phi_{X_0 U}: \mathbf{B} \to \{\Delta \in N_{X_0 U}(\mathbf{E}), \|\Delta\|_F \leq s\}, \hspace{1cm} \phi_{X_0 U} (S,Y) = X_0 U + X_0 S (X_0^T X_0)^{-1} U + Y U$$
is a diffeomorphism.
\label{d:canparam}
\end{lem} 
\begin{proof} 
The map is clearly differentiable, so all we need to show that it is bijective. 


To prove surjectivity of this map, note every $\Delta \in N_{X_0 U}(\mathbf{E})$ by Lemma \ref{l:tangent} can be written, for some $S' \in \mbox{Sym}^k, Y' \in \Real^{d \times k}, X_0^T Y' = 0$ as: 
\begin{align*} 
\Delta &= X_0 U \left((X_0 U)^T (X_0 U)\right)^{-1} S' + Y' \\ 
&= X_0 U \left(U^T X_0^T X_0 U\right)^{-1} S' + Y' \\
&= X_0 U U^T \left(X_0^T X_0\right)^{-1} U S' + Y' \\ 
&= X_0 \left(X_0^T X_0\right)^{-1} U S'
\end{align*}

 Denoting $S := U S' U^T$ and $Y := Y' U^T$, we have $\Delta = X_0 \left(X_0^T X_0\right)^{-1} S U + Y U$. Since $S \in \mbox{Sym}^k$ and $X_0^T Y = 0$, to show surjectivity, it suffices to show $\|\Delta\|_F \leq s$ implies $(S,Y) \in \mathbf{B}$. 

We have: 
\begin{align*} 
\|\Delta\|^2_F &= \|X_0 \left(X_0^T X_0\right)^{-1} S U + Y U\|^2_F \\ 
&\stackrel{\mathclap{\circled{1}}}{=} \|X_0 \left(X_0^T X_0\right)^{-1} S + Y\|^2_F \\
&\stackrel{\mathclap{\circled{2}}}{=} \|X_0 \left(X_0^T X_0\right)^{-1} S\|^2_F + \|Y\|^2_F \\ 
\end{align*} 
where $\circled{1}$ follows by the unitary invariance of the Frobenius norm and $\circled{2}$ follows since $X_0^T Y = 0$, which is what we wanted. 

To prove injectivity, suppose for some $(S,Y), (S',Y') \in \mathbf{B}$, we have
$$X_0 S (X_0^T X_0)^{-1} U + Y U = X_0 S' (X_0^T X_0)^{-1} U + Y' U $$
Multiplying by $U^T$ on the right on both sides, we have $Y = Y'$. Hence, $X_0 S (X_0^T X_0)^{-1}  = X_0 S' (X_0^T X_0)^{-1}$. Multiplying by $X_0^T X_0$ on the right, and by any matrix $R$, s.t. $R X_0 = I_k$, we have $S = S'$.  


The claim thus follows. 
\end{proof} 
\vspace{1cm}
Let us define $F: \mathcal{D} \to \mathbf{B}$ be the mapping s.t. $F(X) = (S,Y)$.  
Let us define by $r$ the distribution 
\begin{equation} q: \mathbf{B} \to \Real, \mbox{ s.t. } q(S,Y) \propto \int_{X \in \mathbf{M}^{(S,Y)}} e^{-\beta f(X)}  \frac{1}{\mbox{det}(\bar{dF}_X)} d\mathbf{M}^{(S,Y)}(X) \end{equation}
where $\bar{J}$ is the normal Jacobian of $F$ and by $\tilde{p}^{(S,Y)}$ the distribution
\begin{equation} \tilde{p}^{(S,Y)}: \mathbf{M}^{(S,Y)} \to \Real, \mbox{ s.t. } \tilde{p}^{(S,Y)}(X) \propto e^{-\beta f(X)} \frac{1}{\mbox{det}(\bar{dF}_X)} \label{eq:tildeprdef}\end{equation}

%% file: curvature_level.tex
\subsection{Poincar\'e constant of $\tilde{p}^{(S,Y)}$}
\label{s:levelsetp}

Following the recipe in Section \ref{l:gensetup}, we will bound the constant $C_{\mbox{level}}$. For that, we will simplify the distribution $\tilde{p}^{(S,Y)}$ significantly: namely, we will prove that it is uniform over $\mathbf{M}^{(S,Y)}$, and subsequently we will lower bound the Ricci curvature of $\mathbf{M}^{(S,Y)}$. Altogether, we will show: 

\begin{lem} For every $(S,Y) \in \mathbf{B}$, the distribution $\tilde{p}^{(S,Y)}$ has Poincar\'e constant satisfying $C_P(r) \lesssim \frac{1}{k \sigma_{\min}^2}$. Hence, $C_{\mbox{level}} \lesssim \frac{1}{k \sigma_{\min}^2}$.  
\label{l:pconstr}
\end{lem}  

First, we show that $\tilde{p}^{(S,Y)}$ is in fact uniform over $\mathbf{M}^{(S,Y)}$. We have:
 
\begin{lem}[Function value is constant on $\tilde{p}^{(S,Y)}$] Let $X,X' \in \mathbf{M}^{(S,Y)}$. Then, for all operators $\mathcal{A}$ $f(X) = f(X')$. 
\label{l:functionvalueconst}
\end{lem}
\begin{proof}
Since $X,X'\in \mathbf{M}^{(S,Y)}$, there are matrices $U,U' \in \mbox{SO}(k)$, s.t. 
$$X = X_0 U + X_0 (X_0^T X_0)^{-1} S U + Y U, \hspace{1 cm} X' = X_0 U' + X_0 (X_0^T X_0)^{-1} S U' + Y U' $$  
Hence, we have that $X' = X(U^T U')$, and also $U'' = U^T U' \in \mbox{SO}(k)$. Since $f(X) = \|\mathcal{A}(XX^T) - b\|^2_2$, and $XX^T = (XU'')(XU'')^T$, we have $f(X) = f(X')$, as we wanted.    
\end{proof} 

Subsequently, we show that the $\mbox{det}(\bar{dF}_X)$ is constant over $\mathbf{M}^{(S,Y)}$ -- in fact
it's a constant over all of $\mathcal{D}$:  

\begin{lem}[Normal Jacobian is constant] The function $\mbox{det}(\bar{dF}_X)$ is constant over $\mathcal{D}$.   
\label{eq:constjacob}
\end{lem} 
\begin{proof}

We will perform the calculation using the diffeomorphism from Lemma \ref{d:canparam}, along with the standard parametrization of the symmetric and orthogonal matrices. 

Let us denote: 
\begin{align*} A^{ij}&:= \frac{1}{\sqrt{2}}\left(e_i e_j^T - e_j e_i^T\right), 1 \leq i < j \leq k\\
S^{ij} &= \frac{1}{\sqrt{2}}\left(e_i e_j^T + e_j e_i^T\right), 1 \leq i < j \leq k, \hspace{0.5cm} S^{ii} = e_i e_i^T, 1 \leq i \leq k\\  
E^{ij} &= e_i e_j^T, 1 \leq i \leq d-k, 1 \leq j \leq k\\
\end{align*} 

Note that the $A$ matrices form a basis of the skew-symmetric matrices in $\Real^{k \times k}$, the $S$ matrices of the symmetric matrices in $\Real^{k \times k}$ and $E^{ij}$ of the matrices $\Real^{(d-k) \times k}$. 

By Lemma \ref{l:tangent}, the tangent space at the identity matrix for $\mbox{SO}(k)$ is the set of skew-symmetric matrices. Since the exponential map for $\mbox{SO}(k)$ is the usual matrix exponential, and is a bijection between $T_{I}(\mbox{SO}(k))$ and $\mbox{SO}(k)$, we can parametrize $\mbox{SO}(k)$ as 
\begin{equation} U: \Real^{k(k-1)/2} \to \Real^{k \times k},  \hspace{1cm} U(\mu) = e^{\sum_{1 \leq i < j \leq k} \mu_{i,j} A^{i,j}} \label{d:canonU}
 \end{equation} 
We parametrize $\mbox{Sym}^k$ the obvious way:  
\begin{equation} S: \Real^{k(k+1)/2} \to \Real^{k \times k}, \hspace{1cm} S(\nu) = \sum_{i \leq j} \nu_{ij} S^{ij} \label{d:canonS} \end{equation} 
Finally, denoting $Y_0 \in \Real^{d \times (d-k)}$ any matrix s.t. $Y_0^T Y_0 = I$, and $X_0^T Y_0 = 0$ (i.e. a the columns form a basis of the orthogonal space to $X_0$), we can parametrize the set of $Y \in \mathbf{B}$ as 
\begin{equation} Y: \Real^{(d-k)k} \to \Real^{d \times k}, \hspace{1cm} Y(\lambda) = Y_0 \sum_{1 \leq i \leq d-k, 1 \leq j \leq k} \lambda_{ij} E^{ij} \label{d:canonY} \end{equation} 

Since composing the above parametrizations with $\phi_X$ results in a diffeomorphism, we can form a basis of $\Real^{dk}$ by taking the partial derivatives with respect to the variables $\mu, \nu, \lambda$. We will calculate these explicitly -- in particular, we will vectorize all of the matrices, heavily using Lemma \ref{l:vectorizeop}.     

We start with the derivatives in $\mu$. We have:  
\begin{align} 
\frac{\partial \mbox{vec}(X)}{\partial \mu_{i,j}} &= \frac{\partial \mbox{vec}(X_0 U + X_0 (X_0^T X_0)^{-1}  S U + YU)}{\partial \mu_{i,j}} \nonumber\\ 
&= \frac{ \partial \left(I_k \otimes (X_0 + X_0 (X_0^T X_0)^{-1} S + Y)\right) \mbox{vec}(U)}{\partial \mu_{i,j}} \label{eq:vectorx}\\ 
&= \left(I_k \otimes (X_0 + X_0 (X_0^T X_0)^{-1} S + Y)\right) \frac{ \partial \mbox{vec}(U)}{\partial \mu_{i,j}} \nonumber\\
&= \left(I_k \otimes (X_0 + X_0 (X_0^T X_0)^{-1} S + Y)\right) \mbox{vec}(U A^{ij}) \label{eq:dircomp} \\ 
&= \left(I_k \otimes (X_0 + X_0 (X_0^T X_0)^{-1}  S + Y)\right) (I_k \otimes U) \mbox{vec}(A^{ij}) \label{eq:vectorxdeux}\\ 
&= \left(I_k \otimes (X_0 + X_0 (X_0^T X_0)^{-1} S + Y)U\right) \mbox{vec}(A^{ij}) \label{eq:xderu}
\end{align} 
where \eqref{eq:vectorx} and \eqref{eq:vectorxdeux} follow from Lemma \ref{l:vectorizeop}, \eqref{eq:dircomp} is by direct computation, and \eqref{eq:xderu} follows from Lemma \ref{l:kronalgebra}.

We proceed to the derivatives in $\nu$ next: 
\begin{align} 
\frac{\partial \mbox{vec}(X)}{\partial \nu_{i,j}} &= \frac{\partial \mbox{vec}(X_0 (X_0^T X_0)^{-1} S U)}{\partial \nu_{i,j}}\nonumber\\ 
&= \frac{ \partial (U^T \otimes \left(X_0 (X_0^T X_0)^{-1} \right)) \mbox{vec}(S)}{\partial \nu_{i,j}} \label{eq:vectorn}\\ 
&= (U^T \otimes \left(X_0 (X_0^T X_0)^{-1}\right)) \mbox{vec}(S^{ij}) \label{eq:xders}
\end{align} 
where \eqref{eq:vectorn} follows from Lemma \ref{l:vectorizeop} and \eqref{eq:xders} by direct computation. 

Finally, for $\lambda$ derivatives, we have 
\begin{align} 
\frac{\partial \mbox{vec}(X)}{\partial \lambda_{i,j}} &= \frac{\partial \mbox{vec}\left(Y_0 \sum_{1 \leq i \leq d-k, 1 \leq j \leq k} \lambda_{ij} E^{ij} U\right)}{\partial \lambda_{i,j}} \nonumber\\ 
&= \frac{(U^T \otimes Y_0) \partial \mbox{vec}\left(\sum_{1 \leq i \leq d-k, 1 \leq j \leq k} \lambda_{ij} E^{ij}\right)}{\partial \lambda_{i,j}} \nonumber\\ 
&= (U^T \otimes Y_0) \mbox{vec}(E^{ij}) \label{eq:yders}
\end{align}

Furthermore, we claim that the kernel of $dF_X$ is spanned by the set of vectors $\left\{\frac{\partial \mbox{vec}(X)}{\partial \mu_{i,j}}|_{1 \leq i < j \leq k}\right\}$. 

Indeed, for a curve parametrized as $\phi(t): (-1,1) \to \Real^{dk}$, by the chain rule, we have 
$$\frac{\partial}{\partial t} F(\phi(t))|_{t=0} = dF_{X}\left(\phi'(0)\right)$$  
For $1 \leq i < j \leq k$, consider the curve $\phi(t) = X((\mu, \nu, \lambda) + t A^{ij})$. By the definition of $F$, since $S, Y$ do not change along $\phi$, we have $\frac{\partial}{\partial t} F(\phi(t))|_{t=0} = 0$. On the other hand, $\phi'(0)$ is exactly the partial derivative with respect to $\mu_{i,j}$, which implies that 
the vectors $\{\frac{\partial \mbox{vec}(X)}{\partial \mu_{i,j}}\}$ lie in the kernel of $F$. 

On the other hand, for the curves $\phi(t) = X((\mu, \nu, \lambda) + t S^{ij})$ and $\phi(t) = X((\mu, \nu, \lambda) + tE^{ij})$, $\frac{\partial}{\partial t} F(\phi(t))|_{t=0}$ is not zero, so the corresponding partial derivative vectors do not belong in the kernel of $dF_{\phi(0)}$ . 

Hence, $\mbox{ker}(dF_X)^{\perp}$ is spanned by 
$V_1 = \mbox{span}\left(\frac{\partial \mbox{vec}(X)}{\partial \nu_{i,j}}|_{1 \leq i \leq j \leq k}\right)$ and $V_2 = \mbox{span}\left(\frac{\partial \mbox{vec}(X)}{\partial \lambda_{i,j}}|_{1 \leq i \leq d-k, 1 \leq j \leq k}\right)$.
To calculate the determinant of $\bar{dF}_X$, we first show that:\\ 
(i) $V_1 \perp V_2$. \\
(ii) $\bar{dF}_X(V_1) \perp \bar{dF}_X(V_2)$.\\
From (i) and (ii), we get 
\begin{equation}\mbox{det}(\bar{dF}_X) = \mbox{det}\left((\bar{dF}_X)_{V_1}\right) \mbox{det}\left((\bar{dF}_X)_{V_2}\right)\label{eq:proddet}\end{equation}
where $(\bar{dF}_X)_{V_1}$ and $(\bar{dF}_X)_{V_2}$ denote the restrictions of $\bar{dF}_X$ to the subspace $V_1$, $V_2$ respectively. \\
To prove (i), $\forall 1\leq i \leq j \leq k, 1 \leq i' \leq d-k, 1 \leq j' \leq k$, we have 
$$\frac{\partial \mbox{vec}(X)}{\partial \lambda_{i,j}}^T \frac{\partial \mbox{vec}(X)}{\partial \nu_{i',j'}} = \mbox{vec}(E^{ij})^T  (U \otimes Y^T_0) (U^T \otimes X_0 (X_0^T X_0)^{-1}) \mbox{vec}(S^{i'j'}) = 0$$ 
since $Y^T_0 X_0 = 0$ -- which shows (i).  

To prove (ii), we will compute the images of $V_1$ and $V_2$ via taking appropriate curves.  
Consider the derivative with respect to $\nu_{i,j}$. Taking the curve $\phi(t) = X((\mu, \nu, \lambda) + t S^{ij})$, we have, $\forall 1 \leq i \leq j \leq k$:
\begin{align} 
\frac{\partial}{\partial t} F(\phi(t))|_{t=0} &= \lim_{t \to 0} \frac{F(\phi(t)) - F(\phi(0))}{t}  \nonumber\\
&= \frac{\left(S(\nu + t S^{ij}), Y(\lambda)\right) - \left(S(\nu),Y(\lambda)\right)}{t} \nonumber\\ 
&= \frac{\left(S(t S^{ij}), 0\right)}{t} \\ 
&= (S^{ij}, 0) \label{eq:derss}
\end{align} 
Vectorizing \eqref{eq:derss} (i.e. picking the standard basis to write it in), we have 
\begin{equation}
dF_{X}\left(\frac{\partial \mbox{vec}(X)}{\partial \nu_{i,j}}\right) = \begin{pmatrix} \mbox{vec}(S^{ij}) \\ 0 \end{pmatrix}\label{eq:derss1}
\end{equation} 

Similarly, taking the curve $\phi(t) = X((\mu, \nu, \lambda) + t E^{ij})$, we have, $\forall 1 \leq i \leq d-k, 1 \leq j \leq k$:  
\begin{align} 
\frac{\partial}{\partial t} F(\phi(t))|_{t=0} &= \lim_{t \to 0} \frac{F(\phi(t)) - F(\phi(0))}{t}  \nonumber\\
&= \frac{(0, t Y_0 E^{ij})}{t} \nonumber\\ 
&= (0, Y_0 E^{ij}) \label{eq:dersyy}
\end{align} 
Vectorizing again, we have
\begin{equation}
dF_{X}\left(\frac{\partial \mbox{vec}(X)}{\partial \lambda_{i,j}}\right) = \begin{pmatrix} 0 \\ \mbox{vec}(Y_0 E^{ij}) \end{pmatrix}\label{eq:dersyy1}
\end{equation} 
From \eqref{eq:derss1} and \eqref{eq:dersyy1}, (ii) immediately follows. 

In light of \eqref{eq:proddet}, it suffices to calculate $\mbox{det}\left((\bar{dF}_X)_{V_1}\right)$ and $\mbox{det}\left((\bar{dF}_X)_{V_2}\right)$.  

Proceeding to $\mbox{det}\left((\bar{dF}_X)_{V_1}\right)$, because of \eqref{eq:xders} and \eqref{eq:derss1} we can view the map $(\bar{dF}_X)_{V_1}$ as sending the vectors $(U^T \otimes X_0) \mbox{vec}(S^{ij})$ to the vectors 
$\mbox{vec}(S^{ij})$. The determinant of this map is 
$$\sqrt{\mbox{det}\left((U \otimes (X_0^T X_0)^{-1} X^T_0)(U^T \otimes X_0 (X_0^T X_0)^{-1})\right)} = \sqrt{\mbox{det}\left(I_k \otimes (X^T_0 X_0)^{-1}\right)}$$
which is a constant. 

Proceeding to $\mbox{det}\left((\bar{dF}_X)_{V_2}\right)$, in light of \eqref{eq:yders} and \eqref{eq:dersyy1} we can view $dF_X$ as sending the vectors $\{(U^T \otimes Y_0) \mbox{vec}(E^{ij})\}$ to the vectors 
$\{Y_0 E^{ij}\}$. We will show that $\mbox{det}\left((\bar{dF}_X)_{V_2}\right) = 1$, by showing both sets of vectors are orthonormal. 

Indeed, 
\begin{align*} 
 \mbox{vec}(E^{ij})^T(U \otimes Y^T_0) (U^T \otimes Y_0) \mbox{vec}(E^{ij}) &= \mbox{vec}(E^{ij})^T(UU^T \otimes Y^T_0 Y_0) \mbox{vec}(E^{ij}) \\ 
&= \mbox{vec}(E^{ij})^T \mbox{vec}(E^{ij}) \\ 
&= 1
\end{align*} 
as well as $\forall (i,j) \neq (i',j')$, 
\begin{align*} 
 \mbox{vec}(E^{ij})^T(U \otimes Y^T_0) (U^T \otimes Y_0) \mbox{vec}(E^{i'j'}) &= \mbox{vec}(E^{ij})^T(UU^T \otimes Y^T_0 Y_0) \mbox{vec}(E^{i'j'}) \\ 
&= \mbox{vec}(E^{ij})^T \mbox{vec}(E^{i'j'}) \\ 
&= 0
\end{align*} 
Similarly, the vectors $\{Y_0 E^{ij}\}$ are orthonormal. Hence, the determinant of this map is 1, which concludes the proof of the lemma. 
\end{proof}  

Given Lemmas \ref{eq:constjacob} and \ref{l:functionvalueconst}, we in fact have that $\tilde{p}^{(S,Y)}$ is the uniform distribution over $\mathbf{N}_{S,Y}$.  

To get a handle on the Poincar\'e constant of this distribution, we will first get a handle on the Poincar\'e constant of the manifolds $\mathbf{M}^{(S,Y)}$, though with a more convenient (left-invariant) metric. This allows us to use the powerful theory of curvatures of Lie groups from Theorem \ref{l:curvlie}. 

\begin{lem}[Ricci curvature with left-invariant metric] 
Let 
$\mathbf{M}^{(S,Y)} = \{X: X = X_0 U + X_0 (X_0^T X_0)^{-1} S U + Y U, U \in \mbox{SO}(k)\}$. Then, \\
(1) $T_X(\mathbf{M}^{(S,Y)}) = \{XR: R \in R \in \skewm^{k \times k}\}$. \\
(2) If we equip $\mathbf{M}^{(S,Y)}$ with the metric
$$\forall XR, XS \in T_X(\mathbf{M}^{(S,Y)}): \langle XR, XS\rangle_{\gamma} = \mbox{Tr}(R^T S)$$ 
the Ricci curvature of $\mathbf{M}^{(S,Y)}$ with this metric satisfies 
$$\forall X \in \mathbf{M}^{(S,Y)}, XU \in T_X(\mathbf{M}^{(S,Y)}), \|XU\|_{\gamma}=1 : \mbox{Ric}(XU) = \frac{k-2}{4}$$    	
\label{l:limetric}
\end{lem} 
\begin{proof} 
Let us denote by $X^* := X_0 + X_0 (X_0^T X_0)^{-1} S + Y$.  
For (1), notice that the manifold $\mathbf{M}^{(S,Y)}$ can be equivalently written as 
$$\mathbf{M}^{(S,Y)} = \{X: X = X^* U\}$$ The claim then follows by Lemma \ref{l:tangent}. So, we proceed to (2).\\
First, we claim $\gamma$ is a left-invariant metric. Towards checking Definition \ref{d:lint}, consider the map $L_A: \mathbf{M}^{(S,Y)} \to \mathbf{M}^{(S,Y)}$ s.t. $L_A(X^* U) = (X^* U)(AU), \forall U \in \mbox{SO}(k)$. Equivalently, denoting $\tilde{U} = X^* U$, 
and $(X^*)^{-1} \in \Real^{k \times d}$ any matrix, s.t. $(X^*)^{-1} X^* = I_k$, we have  
$$L_A(\tilde{U}) = X^* A (X^*)^{-1} \tilde{U}$$ 
As $L_A$ is linear, we have $(L_A)_* = X^* A (X^*)^{-1}$. Hence, for $X^* U \in \mathbf{M}^{(S,Y)}$ and $R,S$ skew-symmetric, we have
\begin{align*} \left\langle (L_A)_*(X^* U R), (L_A)_*(X^* U S) \right\rangle_{\gamma(X^* A U)} &= \left\langle (X^* A U)R, (X^*A U) S \right\rangle_{\gamma(X^*A U)} \\
&= \left\langle R, S\right\rangle \\ 
&= \left\langle X^* U R, X^* U S \right\rangle_{\gamma(X^* U)}  
\end{align*}  
which by Definition \ref{d:lint} shows that $\gamma$ is left-invariant. 

Consider the homomorphism:  
$$\phi: \mathbf{M}^{(S,Y)} \to \mbox{SO}(k), \mbox{ s.t. } \phi(X^* U) = U $$
The pushforward $\phi_*: T \mathbf{M}^{(S,Y)} \to T \mbox{SO}(k)$ can be written as $\phi_*(E) = (X^*)^{-1} E$, for any matrix $(X^*)^{-1} \in \Real^{k \times d}$ s.t. $(X^*)^{-1} X^* = I_k$. 
Hence, by Definition \ref{d:lbkt} the Lie bracket satisfies 
$$\left[X^*U, X^* V\right]_{\mathbf{M}^{(S,Y)}} = X^*(UV - VU)$$ 
Let $\{e_i\}_{i=1}^m$ be a basis of $T_{I_k}\left(\mbox{SO}(k)\right)$. Then, $\{X^*e_i\}_{i=1}^m$ forms an orthonormal basis of $T_{X^*}(\mathbf{M}^{(S,Y)})$.  
By Lemma \ref{l:curvlie}, we have 

\begin{align*}\mbox{Ric}(X^* U) &= \left\langle \frac{1}{4}[[X^*U,X^*e_i]_{\mathbf{N}_{X^*}},X^*e_i]_{\mathbf{N}_{X^*}}, X^* U\right\rangle_{\gamma} \\
&=\left\langle \frac{1}{4}X^*[[U,e_i]_{\mbox{SO}(k)},e_i]_{\mbox{SO}(k)}, X^* U \right\rangle_{\gamma} \\
&=\left\langle \frac{k-2}{4}X^*U, X^* U \right\rangle_{\gamma} \\
&= \frac{k-2}{4}
\end{align*}
as we needed. 
\end{proof}

From this estimate, we will infer a Poincar\'e inequality on $\mathbf{N}_{S,Y}$ with the standard Euclidean metric, thus prove Lemma \ref{l:pconstr} 

\begin{proof} [Proof of Lemma \ref{l:pconstr}]
By Lemma \ref{d:bcp}, uniform distribution over the manifold $\mathbf{M}^{(S,Y)}$ with the metric $\gamma$ from Lemma \ref{l:limetric} satisfies a Poincar\'e inequality, i.e.: 
\begin{equation} \mbox{Var}_{\tilde{p}^{(S,Y)}_\gamma}(g) \lesssim \frac{1}{k} \E_{\tilde{p}^{(S,Y)}_\gamma}(\|\nabla g\|_\gamma^2) \label{eq:poincarescaled}\end{equation}
where $\tilde{p}^{(S,Y)}_{\gamma}$ is the uniform distribution on $\mathbf{M}_{(S,Y)}$ with respect to the volume form of the metric $\gamma$. 

We will infer from this a Poincar\'e inequality with the Euclidean metric. 

As we did in the proof of Lemma \ref{l:limetric} we denote $X^* := X_0 + X_0 (X_0^T X_0)^{-1} S + Y$ and note that the manifold $\mathbf{M}^{(S,Y)}$ can be equivalently written as 
$$\mathbf{M}^{(S,Y)} = \{X: X = X^* U\}$$
Towards that, we first prove the volume form on $\mathbf{M}_{(S,Y)}$ with the metric $\gamma$ is a constant multiple of the volume form with the Euclidean metric. 
Consider the parametrization of $\mathbf{M}_{(S,Y)}$ s.t. $\phi(\mu) = X^* e^{\sum_{i<j} \mu_{i,j} A^{ij}}$. 
Then, $\frac{\partial \phi}{\partial \mu_{i,j}}|_{\mu} = X A^{ij}$, where we denote $X := X^* e^{\sum_{i<j} \mu_{i,j} A^{ij}}$.     
Towards using this parametrization in Definition \ref{d:localvol}, let us denote by $\tilde{J}$ and $\tilde{J}^{\gamma}$ the corresponding Gram matrices of inner products in the respective metrics. Namely, we have     
$$ \tilde{J}_{(i,j),(i',j')} = \mbox{Tr}\left((A^{ij})^T X^T X A^{i'j'}\right) $$
and 
$$ \tilde{J}^{\gamma}_{(i,j),(i',j')} = \mbox{Tr}\left((A^{ij})^T A^{i'j'}\right) $$
If we can show the determinants of these matrices are constant multiples of each other, the claim would follow. 
Clearly, $\tilde{J}^{\gamma} = I$, so $\mbox{det}(\tilde{J}^{\gamma}) = 1$. We will show the determinant of $\tilde{J}$ is independent of $X$. 

If $w \in  \Real^{k(k-1)/2}$, we have $w^T \tilde{J} w = \mbox{Tr}(W^T X^T X W)$ 
where $W = \sum_{i < j} w_{i,j} A^{ij}$ (indexing the coordinates of $w$ with the pairs $1 \leq  i<j \leq k$).  
As the determinant of $\tilde{J}$ is the product of the eigenvalues of the quadratic form $Q_X: \skewm_k \to \Real$, s.t. $Q_X(W) = \mbox{Tr}(W^T X^T X W)$, it suffices to show that they are constant for all $X \in \mathbf{M}^{(S,Y)}$. To show this, note that by the similarity-invariance of trace, we have, 
$$\forall U \in \so, \mbox{Tr}\left(W^T X^T X W\right) = \mbox{Tr}\left(U^T W^T U U^T X^T X U U^T W U\right)$$ 
Also, if $W$ is skew-symmetric, so is $U^T W U$, as $(U^T W U)^T = U^T (-W) U$. Hence if $W$ is an eigenvector of $Q_X$, $U^T W^T U$ is an eigenvector of $Q_{X U}$  
with the same eigenvalue. Thus, the eigenvalues of $Q_X$ are constant on $\mathbf{M}^{(S,Y)}$, which proves the determinant of $\tilde{J}$ is independent of $X$, as we need.

As a consequence of the volume forms being constant multiples of each other, scaling both sides of \eqref{eq:poincarescaled} by an appropriate constant we have
\begin{equation} \mbox{Var}_{\tilde{p}^{(S,Y)}}(g) \lesssim \frac{1}{k} \E_{\tilde{p}^{(S,Y)}}(\|\nabla g\|_{\gamma}^2) \label{eq:downscaledpoincare}\end{equation} 
Finally, we massage the RHS of \eqref{eq:downscaledpoincare} to get a Poincar\'e inequality with the Euclidean metric.   

By the definition of a gradient (Definition \ref{d:jac}), we have 
\begin{align*} \|\nabla g(X)\|_{\gamma} &= \sup_{v \in T_X(\mathbf{M}^{(S,Y)})} \frac{|(g \circ \phi)'(0)|}{\|v\|_{\gamma}} \\ 
\end{align*}
where $\phi:(-1,1) \to \mathbf{M}$ is a curve with $\phi(0) = X$ and $\phi'(0) = v$.   
Hence, we will show that: 
\begin{equation} \|\nabla g(X)\|_{\gamma} \geq \frac{1}{\sigma_{\min^2}(X)}\|\nabla g(X)\| \end{equation} 
by showing that 
\begin{align*}\forall X \in \mathbf{M}^{(S,Y)}, \tilde{R} \in T_X \mathbf{M}^{(S,Y)}: \|\mbox{vec}(\tilde{R})\|_{\gamma} &\leq \frac{1}{\sigma^2_{\min}(X)} \|\mbox{vec}(\tilde{R})\|_2 \end{align*}

We have: 
\begin{align*}\|\mbox{vec}(\tilde{R})\|_{\gamma} &= \mbox{vec}(\tilde{R})^T \left(I_{d} \otimes ((X^{-1})^T X^{-1})\right) \mbox{vec} (\tilde{R}) \\
&\leq \sigma_{\max}\left(I_{d} \otimes ((X^{-1})^T X^{-1})\right) \|\mbox{vec}(\tilde{R})\|_2 \end{align*}
where $X^{-1} \in \Real^{k \times d}$ is any matrix s.t. $X^{-1} X = I_k$, and the first equality 
follows by writing the inner product $\gamma$ in its vectorized form.  
Since the eigenvalues of $A \otimes B$ are the product of the eigenvalues of $A$ and $B$, we have 
$$\sigma_{\max}\left(I_{d} \otimes ((X^{-1})^T X^{-1})\right) \leq \sigma_{\max}\left(((X^{-1})^T X^{-1})\right) = \frac{1}{\sigma^2_{\min}(X)} $$
Plugging this back in \eqref{eq:downscaledpoincare}, we have  
\begin{equation} \mbox{Var}_{p^{X^*}}(g) \lesssim \frac{1}{k} \frac{1}{\sigma^2_{\min}(X)}\E_{p^{X^*}}(\|\nabla g\|^2) \label{eq:poincareregular}\end{equation}
Finally, 
\begin{align*} 
\sigma_{\min}(X) &= \sigma_{\min}(X^*) \\
&= \sigma_{\min}(X_0 + X_0 (X_0^T X_0)^{-1} S + Y) \\
&\geq \sigma_{\min}(X_0) - \sigma_{\max}(X_0 (X_0^T X_0)^{-1} S + Y) \\ 
&\geq \frac{\sigma_{\min}(X_0)}{2}
\end{align*}  
where the last inequality follows by the bounds $s^{\mbox{mf}}, s^{\mbox{ms}}, s^{\mbox{mc}}$.  

The Lemma thus follows. 

\end{proof}

\subsection{Poincar\'e constant of $q$}
\label{s:poincarr}

Finally, we characterize the Poincar\'e constant of $q$. 
\begin{lem} The distribution $q: \mathbf{B} \to \Real$ has Poincar\'e constant satisfying 
$C_P(q) \lesssim 1 $. 
\label{l:rconst}
\end{lem}  
\begin{proof} 

We will use Lemma \ref{l:constrained}. Towards that, we will show that the set $\mathbf{B}$ is in fact convex: for any pair $(S_1, Y_1), (S_2, Y_2) \in \mathbf{B}$, we have $(1-\theta)(S_1, Y_1) + \theta (S_2, Y_2) \in \mathbf{B}$.  
This follows by the convexity of the squared 2-norm: namely, we have 
\begin{align*}
&\|(1-\theta) X_0 (X_0^T X_0)^{-1}S_1 + \theta X_0 (X_0^T X_0)^{-1} S_2)\|^2_F + \|(1-\theta) Y_1 + \theta Y_2\|^2_F \\
&\leq 
(1-\theta)\left(\|X_0 (X_0^T X_0)^{-1} S_1\|^2_F + \|Y_1\|^2_F\right) + \theta\left(\|X_0 (X_0^T X_0)^{-1} S_2\|^2_F + \|Y_2\|^2_F\right) \\ 
&\leq s^2
\end{align*}

Next, towards using the Holley-Stroock perturbation bound (Lemma \ref{l:holleystroock}), we will show that the function $\beta^2 f$ is close to being convex as a function of $S,Y$.  We proceed essentially by Taylor expanding.
Let us denote $\Delta:= \frac{X - \Pi(X)}{\|X - \Pi(X)\|_F}$, and $\tilde{s}:= \|X - \Pi(X)\|_F$. 
 We have: 
\begin{align} 
&\|\mathcal{A}(XX^T) - b\|^2_F \nonumber\\
&= \|\mathcal{A}\left(\left(\Pi(X) + \tilde{s}\Delta\right)\left(\Pi(X) + \tilde{s} \Delta\right)^T\right) - \mathcal{A}\left(\Pi(X) \Pi(X)^T\right) - n\|^2_F \nonumber\\ 
&\leq \tilde{s}^2 \|\mathcal{A}\left(\Pi(X) \Delta^T + \Delta \Pi(X)^T\right)\|^2_F + 2\tilde{s} \langle \mathcal{A}\left(\Pi(X) \Delta^T + \Delta \Pi(X)^T\right), n \rangle\nonumber\\ 
&+2\tilde{s}^2 \langle \Delta \Delta^T, n \rangle \nonumber\\ 
&+2\tilde{s}^3 \left(\langle \mathcal{A}\left(\Pi(X) \Delta^T + \Delta \Pi(X)^T\right), \mathcal{A}(\Delta \Delta^T)\rangle +\|\mathcal{A}(\Delta \Delta^T)\|^2_2\right)\label{eq:exprr}\\  
\end{align} 
For all linear operators $\mathcal{A}$ in question, we will be able to bound the terms 
$$\left(\langle \mathcal{A}\left(\Pi(X) \Delta^T + \Delta \Pi(X)^T\right), \mathcal{A}(\Delta \Delta^T)\rangle +\|\mathcal{A}(\Delta \Delta^T)\|^2_2\right)$$
and $\langle \mathcal{A}(\Delta \Delta^T), n \rangle$. 

Proceeding to the former, we have: 
\begin{enumerate} 
\item For $\mathcal{A}$ corresponding to matrix factorization, we have 
\begin{align*} 
\langle \mathcal{A}\left(\Pi(X) \Delta^T + \Delta \Pi(X)^T\right), \mathcal{A}(\Delta \Delta^T)\rangle &\leq \|\Pi(X) \Delta^T + \Delta \Pi(X)^T\|_F \\ 
&\leq 2 \|\Pi(X)\|_F \\ 
&\leq 2 k \sigma_{\max} 
\end{align*} 
and $\|\mathcal{A}(\Delta \Delta^T)\|^2_2 \leq 1$, so 
\begin{equation} \left(\langle \mathcal{A}\left(\Pi(X) \Delta^T + \Delta \Pi(X)^T\right), \mathcal{A}(\Delta \Delta^T)\rangle +\|\mathcal{A}(\Delta \Delta^T)\|^2_2\right) \leq 3 k \sigma_{\max} \label{eq:hotqmf}\end{equation}
\item For $\mathcal{A}$ corresponding to matrix sensing, since $\Pi(X) \Delta^T + \Delta \Pi(X)^T$ is of rank $k$, as is $\Delta \Delta^T$, by the $(k,\frac{1}{10})$-RIP condition, we have    
\begin{equation} \left(\langle \mathcal{A}\left(\Pi(X) \Delta^T + \Delta \Pi(X)^T\right), \mathcal{A}(\Delta \Delta^T)\rangle +\|\mathcal{A}(\Delta \Delta^T)\|^2_2\right) \leq  6k\sigma_{\max}\label{eq:hotqms}\end{equation}
\item For $\mathcal{A}$ corresponding to matrix completion, we have 
\begin{align} &\left(\langle \mathcal{A}\left(\Pi(X) \Delta^T + \Delta \Pi(X)^T\right), \mathcal{A}(\Delta \Delta^T)\rangle +\|\mathcal{A}(\Delta \Delta^T)\|^2_2\right) \nonumber\\
&\leq \|\mathcal{A}\left(\Pi(X) \Delta^T + \Delta \Pi(X)^T\right)\|_2 \|\mathcal{A}(\Delta \Delta^T)\|_2 + \|\mathcal{A}(\Delta \Delta^T)\|^2_2 \nonumber\\ 
&\leq \|\Pi(X) \Delta^T + \Delta \Pi(X)^T\|_2 \|\Delta \Delta^T\|_2 + \|\Delta \Delta^T\|^2_2 \label{eq:preend} \\
&\leq 3 k \sigma_{\max} \label{eq:hotqmc}
\end{align}
where \eqref{eq:preend} follows since applying $P_{\Omega}$ can only reduce the Frobenius norm. 
\end{enumerate} 

Proceeding to the latter term, 
\begin{enumerate} 
\item For $\mathcal{A}$ corresponding to matrix factorization, we have, with high probability 
\begin{align} 
\langle \mathcal{A}(\Delta \Delta^T), n \rangle &= \langle \Delta \Delta, M - M^* \rangle \nonumber\\
&\stackrel{\mathclap{\circled{1}}}{\leq} \|\Delta \Delta\|_F \|M-M^*\|_2 \nonumber\\ 
&\stackrel{\mathclap{\circled{2}}}{\leq} \frac{\sqrt{d} \log d}{\sqrt{\beta}} \label{eq:mfnoiser}
\end{align}
where $\circled{1}$ follows from $\|AB\|_F \leq \|A\|_F \|B\|_2$, $\circled{2}$ since $M-M^*$ is a matrix with Gaussian entries.
\item For $\mathcal{A}$ corresponding to matrix sensing, by Lemma 34 in \citep{ge2017no}, since $\Delta\Delta^T$ is of rank $k$, we have  
\begin{align} 
\langle \mathcal{A}(\Delta \Delta^T), n \rangle \leq \frac{10}{\sqrt{\beta}} \sqrt{d k \log \numm} \label{eq:msnoiser}
\end{align} 
\item For $\mathcal{A}$ corresponding to matrix completion, we have by \eqref{eq:highprobdevmc}, 
\begin{align} 
\langle \mathcal{A}(\Delta \Delta^T), n \rangle \leq \frac{20}{\sqrt{\beta}}\sqrt{d \log d} \label{eq:mcnoiser}
\end{align}
\end{enumerate} 

We put together these bounds. For matrix factorization, plugging \eqref{eq:hotqmf} and \eqref{eq:mfnoiser}
in \eqref{eq:exprr}, we have 
\begin{align*} 
&\beta\left( f(X) - \left(\tilde{s}^2 \|\mathcal{A}\left(\Pi(X) \Delta^T + \Delta \Pi(X)^T\right)\|^2_F + \tilde{s} \langle \mathcal{A}\left(\Pi(X) \Delta^T + \Delta \Pi(X)^T + \Delta \Delta^T\right), n \rangle\right)\right)  \\ 
&\lesssim \beta \left( \tilde{s}^3 k \sigma_{\max} + \tilde{s}^2 \frac{\sqrt{d}}{\sqrt{\beta}}\right) \\ 
&\lesssim \sqrt{\frac{\lbdmf}{\beta}}\\
&\lesssim 1
\end{align*} 
where the last inequality follows since $\beta \gtrsim \lbdmf$. 
Similarly, 
$$\beta\left( f(X) - \left(\tilde{s}^2 \|\mathcal{A}\left(\Pi(X) \Delta^T + \Delta \Pi(X)^T\right)\|^2_F + \tilde{s} \langle \mathcal{A}\left(\Pi(X) \Delta^T + \Delta \Pi(X)^T + \Delta \Delta^T\right), n \rangle\right)\right) \gtrsim 1$$ 

Analogously, for matrix sensing, from \eqref{eq:hotqms} and \eqref{eq:msnoiser} we have
\begin{align*} 
&\beta\left| f(X) - \left(\tilde{s}^2 \|\mathcal{A}\left(\Pi(X) \Delta^T + \Delta \Pi(X)^T\right)\|^2_F + \tilde{s} \langle \mathcal{A}\left(\Pi(X) \Delta^T + \Delta \Pi(X)^T + \Delta \Delta^T\right), n \rangle\right)\right|  \\ 
&\lesssim \sqrt{\frac{\lbdms}{\beta}} \\
&\lesssim 1
\end{align*} 
where the last inequality follows since $\beta \gtrsim \lbdms$. 

Finally, for matrix completion from \eqref{eq:mcnoiser} and \eqref{eq:hotqmc} we have 
\begin{align*} 
&\beta\left| f(X) - \left(\tilde{s}^2 \|\mathcal{A}\left(\Pi(X) \Delta^T + \Delta \Pi(X)^T\right)\|^2_F + \tilde{s} \langle \mathcal{A}\left(\Pi(X) \Delta^T + \Delta \Pi(X)^T + \Delta \Delta^T\right), n \rangle\right)\right|  \\ 
&\lesssim \sqrt{\frac{\lbdmc}{\beta}} \\
&\lesssim 1
\end{align*} 
where the last inequality follows since $\beta \gtrsim \lbdmc$.

 %

 
Hence, denoting $\tilde{q}(S,Y): \mathbf{B} \to \Real$ the distribution 
\begin{align*} &\tilde{q}(S,Y)  \propto \\
&e^{-\beta^2 \left(\|\mathcal{A}\left(X_0 S (X_0^T X_0)^{-1} X_0^T + X_0 (X_0^T X_0)^{-1} S X_0^T+  X_0 Y^T + Y X_0^T\right)\|^2_2 + 2\langle \mathcal{A}\left(X_0 S (X_0^T X_0)^{-1} X_0^T + X_0 (X_0^T X_0)^{-1} S X_0^T+  X_0 Y^T + Y X_0^T\right), n \rangle\right)}\end{align*} 
by Lemma \ref{l:holleystroock} we have $C_P(q) \lesssim C_P(\tilde{q})$. 

Thus, by Lemma \ref{l:constrained}, it suffices to show that the function 
\begin{align*} &\|\mathcal{A}\left(X_0 S (X_0^T X_0)^{-1} X_0^T + X_0 (X_0^T X_0)^{-1} S X_0^T+  X_0 Y^T + Y X_0^T\right)\|^2_2 \\
&+ 2\langle \mathcal{A}\left(X_0 S (X_0^T X_0)^{-1} X_0^T + X_0 (X_0^T X_0)^{-1} S X_0^T+  X_0 Y^T + Y X_0^T\right), n \rangle\end{align*} 
is convex (viewed as a function of $S, Y$). As the second term is linear (hence convex), it suffices to show the first term is convex. 
 
Vectorizing the matrices, and denoting by $A$ the matrix s.t. $A \mbox{vec}(X) = \mathcal{A}(X)$, the function in question is (using Lemma \ref{l:vectorizeop} repeatedly) 
\begin{equation} \|\left((X_0 (X_0^T X_0)^{-1} \otimes X_0) + (X_0 \otimes X_0 (X_0^T X_0)^{-1}) \right)\mbox{vec}(S) + \left((I \otimes X_0) C + (X_0 \otimes I)\right) \mbox{vec}(Y)\|^2  \label{eq:objconvex} \end{equation}
If we denote by $\mbox{vec}(S,Y)$ the concatenation of the vectors $\mbox{vec}(S), \mbox{vec}(Y)$, and denote 
$$ \begin{blockarray}{r@{}cc}
\begin{block}{r(c c)}
B:={} & (X_0 (X_0^T X_0)^{-1} \otimes X_0) + (X_0 \otimes X_0 (X_0^T X_0)^{-1}) &  (I \otimes X_0) C + (X_0 \otimes I)\\
\end{block}
\end{blockarray}
$$
We can then write \eqref{eq:objconvex} as
$\|B \mbox{vec}(S,Y)\|^2_2 = \mbox{vec}^T(S,Y) B^T B \mbox{vec}(S,Y)$
which is convex. The claim thus follows.
\end{proof} 

%% file: gradtoval.tex
\subsection{Bound on gradient-to-value ratios} 
\label{s:gradienttoval}

In this section, show that $C_{\mbox{change}} = 0$, namely: 
\begin{lem} For $\tilde{p}^{(S,Y)}, \mathbf{M}^{(S,Y)}$ as defined in \eqref{eq:tildeprdef}, it holds that $C_{\mbox{change}} = 0$.   
\end{lem}  
\begin{proof} 

By Lemmas \ref{l:functionvalueconst} and \ref{eq:constjacob} we have $\tilde{p}^{(S,Y)}$ is uniform over $\mathbf{M}^{(S,Y)}$, so 
$$\tilde{p}^{(S,Y)}(X) = \frac{1}{\mbox{vol}(\mathbf{M}^{(S,Y)})} = \frac{1}{\sqrt{\mbox{det}\left((X_0 + X_0 (X_0^T X_0)^{-1} S + Y)^T (X_0 + X_0 (X_0^T X_0)^{-1} S + Y)\right)} \mbox{vol}\left(\mbox{SO}(K)\right)}$$ 
where the second equality follows since $\mathbf{M}^{(S,Y)}$ can be written as the image of the linear map from $\mbox{SO}(K)$, namely $U \to \left(X_0 + X_0 S + Y\right) U$. For the same reason, by Definition \ref{d:volform}, we have 
$$d\mathbf{M}^{(S,Y)}(X) = \sqrt{\mbox{det}\left((X_0 + X_0 (X_0^T X_0)^{-1} S + Y)^T (X_0 + X_0 (X_0^T X_0)^{-1}S + Y)\right)} d\mbox{SO}(K)\left(U\right)$$
This implies that 
$$\tilde{p}^{S,Y}(X) d\mathbf{M}^{(S,Y)}(X) = \frac{1}{\mbox{vol}\left(\mbox{SO}(K)\right)} d\mbox{SO}(K)\left(U\right)$$  
which does not depend on $S,Y$, proving the lemma. 

\end{proof}

%% file: pf-overview_rescaled.tex
\subsection{Putting components together and discretization} 
\label{l:puttogether}


Plugging the bounds from Lemmas \ref{l:pconstr}, \ref{l:rconst} and \ref{s:gradienttoval} in Theorem~\ref{l:abstract}, we almost immediately get part (1) of Theorem 
\ref{t:maincompletion}. The only change is that we wish to prove mixing to the distribution $p_i$ defined in Section \ref{sec:setup} instead of $\tilde{p}_i$ which is supported on $\mathcal{D}^{j}_i$, as defined in \eqref{l:regionmf}, \eqref{l:regionms}, \eqref{l:regionmc}. However, by Lemma \ref{l:separation}, 
$$\mathcal{D}^j_i \subseteq \{X: \|X -  \Pi_{\mathbf{E}_i}(X)\|_F < \|X -  \Pi_{\mathbf{E}_{3-i}}(X)\|_F \}$$
for $\beta$ satisfying the lower bound in Theorem 
and \ref{t:maincompletion}.

First, we briefly take care of the initialization part. We focus on finding a point with the appropriate proximity to one of the manifolds $\lopt_i, i \in \{1,2\}$---absolute continuity is then easily ensured by initializing with a Gaussian with a small variance centered at that point.  

This mostly follows from prior results, but we re-state the guarantees here for completeness. 
\begin{lem} [Initialization guarantees] Starting from an initial point $\tilde{X}_0$, s.t. $\|\tilde{X}_0\|_F \leq R$, a strict-saddle avoiding algorithm (\cite{jin2017escape}) can find a point $X_0$, s.t. 
$$\|X_0 - M^*\|_F \leq 40 \left\{\frac{1}{\sqrt{\beta}} \frac{\sqrt{d k \log d}}{\sigma_{\min}}, \hspace{0.2cm} \frac{1}{\sqrt{\beta}} \frac{1}{\sigma_{\min}}\sqrt{\frac{d k \log L}{L}}, \hspace{0.2cm} \frac{1}{\sqrt{\beta}} \frac{\sqrt{d k \log d}}{p\sigma_{\min}}\right\} $$   
for matrix factorization, sensing and completion respectively. 

Furthermore, the algorithm runs in time $\mbox{poly}(d,\frac{1}{\sigma_{\min}}, \sigma_{\max}, R)$. 
\label{l:initialization}
\end{lem} 
\begin{proof} 
The results essentially follow by (the appropriate version) of Theorem 31 in \citep{ge2017no} and Lemma 5.4 in \citep{tu2015low}. Namely, we will show that any point $X_0$ satisfying the first/second order criticality conditions satisfies the initialization closeness in the statement. The strict-saddle avoiding gradient descent algorithm (e.g. \cite{jin2017escape}) has the required runtime guarantee as per Corollary 17 in \citep{ge2017no}.    

Namely, Theorem 31 in \cite{ge2017no} implies that: 
\begin{itemize} 
\item For matrix factorization: with high probability, any point $X_0$ satisfying the first/second order criticality conditions satisfies  
$$\|X_0 X^T_0 - M^*\|_F \leq 40 \frac{1}{\sqrt{\beta}} \sqrt{d k \log d}$$ 
Subsequently, by Lemma 5.4 in \cite{tu2015low}, we have 
$$\|X_0 - M^*\|_F \leq 40 \frac{1}{\sqrt{\beta}} \frac{\sqrt{d k \log d}}{\sigma_{\min}} $$   
(Notice, alternatively we can get a comparable guarantee by just using the $k$-SVD of $M$ and applying Wedin's theorem.)
\item For matrix sensing: with high probability, any point $X_0$ satisfying the first/second order criticality conditions satisfies  
$$\|X_0 X^T_0 - M^*\|_F \leq 40 \frac{1}{\sqrt{\beta}} \sqrt{\frac{d k \log L}{L}}$$ 
Subsequently, by Lemma 5.4 in \cite{tu2015low}, we have 
$$\|X_0 - M^*\|_F \leq 40 \frac{1}{\sqrt{\beta}} \frac{1}{\sigma_{\min}}\sqrt{\frac{d k \log L}{L}}$$   
\item For matrix completion: with high probability, any point $X_0$ satisfying the first/second order criticality conditions satisfies  
$$\|X_0 X^T_0 - M^*\|_F \leq 40 \frac{1}{\sqrt{\beta}} \frac{\sqrt{d k \log d}}{p}$$ 
Subsequently, by Lemma 5.4 in \cite{tu2015low}, we have 
$$\|X_0 - M^*\|_F \leq 40 \frac{1}{\sqrt{\beta}} \frac{\sqrt{d k \log d}}{p\sigma_{\min}}$$   
\end{itemize} 

\end{proof} 


Finally, we prove the discretization results. These mostly follow previous techniques (essentially applying Girsanov's formula), with minor complications due to the fact that $\nabla f$ does not have a bounded Lipschitz constant.   

%% file: discretization.tex
 
\begin{lem} [Discretization bound] Let $X_t$ follow the SDE $dX_t = -\nabla f(X_t) dt + \frac{1}{\beta} dB_t$, and let's denote by $p_T$ the pdf of $X_t: t \in [0,T]$. \\
Let $\hat{X}_t$ follow the SDE  $d\hat{X}_t = -\nabla f(X_{\lfloor t/h \rfloor}h) dt + \frac{1}{\beta} dB_t$, and let's denote by $\hat{p}_T$ the pdf of $\hat{X}_t: t \in [0,T]$. 
Then, $\mbox{KL}(\hat{p}_T || p_T) \leq \beta^2\mbox{poly}(d,p,\|M\|_F) Th$.   
\end{lem} 
\begin{proof} 

As a notational convenience, let $X_{[0,T]}$ denote a function $X_{[0,T]}: [0,T] \to \mathbb{R}$, s.t. 
$X_{[0,T]}(t) = X_t$. By Girsanov's formula, we have 
\begin{align} &\mbox{KL}(\hat{p}_T, p_T) \nonumber\\
&= \E_{X_{[0,T]} \sim \hat{p}_T} \log(\hat{p}_T(X_{[0,T]})/p_T(X_{[0,T]})) \nonumber\\ 
&= \E_{X_{[0,T]} \sim \hat{p}_T}  \log\left(\exp\left(-\beta \int_{0}^T (\nabla f(X_t) - \nabla f(X_{\lfloor t/h \rfloor}h))^T (dX_t - \nabla f(X_t) dt) + \beta^2 \int_{0}^T \|\nabla f(X_t) - \nabla f(X_{\lfloor t/h \rfloor}h) \|^2 dt \right)\right)  \nonumber\\
&= \beta^2 \E_{X_{[0,T]} \sim \hat{p}_T} \int_{0}^T \|\nabla f(X_t) - \nabla f(X_{\lfloor t/h \rfloor}h) \|^2 dt \label{eq:toplevel}
\end{align} 
For notational convenience, let $\tilde{t} = \lfloor t/h \rfloor h$, and let's denote $\delta := X_t - X_{\lfloor t/h \rfloor}h$ 
We will show that for all $f$, 
$$  \|\nabla f(X_t) - \nabla f(X_{\lfloor t/h \rfloor}h) \|_2 \lesssim \|\delta\|_F^3 + 3 \|\delta\|^2_F \|X_t\|_F + 2 \|\Delta\|_F \|X_t\|^2_F$$ 

We will proceed to matrix factorization, the proof is analogous for the other operators $\mathcal{A}$. We have:
\begin{align*}  \|\nabla f(X_t) - \nabla f(X_{\lfloor t/h \rfloor}h) \|_2 &= \|\delta X^T_t(X_t + \delta) + X_t \delta^T (X_t+\delta) + \delta \delta^T (X_t+\delta)\|_F \\
&\leq \|\delta\|_F^3 + 3 \|\delta\|^2_F \|X_t\|_F + 2 \|\Delta\|_F \|X_t\|^2_F 
\end{align*} 

Furthermore, $\delta = -\nabla f(X_t) (t - \tilde{t}) + \frac{1}{\beta} (W_t - W_{\tilde{t}})$. Denoting by $\xi_t =  \frac{1}{\beta} (W_t - W_{\tilde{t}})$, we have 
by the AM-GM inequality $(a+b)^3 \leq 4 (a^3 + b^3)$ and $(a+b)^2 \leq 2(a^2 + b^2)$, for any $a,b\geq 0$ so 
\begin{align} 
\|\delta\|_F^3 &\leq 4 \left( \|\nabla f(X_t) (t - \tilde{t})\|^3_F +  \|\xi_t\|^3_F \right) \nonumber\\ 
&\leq 4 h^3 (\|M\|_F \|X_t\|_F + \|X_t X^T_t X_t\|_F)^3 + 4 \|\xi_t\|^3_F \nonumber\\ 
&\leq 16 h^3 (\|M\|^3_F + 2\|X_t\|^9_F) + 4 \|\xi_t\|^3_F \label{eq:delta1}
\end{align} 
and similarly, 
\begin{align} 
\|\delta\|_F^2 &\leq 2\left( \|\nabla f(X_t) (t - \tilde{t})\|^2_F +  \|\xi_t\|^2_F \right) \nonumber\\ 
&\leq 4 h^2 (\|M\|^2_F + 2\|X_t\|^4_F) + 2 \|\xi_t\|^2_F \label{eq:delta2}
\end{align}
We will prove that $\E[\|X_t\|^p_F] \leq \mbox{poly}(d, p, \|M\|_F), p \geq 2$, from which the claim will follow. Indeed, by standard Gaussian moment bounds, we have  
$\E[\|\xi_t\|^p_F] \lesssim (\frac{\sqrt{h}}{\beta})^p, p \geq 2$. Together with \eqref{eq:delta1} and \eqref{eq:delta2} (using these inequalities for the appropriate $p$) we have 
$$\|\nabla f(x_t) - \nabla f(x_{\lfloor t/h \rfloor}h) \|_2 \leq \mbox{poly}(d,\|M\|_F) h$$
Plugging this back into \eqref{eq:toplevel}, we have 
\begin{align*} \mbox{KL}(\hat{p}_T, p_T) &\leq \frac{1}{\beta^2} \E_{X_{[0,T]} \sim \hat{p}_T} \sum_{i=0}^{T/h} \mbox{poly}(d,\|M\|_F) h^2 \end{align*}

We turn to bounding the moments $\E[\|X_t\|^p_F]$. by It\'o's Lemma, we have 
\begin{align*} 
d\E[\|X_t\|^p_F] &= \E\left[\left\langle p \|X_t\|^{p-2} X_t, dX_t  \right\rangle + \frac{1}{\beta^2} \mbox{Tr}\left(p \|X_t\|^{p-2}_F I + p(p-2) \|X_t\|^{p-4}_F X_t X^T_t\right)\right] \\ 
&= \E\left[\left\langle p \|X_t\|^{p-2} X_t, dX_t \right\rangle + \frac{1}{\beta^2} d p (p-1) \|X_t\|^{p-2}_F \right] \\
&= \E\left[\left\langle p \|X_t\|^{p-2} X_t, (M - X_t X^T_t) X_t + \frac{1}{\beta} dB_t \right\rangle + \frac{1}{\beta^2} d p (p-1) \|X_t\|^{p-2}_F \right] \\  
\end{align*}
Note that $\left\langle X_t, M X_t \right\rangle \leq \|M\|_F \|X_t\|^2_F$  
and $\left\langle X_t, X_t X^T_t X_t \right\rangle = \|X_t X^T_t\|^2_F$. 
Furthermore, by the power mean inequality,
\begin{align*}  \|X_t X^T_t\|^2_F &\geq \|X_t X^T_t\|^2_2 \\ 
&\geq \frac{1}{d^2} \|X_t\|^4_F \\
\end{align*}

Altogether, we have 
\begin{equation} \left\langle X_t, (M - X_t X^T_t) X_t \right\rangle \leq - \frac{1}{d^2} \|X_t\|^4_F + \|M\|_F \|X_t\|^2_F \label{eq:crossterm}\end{equation} 
Putting together, we get 
\begin{align*} 
d\E[\|X_t\|^p_F] &\leq \E\left[ -\frac{p}{d^2} \|X_t\|^{p+2}_F + p\|M\|_F \|X_t\|^p_F + \frac{1}{\beta^2} d p (p-1) \|X_t\|^{p-2}_F \right] \\  
\end{align*}
Furthermore, we have: 
$$-\frac{p}{d^2} \|X_t\|^{p+2}_F + p\|M\|_F \|X_t\|^p_F + \frac{1}{\beta^2} d p (p-1) \|X_t\|^{p-2}_F \leq -\frac{4}{5} \frac{p}{d^2} \|X_t\|^{p}_F + 20 d^2 \left(\|M\|_F  + \frac{1}{\beta} p (p-1)\right) $$
(This inequality can be immediately checked by separately considering the case that $\|X_t\|_F \leq 1$ and $\|X_t\|_F > 1$.) This then implies:
$$ d\E[\|X_t\|^p_F] \leq  -\frac{4}{5} \frac{p}{d^2} \E\left[\|X_t\|^{p}_F\right] + 20 d^2 \left(\|M\|_F  + \frac{1}{\beta} p (p-1)\right) $$ 

Since the ODE $dY_t = - A Y_t + B $ solves to $Y_t = Y_0 e^{- A t} + \frac{B}{A}$, we then have 
\begin{equation} \E[\|X_t\|^p_F] \leq \|X_0\|_F^p e^{- \frac{4}{5} t \frac{p}{d^2}} + 25 d^2 \left(\frac{\|M\|_F}{p}  + \frac{1}{\beta^2} d p\right) \label{eq:normbound}\end{equation}
which is what we wanted.   
\end{proof} 

Given this lemma and Pinsker's inequality, the discretization part of Theorem~\ref{t:maincompletion} follows.

%% file: helper_noiseless_2.tex
\section{Helper Lemmas about Manifold of Minima} 

First, we calculate the tangent and normal spaces of manifolds that will continually appear in our calculations.  
\begin{lem} Let $\mathbf{M} = \{YU: U \in \mbox{SO}(k)$, for some $Y \in \Real^{d \times k}$. 
Then, the tangent space and normal space at $X \in \mathbf{M}$ satisfy 
$$T_{X}(\mathbf{M}) = \{X R, R \in \skewm^{k \times k}\}, \hspace{1cm} 
N_{X}(\mathbf{M}) = \{X (X^T X)^{-1} S + Y, S \in \symm^{k \times k}, Y^T X = 0\} $$
\label{l:tangent}
\end{lem} 
\begin{proof}
Consider any curve $x(t) \in \mathbf{M}, t \in [0,1]$, s.t. $x(0) = X$. 
Since $x(t) x(t)^T = YY^T$, taking derivatives on both sides, we get 
$$x(t)' x(t)^T + x(t) (x(t)')^T = 0$$  
Evaluating this equation at $t = 0$, we get 
$$x(0)' X^T + X (x(0)')^T = 0 $$ 
All $x(0)'$ of the form $X R$, for $R \in \skewm_k$ clearly satisfy the equation above. 
Since the dimension of $\mathbf{M}$ (and hence it's tangent space) is $\binom{k}{2}$, 
which is the same as the dimension of $\skewm_k$, the tangent space at $X$ is
$T_{X}(\mathbf{M}) = \{X R, R \in \skewm_k\}$, as we need. 

On the other hand, consider a matrix of the form $X(X^T X)^{-1} S + Y$, s.t. $S \in \symm_k$ and $Y^T X = 0$.  For any matrix $XR \in T_{X}(\mathbf{M})$, we have 
\begin{align*} 
\langle \mbox{vec}(XR), \mbox{vec}\left(X(X^T X)^{-1} S + Y\right) \rangle &= \mbox{Tr}\left((XR)^T \left(X(X^T X)^{-1} S + Y\right)\right) \\ 
&= \mbox{Tr}\left((XR)^T X(X^T X)^{-1} S\right) \\
&= \mbox{Tr}(R^T S) \\ 
&= 0 
\end{align*}
where the last equality follows since $S$ is symmetric and $R$ is skew-symmetric. 
The dimension of the space $\{X (X^T X)^{-1} S + Y, S \in \symm^{k \times k}, Y^T X = 0\}$ is 
$(k^2 - \binom{k}{2}) + dk - k^2 = dk - \binom{k}{2}$: this can be seen by parametrizing the symmetric matrices and $Y$ separately, and noting that the symmetric matrices have dimension $k^2 - \binom{k}{2}$ and the space of $Y$ is of dimension $dk - k^2$ (by writing $Y = X^{\perp} Z$ for a matrix $X^{\perp} \in \Real^{d \times (d-k)}$ with columns spanning the orthogonal subspace to the column span of $X$ and $Z \in \Real^{(d-k) \times k}$).  
Hence, it is indeed the normal space at $X$. 
\end{proof}

\begin{lem}[Separation of manifolds] 
Let $X \in \mathbf{E}_1$ and $Y \in \mathbf{E}_2$. Then, $\|X - Y\|_F \geq \frac{2 \sigma_{\min}}{k}$.
\label{l:separation}
\end{lem} 
\begin{proof} 
We have $X = X^* U$ and $Y = X^* V$, for $U \in O(k), \mbox{det}(U) = 1$ and $V \in O(k), \mbox{det}(V) = -1$. 
Then, $\|X - Y\|_F = \|X^*(U-V)\|_F$. If $X^* = U \Sigma V^T$, for $U \in \Real^{N \times k}, \Sigma \in \Real^{k \times k}, V \in \Real^{k \times k}$ let us denote $X^{-1} =  V \Sigma^{-1} U^T$. We have 
\begin{align}\|U-V\|_F &= \|X^{-1} X^*(U-V)\|_F \nonumber\\
&\leq \|X^{-1}\|_F \|X^*(U-V)\|_F  \nonumber\\
&\leq \frac{k}{\sigma_{\min}} \|X^*(U-V)\|_F \label{eq:lbduv}
\end{align} 
From the unitary invariance of the Frobenius norm, to lower bound $\|U-V\|_F$ it suffices to consider $U = I$.
Since $V$ is orthogonal and $\mbox{det}(V) = -1$ it has to have $-1$ as an eigenvalue: namely, the eigenvalues of $V$ are either $\pm 1$ or $e^{i \theta}$. The complex eigenvalues come in conjugate pairs, and their product is $e^{i \theta} e^{-i \theta} = 1$, so at least one eigenvalue must be $-1$. 
 
Consider an eigenvector $u$ of $V$ with eigenvalue $-1$. We then have 
$$u^T(U - V)u = u^T u - (-u^T u) = 2$$
which means $\|U-V\|_F \geq \|U-V\|_2 \geq 2$. 
Plugging this back in \eqref{eq:lbduv}, we get $\|X^*(U-V)\|_F \geq \frac{2 \sigma_{\min}}{k}$. 

\end{proof}

\begin{lem}[Projection onto manifolds $\lopt_i$] Let $X \in \mathbb{R}^{d \times k}$, s.t. $\|X - \Pi_{\lopt_i}(X)\|_F < \|X -\Pi_{\lopt_{3-i}}(X)\|_F, i \in \{1,2\}$. 
$X = X_0 R + V$ be the decomposition of $X$ into the component in the subspace $\mbox{colspan}(X_0)$ and the orthogonal subspace: in particular, $R \in \mathbb{R}^{k \times k}$ is invertible and $\mbox{Tr}(V^{\top} X) = 0$ for any $X \in \mbox{colspan}(X_0)$. Then, the projection to the manifold $\lopt_i$ 
can be described as 
$$ \Pi_{\lopt_i}(X) = X_0 B A^T $$ 
where $A \Sigma B^T$ is the singular value decomposition of $R^T X^T_0 X_0$.    
\label{l:project}
\end{lem} 
\begin{proof} 
The proof is essentially the same as the solution to the Orthogonal Procrustes problem. Consider first the projection onto $\mbox{O}(k)$. We have: 
\begin{align*}
\mbox{argmin}_{O \in \mbox{O}(k)} \| X_0 O - X\|_F &= \mbox{argmin}_{O \in \mbox{O}(k)} \| X_0 O - X_0 R\|^2_F \\ 
&=  \mbox{argmin}_{O \in \mbox{O}(k)} \| O^T X^T_0  - R^T X^T_0\|^2_F
\end{align*} 
The optimal $O$ of this optimization problem is given by the Orthogonal Procrustes problem: namely, 
if $A \Sigma B^T$ is the singular value decomposition of $R^T X^T_0 X_0$, then $O = B A^T$. 



On the other hand, since $\|X - \Pi_{\lopt_i}(X)\|_F < \|X -\Pi_{\lopt_{3-i}}(X)\|_F$, $\Pi_{\lopt_i}(X) = X_0 O$, which proves the claim. 
 
\end{proof} 

Using this, we provide a lower bound on the size of the neighborhood, in which the projection doesn't change along the line $X$ to $\Pi(X)$:
\begin{lem}[Large tubular neighborhood] Let $X$ be s.t. $\|X - \Pi_{\lopt_i}(X)\|_F \leq D, i \in \{1,2\}$ and let $\tilde{X} = \Pi_{\lopt_i}(X) + r (X - \Pi_{\lopt_i}(X))$ for $r < \frac{2 \sigma_{\min}}{kD}$. Then, 
$$\Pi(\tilde{X}) = \Pi(X)$$ 
As a corollary, for any $X \in \lopt_i$, and $r < \frac{2 \sigma_{\min}}{k}$, 
$$\Pi_{\lopt_i}(X + r N) = X$$ 
where $N$ is a unit normal vector in $T^{\perp}_X(\mathbf{M})$.
\label{l:tubular}  
\end{lem}
\begin{proof}
By Lemma~\ref{l:project}, we have $\Pi_i(X) = X_0 (B A^T)$, where $B,A$ are defined s.t. $X = X_0 R + V$ and $A \Sigma B^T$ is the singular value decomposition of $R^T X_0^T X_0$.    
Hence, 
\begin{align*}
\tilde{X} &= X_0 (B A^T) + \alpha(X_0 R + V - X_0 (BA^T)) \\ 
&= X_0 (B A^T) + \alpha(X_0 (X_0^T X_0)^{-1} B \Sigma A^T + V - X_0 (BA^T)) \\ 
&= X_0 ((1-\alpha) BA^T + \alpha (X_0^T X_0)^{-1} B \Sigma A^T) + \alpha V
\end{align*} 
We proceed similarly as in the proof of Lemma~\ref{l:project}. We have: 
\begin{align*}
\min_{O \in \mbox{O}(k)} \|X_0 O - \tilde{X}\|^2_F &= \min_{O \in \mbox{O}(k)} \|X_0 O - X_0 ((1-\alpha) BA^T + \alpha (X_0^T X_0)^{-1} B \Sigma A^T)\|^2_F \\ 
&= \min_{O \in \mbox{O}(k)} \|O^T X^T_0 - ((1-\alpha) AB^T + \alpha A \Sigma B^T (X_0^T X_0)^{-1}) X^T_0 \|^2_F \\
&= \max_{O \in \mbox{O}(k)} (1-\alpha) \langle A B^T X_0^T X_0, O^T\rangle + \alpha \langle A \Sigma B^T, O^T\rangle 
\end{align*} 
Denoting the singular value decomposition of $X_0^T X_0$ by $U_0 \Sigma_0^T U_0^T$, we have 
\begin{align*} \min_{O \in \mbox{O}(k)} \|X_0 O - \tilde{X}\|^2_F &= \max_{O \in \mbox{O}(k)} (1-\alpha) \langle U_0^T BA^T O^T U_0, \Sigma_0 \rangle + \alpha \langle A^T O^T B, \Sigma \rangle
\end{align*} 
As $\Sigma_0, \Sigma$ are PSD matrices, the maximum is reached when $U_0^T BA^T O U_0$ and $A^T O B$ are both identity: this can be achieved (for both simultaneously) if $O = B A^T$. 

Moreover, by Lemma ~\ref{l:separation}, since $\|X - \Pi_{\lopt_i}(X)\| \leq D$ and $r < \frac{2 \sigma_{\min}}{kD}$, we have $\|\tilde{X} - \Pi_{\lopt_i}(\tilde{X})\| < \|\tilde{X} - \Pi_{\lopt_{3-i}}(\tilde{X})\|$, which proves that the projection to $\mbox{O}(k)$ agrees with the projection to $\lopt_i$. Thus, $\Pi_{\lopt_i}(\tilde{X}) = X_0 O$, which implies the statement of the Lemma.


\end{proof} 

\section{Helper Lemmas about Matrix Calculus} 

We will prove a few lemmas about matrix calculus: 
\begin{lem}[Matrix differentials] Let $X, \Delta \in \Real^{d \times k}$. Then, 
$$\frac{\partial \| \Delta X^T + X \Delta^T\|^2_F}{\partial X} = 2\Delta X^T \Delta + 2X \Delta^T \Delta $$ 
\label{l:matrixdifferentials}
\end{lem} 
\begin{proof}
Rewriting $\| \Delta X^T + X \Delta^T\|^2_F = \mbox{Tr}(\Delta^T X \Delta^T X) + \mbox{Tr}(\Delta^T \Delta X^T X) $, we need only calculate the differentials of $\mbox{Tr}(\Delta^T X \Delta^T X)$ and $\mbox{Tr}(\Delta^T \Delta X^T X)$. These follow from standard Lemmas in matrix calculus. We have:
$$\frac{\partial \mbox{Tr}(\Delta^T X \Delta^T X)}{\partial X} = 2\Delta X^T \Delta $$ 
by equation (102) in \cite{petersen2008matrix} and 
$$\frac{\partial \mbox{Tr}(\Delta^T \Delta X^T X)}{\partial X} = 2X \Delta^T \Delta $$ 
by equation (101) in \cite{petersen2008matrix}.  
\end{proof}

We will also frequently switch between viewing matrices as vectors. The following lemma about the vectorizing operator will be useful:  
\begin{lem}[Vectorizing matrices] Let $\mbox{vec}(X): \Real^{m \times n} \to \Real^{mn}$ 
be defined as 
$$\mbox{vec}(X) = (X_{1,1}, X_{2,1}, \dots, X_{m,1}, \dots, X_{1,n}, X_{2,n}, \dots, X_{m,n})^T$$
Then, if $A \in \Real^{m \times n}$ and $B \in \Real^{n \times k}$
$$\mbox{vec}(AB) = (I_k \otimes A) \mbox{vec}(B) = (B^T \otimes I_m) \mbox{vec}(A) $$
Finally, if $A \in \Real^{m \times n}, B \in \Real^{n \times k}, C \in \Real^{k \times l}$, 
$$\mbox{vec}(ABC) = (C^T \otimes A) \mbox{vec}(B) $$
\label{l:vectorizeop}
\end{lem}

Finally, we will need to following simple proposition about Kronecker products: 
\begin{lem}[Kronecker products] The Kronecker product operation satisfies the following properties: 
\begin{enumerate} 
\item If $A,B,C,D$ are matrices of dimensions s.t. the products $AC$, $BD$ can be formed, we have 
$$(A \otimes B)(C \otimes D) = AC \otimes BD $$
\item For invertible matrices $A,B$, we have 
$$(A \otimes B)^{-1} = A^{-1} \otimes B^{-1} $$ 
\item If $\{\lambda_1, \lambda_2, \dots, \lambda_n\}$ are eigenvalues of $A \in \Real^{n \times n}$,  
$\{\mu_1, \mu_2, \dots, \mu_m\}$ are eigenvalues of $B \in \Real^{m \times m}$, the eigenvalues of $A \otimes B$ are $\{\lambda_i \mu_j, 1 \leq i \leq n, 1 \leq j \leq m\}$. 
\end{enumerate}  
\label{l:kronalgebra}
\end{lem}

\section{Helper Lemmas about Gradients} 

In this Section, we collect various estimates about gradients of the functions $f$ we are considering. These are either standard, or follow easily from standard results in the context of matrix completion and sensing  (in particular, they are about ``local restricted convexity'' of the objectives) -- but we write them out here for completeness. 

\begin{lem}[Norms of matrices] Let $X^* \in \mathbf{M}$ and $\Delta \in N_{\mathbf{M}}(X^*)$. Then: \\
$2 \sigma_{\max} \geq \|X^* \Delta^T + \Delta (X^*)^T\|_F \geq \left(\frac{1}{\kappa}\right)^6 \sigma^2_{\min} $ 
\label{l:normstang}
\end{lem} 
\begin{proof}
We proceed to (1) first. We handle the lower bound first. 

Since $\Delta \in N_{\mathbf{M}}(X^*)$, by Lemma \ref{l:tangent} we have $\Delta = X^* ((X^*)^T X^*)^{-1} S + Y$, for $R \in \mathbb{R}^{k \times k}$ symmetric and $Y^T X = 0$. For notational convenience, let us denote 
$\tilde{S} = ((X^*)^T X^*)^{-1} S$.  Then, 
\begin{align} \|X^* \Delta^{\top} &+ \Delta (X^*)^{\top}\|^2_F = \|X^* (\tilde{S}^T + \tilde{S}) (X^*)^T + 
X^* Y^T + Y (X^*)^T\|^2_F \nonumber\\ 
&=\|X^* (\tilde{S}^T + \tilde{S}) (X^*)^T\|^2_F + \|X^* Y^T + Y (X^*)^T\|^2_F + 2 \mbox{Tr}\left(
\left(X^* (\tilde{S}^T + \tilde{S}) (X^*)^T\right) \left(X^* Y^T + Y (X^*)^T\right)\right) \nonumber\\ 
&=\|X^* (\tilde{S}^T + \tilde{S}) (X^*)^T\|^2_F + \|X^* Y^T + Y (X^*)^T\|^2_F \nonumber\\
&= \|X^* (\tilde{S}^T + \tilde{S}) (X^*)^T\|^2_F + 2\|X^* Y^T\|^2_F \label{eq:intermbd1}
\end{align}
where the last two equalities follow by $Y^T X^* = 0$ and cyclicity of the trace operator. 

We will lower bound both of the summands in term. For the first term, 
consider the SVD decomposition $X^* = U \Sigma V^T$, where 
$$U \in \Real^{d \times d}, \Sigma \in \Real^{d \times d}, V \in \Real^{k \times d}$$ 
and $\Sigma$ is diagonal, with only the first $k$ entries on the diagonal non-zero. 
Abusing notation, denote by $\Sigma^{-1}$ the diagonal matrix, s.t. 
$\Sigma^{-1}_{i,i}=\frac{1}{\Sigma_{i,i}}$ if $\Sigma_{i,i} \neq 0$, and $\Sigma^{-1}_{i,i} = 0$ otherwise.  
Also, let us denote $R = \tilde{S}^T + \tilde{S}$ and $D = ((X^*)^T X^*)^{-1}$. 

 Then, 
$$\|\Sigma^{-1} U^{T} X^* R (X^*)^{\top} U \Sigma^{-1}\|_F = \|V^T R V\|_F $$ 
Furthermore, 
$$\|V^T R V\|^2_F = \mbox{Tr}(R V V^{\top} R V V^{\top}) = \mbox{Tr}(R^T R) = \|R\|^2_F$$
From this we have 
\begin{align} \|R\| &= \|\Sigma^{-1} U^{T} X^* R (X^*)^{\top} U \Sigma^{-1}\|_F \nonumber\\
&\leq \|\Sigma^{-1} U^{T}\|_2 \|X^* R (X^*)^{\top} U \Sigma^{-1}\|_F \nonumber\\
&\leq \|\Sigma^{-1} U^{T}\|_2 \|X^* R (X^*)^{\top}\|_F \|U \Sigma^{-1}\|_2 \nonumber\\
&\leq \frac{1}{\sigma_{\min}^2} \|X^* R (X^*)^{\top}\|_F \label{eq:term1bd1} \end{align}
Furthermore, 
\begin{align} 
\|R\|^2_F &= \|\tilde{S}^T + \tilde{S}\|^2_F \nonumber\\ 
&= \|\left(D S\right)^T + D S\|^2_F \nonumber\\
&= \|S D + D S\|^2_F \nonumber\\
&= \sum_{i,j=1}^k \left(D_{i,i}+D_{j,j}\right)^2 S^2_{i,j}\nonumber\\
&\geq \frac{4}{\sigma^4_{\max}}\|S\|^2_F \label{eq:term1bd2}
\end{align} 

Putting \eqref{eq:term1bd1} and \eqref{eq:term1bd2} together, we get 
\begin{equation} \|X^* R (X^*)^{\top}\|_F \geq 2\frac{\sigma^2_{\min}}{\sigma^2_{\max}} \|S\|_F \label{eq:term1final}\end{equation}  

For the second term, we have 
\begin{align}\|X^* Y^T\|^2_F &=\mbox{Tr}((X^*)^T X^* Y^T Y) \nonumber\\
&\geq \sigma^2_{\min} \|Y\|^2_F \label{eq:term2bd}
\end{align} 

Since 
\begin{align} 1 &= \|\Delta\|^2_F \nonumber\\
&= \|X^* \tilde{S}\|^2_F + \|Y\|^2_F \nonumber\\
&\leq \|X^*\|^2_2 \|D\|^2_2 \|S\|^2_F + \|Y\|^2_F \nonumber\\ 
&\leq \frac{\sigma^2_{\max}}{\sigma^4_{\min}}(\|S\|^2_F + \|Y\|^2_F) \label{eq:sumbd}
\end{align} 
Combining this with \eqref{eq:term2bd} and \eqref{eq:term1final} and plugging it in in \eqref{eq:intermbd1}, we get 
$$ \|X^* \Delta^{\top} + \Delta (X^*)^{\top}\|^2_F \geq 
2\frac{\sigma^4_{\min}}{\sigma^4_{\max}} (\|S\|^2_F + \|Y\|^2_F) \geq 2\frac{\sigma^8_{\min}}{\sigma^8_{\max}} $$
 
For the left part, we only need note 
$$\|X^* \Delta^T + \Delta (X^*)^{\top}\|_F \leq 2 \|X^*\|_2\|\Delta\|_F \leq 2 \sigma_{\max}$$ 
by the triangle inequality and submultiplicativity of the Frobenius norm.

\end{proof} 

\begin{lem} Let $X \in \mathcal{D}^{\mbox{e}}_i, e \in \{\mbox{mf, ms, mc}\}$. Then, for the corresponding measurement operators $\mathcal{A}$ and losses $f$, with high probability over $\{n_i, i \in [L]\}$ we have: 
\begin{enumerate} 
\item For $\mathcal{A}$ corresponding to matrix factorization, 
$$\langle \nabla f(X), X - \Pi(X)\rangle \geq \frac{1}{16}\sigma^2_{\min} \|X-\Pi(X)\|^2_F - 16 \frac{k^2 \kappa^2 d}{\beta^2}$$   
\item For $\mathcal{A}$ corresponding to matrix sensing, 
$$\langle \nabla f(X), X - \Pi(X)\rangle \geq \frac{1}{16}\sigma^2_{\min} \|X-\Pi(X)\|^2_F - 200 \frac{d k \kappa^2 \log \numm}{\beta^2}$$  
\item For $\mathcal{A}$ corresponding to matrix completion, 
$$\langle \nabla f(X), X - \Pi(X)\rangle \geq \frac{p}{16 \kappa^4} \|X-\Pi(X)\|^2_F - 400 \frac{d k^3 \kappa^2 \log d}{p \beta^2}$$ 
\end{enumerate}  
\label{l:gradcorr}
\end{lem} 
\begin{proof}

(1): For $\mathcal{A}$ corresponding to matrix factorization, we have  
\begin{align*} 
\nabla f(X) &= (M - XX^T) X  \\ 
&= (M^* - XX^T)X + (M-M^*)X 
\end{align*} 
By Lemma 5.7 in \cite{tu2015low}, since $X \in \mathbf{D}_i$, we have 
$$\langle (M^* - XX^T)X, X - \Pi(X) \rangle \geq \frac{1}{4}\sigma^2_{\min} \|X - \Pi(X)\|^2 $$ 
Furthermore, we have 
\begin{align*} 
\langle \mbox{vec}((M-M^*)X), X - \Pi(X) \rangle &= \mbox{Tr}\left(((M-M^*)X)^T(X-\Pi(X))\right) \\ 
&\leq \|(M-M^*)X\|_F \|X-\Pi(X)\|_F\\
&\stackrel{\mathclap{\circled{1}}}{\leq} \|M-M^*\|_2 \|X\|_F \|X-\Pi(X)\|_F  \\
&\stackrel{\mathclap{\circled{2}}}{\leq} \frac{\sqrt{d}}{\beta} \|\Pi(X) + X - \Pi(X)\|_F \|X-\Pi(X)\|_F \\
&\leq \frac{\sqrt{d}}{\beta} \left(k \sigma_{\max} + \|X-\Pi(X)\|_F\right) \|X-\Pi(X)\|_F \\
&= k \frac{\sqrt{d}}{\beta} \sigma_{\max} \|X-\Pi(X)\|_F + \frac{\sqrt{d}}{\beta} \|X-\Pi(X)\|^2_F  
\end{align*} 
where $\circled{1}$ follows from $\|AB\|_F \leq \|A\|_2 \|B\|_F$, and $\circled{2}$ with high probability since $M-M^*$ is a random Gaussian matrix.

Finally, we have  
\begin{align*} 
&\frac{1}{4}\sigma^2_{\min} \|X-\Pi(X)\|^2_F - k \frac{\sqrt{d}}{\beta} \sigma_{\max} \|X-\Pi(X)\|_F - \frac{\sqrt{d}}{\beta} \|X-\Pi(X)\|^2_F \\
&= 
\left(\frac{1}{4}\sigma^2_{\min} - \frac{\sqrt{d}}{\beta}\right) \|X-\Pi(X)\|^2_F - k \frac{\sqrt{d}}{\beta} \sigma_{\max}\|X-\Pi(X)\|_F  \\
&\stackrel{\mathclap{\circled{1}}}{\geq} \frac{1}{8}\sigma^2_{\min} \|X-\Pi(X)\|^2_F - k \frac{\sqrt{d}}{\beta} \sigma_{\max}\|X-\Pi(X)\|_F \\ 
&\stackrel{\mathclap{\circled{2}}}{\geq} \frac{1}{16}\sigma^2_{\min} \|X-\Pi(X)\|^2_F - 16 k^2 \kappa^2 \frac{d}{\beta^2} \\ 
\end{align*}   
where $\circled{1}$ follows since $\beta \geq 16 \frac{\sqrt{d}}{\kappa^2} $, and $\circled{2}$ since the quadratic \Anote{this bd}
$$\frac{1}{16}\sigma^2_{\min} \|X-\Pi(X)\|^2_F - \frac{\sqrt{d}}{\beta} \sigma_{\max}\|X-\Pi(X)\|_F + 16 k^2 \kappa^2 \frac{d}{\beta^2} $$
has no real roots. Hence, we have 
\begin{equation} \langle \nabla f(X), X - \Pi(X)\rangle \geq \frac{1}{16}\sigma^2_{\min} \|X-\Pi(X)\|^2_F - 16 k^2 \kappa^2 \frac{d}{\beta^2}\label{eq:firsttermbound}\end{equation}
which completes the bound on the first term for $\mathcal{A}$ corresponding to matrix factorization. 

(2): Proceeding to $\mathcal{A}$ corresponding to matrix sensing, we have 
\begin{align*} 
\nabla f(X) &= \sum_{i=1}^M \left(\langle A_i, XX^T\rangle - b_i\right) A_i X \\
&=\sum_{i=1}^M \left(\langle A_i, XX^T - M^*\rangle \right) A_i X + \sum_{i=1}^M n_i A_i X \\ 
\end{align*} 
By Lemma 5.7 in \cite{tu2015low}, since $X \in \mathbf{D}_i$, we have 
$$\langle \sum_{i=1}^M \left(\langle A_i, XX^T - M^*\rangle \right) A_i X, X - \Pi(X) \rangle \geq \frac{1}{4}\sigma^2_{\min} \|X-\Pi(X)\|^2_F $$  
On the other hand, by Lemma 34 in \cite{ge2017no}, we have 
\begin{align*}
\langle \sum_{i=1}^{\numm} n_i A_i X, \nabla \eta(X) \rangle &= \langle \sum_{i=1}^{\numm} n_i A_i X, X - \Pi(X) \rangle \\
&= \langle \sum_{i=1}^{\numm} n_i A_i, (X - \Pi(X))X^T \rangle \\ 
&\leq \frac{10}{\beta} \sqrt{d k \log \numm}\|X - \Pi(X)\|_F \|X\|_F \\
&= \frac{10}{\beta} \sqrt{d k \log \numm} \|X-\Pi(X)\|_F \|\Pi(X) + X - \Pi(X)\|_F \\
&\leq \frac{10}{\beta} \sqrt{d k \log \numm} \|X-\Pi(X)\|_F \left(k \sigma_{\max} + \|X-\Pi(X)\|_F\right)\\
&= \frac{10}{\beta} \sqrt{d k^3 \log \numm} \sigma_{\max} \|X-\Pi(X)\|_F + \frac{10}{\beta} \sqrt{d k \log \numm} \|X-\Pi(X)\|^2_F       
\end{align*} 

Finally, we also have  
\begin{align*} 
&\frac{1}{4}\sigma^2_{\min} \|X-\Pi(X)\|^2_F - \frac{10}{\beta} \sqrt{\numm d k^3 \log \numm} \sigma_{\max} \|X-\Pi(X)\|_F - \frac{10}{\beta} \sqrt{d k \log \numm} \|X-\Pi(X)\|^2_F \\
&= \left(\frac{1}{4}\sigma^2_{\min} - \frac{10}{\beta} \sqrt{d k \log \numm}\right) \|X-\Pi(X)\|^2_F - \frac{10}{\beta} \sqrt{d k \log \numm} \sigma_{\max}\|X-\Pi(X)\|_F  \\
&\stackrel{\mathclap{\circled{1}}}{\geq} \frac{1}{8}\sigma^2_{\min} \|X-\Pi(X)\|^2_F - \frac{10}{\beta} \sqrt{d k \log \numm} \sigma_{\max}\|X-\Pi(X)\|_F \\ 
&\stackrel{\mathclap{\circled{2}}}{\geq} \frac{1}{16}\sigma^2_{\min} \|X-\Pi(X)\|^2_F - \frac{200}{\beta^2} d k \kappa^2 \log \numm \\ 
\end{align*}   
where $\circled{1}$ follows since $\beta \geq \frac{200}{\sigma_{\min}^2}\sqrt{d k \log \numm}$ \Anote{this bd}, and $\circled{2}$ since the quadratic 
$$\frac{1}{16}\sigma^2_{\min} \|X-\Pi(X)\|^2_F -\frac{1}{\beta} \sqrt{d k \log \numm} \sigma_{\max}\|X-\Pi(X)\|_F + \frac{200}{\beta^2} d k \kappa^2 \log \numm $$
has no real roots. Hence, we have 
\begin{equation} \langle \nabla f(X), X - \Pi(X)\rangle \geq \frac{1}{16}\sigma^2_{\min} \|X-\Pi(X)\|^2_F -  \frac{200}{\beta^2} d k \kappa^2 \log \numm \label{eq:firsttermboundsensing}\end{equation}
which completes the bound on the first term for $\mathcal{A}$ corresponding to matrix sensing.

(3). For $\mathcal{A}$ corresponding to matrix completion, we have 
\begin{align*} 
\nabla f(X) &= \left(P_{\Omega}(XX^T\rangle - M)\right) X \\
&=\left(P_{\Omega}(XX^T\rangle - M^*)\right) X +  \left(P_{\Omega}(M^* - M)\right)X \\ 
\end{align*} 
We handle $\langle \left(P_{\Omega}(XX^T\rangle - M^*)\right) X, X - \Pi(X) \rangle$ first. 
For notational convenience, let us denote $\Delta:= X - \Pi(X)$, as well as denote $a:= \Pi(X) \Delta^T + \Delta \Pi(X)^T$ and $b:= \Delta \Delta^T$. We then have:  
\begin{align*} 
2\langle \left(P_{\Omega}(XX^T\rangle - M^*)\right) X, \nabla \eta(X) \rangle &= 2\langle \left(P_{\Omega}(XX^T\rangle - M^*)\right) X, \Delta \rangle \\ 
&= 2\langle P_{\Omega}(a+b) X, \Delta \rangle \\
&= \langle P_{\Omega}(a+b), \Delta X^T + + X \Delta^T \rangle \\ 
&= \langle P_{\Omega}(a+b), P_{\Omega}(a + 2b) \rangle \\ 
&= \langle \|P_{\Omega}(a)\|^2_F + 2\|P_{\Omega}(b)\|^2_F + 3 \langle P_{\Omega}(a), P_{\Omega}(b) \rangle \\ 
&\geq \langle \|P_{\Omega}(a)\|^2_F + 2\|P_{\Omega}(b)\|^2_F - 3 \|P_{\Omega}(a)\|_F\|P_{\Omega}(b)\|_F \\  
\end{align*} 
We will lower bound $\|P_{\Omega}(a)\|^2_F$ and upper bound $\|P_{\Omega}(b)\|^2_F$. The first term can be lower bounded, intuitively because $a \in T_{\Pi(X)}(\mathbf{E}_i)$. This is a standard Lemma in matrix completion --  e.g. by Lemma C.6 in \cite{ge2016matrix} and Lemma \ref{l:normstang}, we have    
\begin{equation} \|P_{\Omega}(a)\|^2_F \geq \frac{5}{6} \|a\|^2_F \geq \frac{5p \sigma^2_{\min}}{6 \kappa^4} \eta(X) \label{eq:projlbd}\end{equation}
Upper bounding $\|P_{\Omega}(b)\|_F$, we have, by (56)-(58) in \cite{sun2016guaranteed}, that there exist some constant $C_1$, s.t. for $p = \Omega(\mbox{poly}(k,\kappa,\mu))$ \Anote{maybe drop}      
\begin{equation} \|P_{\Omega}(b)\|^2_F \leq C_1 p \eta(X) \label{eq:projubd} \end{equation}    
Putting \eqref{eq:projlbd} and \eqref{eq:projubd} together, we have 
\begin{equation} 
2\langle \left(P_{\Omega}(XX^T\rangle - M^*)\right) X, X - \Pi(X) \rangle \geq \frac{p \sigma^2_{\min}}{2\kappa^4} \|X-\Pi(X)\|^2_F
\end{equation} 
We handle $\langle \left(P_{\Omega}(M^* - M)\right)X, X-\Pi(X) \rangle$ next: we have  
\begin{align*}
\langle \left(P_{\Omega}(M^* - M)\right)X, X-\Pi(X) \rangle &= \langle P_{\Omega}(M^* - M), P_{\Omega}\left((X - \Pi(X))X^T\right) \rangle 
\end{align*} 
To bound the RHS term, we will use Hoeffding's inequality, along with an epsilon-net argument. 
Denoting $Y = (X - \Pi(X))X^T$, we have 
$$\langle P_{\Omega}(M^* - M), P_{\Omega}\left((X - \Pi(X))X^T\right) \rangle  = \sum_{i,j=1^d} P_{i,j} n_{i,j} Y_{i,j}$$ 
where $n_{i,j} \sim N(0,\frac{1}{\beta})$ are independent Gaussian samples, and $P_{i,j} \sim \mbox{Ber}(p)$ are independent samples from a Bernoulli distribution.  
We will show that with high probability,     
\begin{equation} \sum_{i,j=1}^d P_{i,j} n_{i,j} Y_{i,j}  \leq \frac{20}{\beta}\sqrt{d \log d}\|Y\|_F \label{eq:highprobdevmc}\end{equation}
By scaling, it suffices to show this inequality for $\|Y\|_F = 1$. Consider a $1/d$-net of rank-$k$ matrices with Frobenius norm at most 1: namely, a set $\Gamma$, s.t. $\forall Y \in \Real^{d \times d}, \|Y\|_F = 1$ of rank $k$, $\exists \hat{Y} \in \Gamma$, s.t. $\|Y - \hat{Y}\|_F \leq \frac{1}{d}$. By Lemma E.3 in \cite{ge2016matrix}, such a set $\Gamma$, s.t. $|\Gamma| \leq d^{10dk}$ exists.     

Furthermore, for a fixed $Y$, by Hoeffding's inequality (applied to the sub-Gaussian variables $P_{i,j} n_{i,j} Y_{i,j}$), we have, 
$$\Pr\left[\sum_{i,j=1}^d P_{i,j} n_{i,j} Y_{i,j} \geq \frac{10 \sqrt{d k \log d}}{\beta}\right] \leq e^{-100 d k \log d}$$	
Hence, we have that with high probability $1 - e^{-\Omega(d k \log d)}$, 
\begin{equation}\forall \hat{Y} \in \Gamma, \sum_{i,j=1}^d P_{i,j} n_{i,j} \hat{Y}_{i,j} \leq \frac{10}{\beta} \sqrt{d k \log d}\label{eq:epsnetdev}\end{equation}   
Furthermore, with probability $1 - \exp(-\log^2 d)$, we also have $\forall Y \in \Real^{d \times d}, \|Y\|_F = 1$ of rank $k$, and $\hat{Y} \in \Gamma$, s.t. $\|Y - \hat{Y}\|_F \leq \frac{1}{d}$: 
\begin{align*} 
\sum_{i,j=1}^d P_{i,j} n_{i,j} (Y_{i,j} - \hat{Y}_{i,j}) &\stackrel{\mathclap{\circled{1}}}{\leq} \sqrt{\sum_{i,j=1}^d n^2_{i,j}} \sqrt{\sum_{i,j=1}^d (Y_{i,j} - \hat{Y}_{i,j})^2} \\ 
&\stackrel{\mathclap{\circled{2}}}{\leq} 2 \frac{d}{\beta} \frac{1}{d}\\   
&= \frac{1}{\beta} 
\end{align*} 	
where $\circled{1}$ follows by Cauchy-Schwartz, and $\circled{2}$ with probability $1 - \exp(-\log^2 d)$ by standard tail bounds for Chi-square variables. Combining this with \eqref{eq:epsnetdev}, we have that with high probability, \eqref{eq:highprobdevmc} follows. 

Estimating the Frobenius norm of $\|Y\|_F$, we have with high probability, 
\begin{align*} 
\langle P_{\Omega}(M^* - M), P_{\Omega}\left((X - \Pi(X))X^T\right) \rangle &\leq \frac{20}{\beta} \sqrt{d k \log d} \|(X - \Pi(X))X^T\|_F \\ 
&\leq \frac{20}{\beta} \sqrt{d k \log d} \|X\|_F \|X - \Pi(X)\|_F \\ 
&\leq \frac{20}{\beta} \sqrt{d k \log d} \left(k \sigma_{\max} + \|X-\Pi(X)\|_F\right) \|X - \Pi(X)\|_F \\ 
&=\frac{20}{\beta} \sqrt{d k^3 \sigma^2_{\max}\log d} \|X-\Pi(X)\|_F + \frac{20}{\beta} \sqrt{d k^3 \log d} \|X-\Pi(X)\|^2_F 
\end{align*} 

Finally, we also have  
\begin{align*} 
&\frac{p \sigma^2_{\min}}{2\kappa^4} \|X-\Pi(X)\|^2_F - 20\frac{\sqrt{\sigma^2_{\max} d k^3 \log d}}{\beta}\|X-\Pi(X)\|_F - 20\frac{\sqrt{d k^3 \log d}}{\beta} \|X-\Pi(X)\|^2_F  \\
&= \left(\frac{p}{2 \kappa^4} - 20\frac{\sqrt{d  k^3 \log d}}{\beta}\right) \|X-\Pi(X)\|^2_F - 20\frac{\sqrt{\sigma^2_{\max} k^3 \log d}}{\beta}\|X-\Pi(X)\|_F\\
&\stackrel{\mathclap{\circled{1}}}{\geq} \frac{p \sigma^2_{\min}}{4\kappa^4} \|X-\Pi(X)\|^2_F - 20\frac{\sqrt{\sigma^2_{\max} d k^3 \log d}}{\beta}\|X-\Pi(X)\|_F \\ 
&\stackrel{\mathclap{\circled{2}}}{\geq} \frac{p \sigma^2_{\min}}{16\kappa^4}\|X-\Pi(X)\|^2_F - 400 \frac{d \kappa^4 k^3 \log d}{p \beta^2} \\ 
\end{align*}   
where $\circled{1}$ follows since $\beta \geq 300 \frac{\sqrt{d k^3 \log d}\kappa^4}{\beta p \sigma^2_{\min}}$ \Anote{this bd}, and $\circled{2}$ since the quadratic 
$$\frac{p \sigma^2_{\min}}{4\kappa^4} \|X-\Pi(X)\|^2_F - 20\frac{\sqrt{\sigma^2_{\max} d k^3 \log d}}{\beta}\|X-\Pi(X)\|_F + 400 \frac{d \kappa^4 k^3 \log d}{p \beta^2} $$
has no real roots. Hence, we have 
\begin{equation} \langle \nabla f(X), X - \Pi(X)\rangle \geq \frac{p \sigma^2_{\min}}{16\kappa^4} \|X-\Pi(X)\|^2_F - 400 \frac{d \kappa^2 k^3 \log d}{p \beta^2} \label{eq:firsttermboundsensing}\end{equation}
which completes the bound on the first term for $\mathcal{A}$ corresponding to matrix completion.  
  
\end{proof} 

\pagebreak 
\section{Proof of Posterior Proposition} 
\label{s:posterior}
\begin{proof}[Proof of Proposition \ref{p:posterior}] 
By Bayes rule, we have $p(X|b) \propto e^{-\beta^2 f(X)}$. We will first show that the partition functions of $p(X|b)$, which we'll denote $Z$ and $\tilde{p}(X)$, which we'll denote $\tilde{Z}$ are close to each other.   

We have 
\begin{equation} 
\tilde{Z} = Z + \int_{\|X\|_F \geq \alpha} e^{-\beta^2 f(X)} \label{eq:zbound}
\end{equation}
so we immediately have $\tilde{Z} \geq Z$. Next, we upper bound $\tilde{Z}$.  
We have: 
\begin{align*} 
\|\mathcal{A}(XX^T) - b\|_2 &\geq \|\mathcal{A}(XX^T)\|_2 - \|b\|_2 
\end{align*} 
For all $\mathcal{A}$ of interest, there is a constant $c > 0$, s.t. $\|\mathcal{A}(XX^T)\|_2 \geq c \|XX^T\|_F$, \Anote{write this explicitly} which implies that    
\begin{align*} 
\|\mathcal{A}(XX^T) - b\|_2 &\geq c\|X\|^2_F - \|b\|_2 \\ 
&\geq \frac{c \|X\|^2_F} {2} 
\end{align*} 
where the last inequality follows for $\alpha$ sufficiently large. 
Hence, we have 
\begin{align*} 
\int_{\|X\|_F \geq \alpha} e^{-\beta^2 f(X)} &\leq \int_{r \geq \alpha} e^{\frac{-c \beta^2 r^2}{2}} (2 \pi r)^{d^2} dr \\ 
&= \int_{r \geq \alpha} e^{\frac{-c \beta^2 r^2}{2}+ d^2 \log(2 \pi r)} dr\\
&\stackrel{\mathclap{\circled{1}}}{\leq} \int_{r \geq \alpha} e^{\frac{-c \beta^2 r^2}{4}} dr \\ 
&\leq \int_{r \geq \alpha} e^{\frac{-c \beta^2 r^2}{4}} dr \\
&\leq \int_{r \geq \alpha} \frac{r}{\alpha} e^{\frac{-c \beta^2 r^2}{4}} dr\\
&\stackrel{\mathclap{\circled{2}}}{\leq} 2\frac{e^{\frac{-c \beta^2 \alpha^2}{4}}}{\alpha c \beta^2}
\end{align*} 
where $\circled{1}$ follows for large enough $\alpha$, and $\circled{2}$ by immediate integration. Plugging this back in \eqref{eq:zbound}, we get, for some $g(\alpha)$, s.t. $g(\alpha) \to 0$ as $\alpha \to \infty$:  
\begin{equation}
Z \leq \tilde{Z} \leq Z(1+g(\alpha))
\label{eq:zbound2} 
\end{equation} 
From this, we can also bound $\Pr_{\tilde{p}}[\|X\|_F \geq \alpha]$: 
\begin{align*}
\Pr_{\tilde{p}}[\|X\|_F \geq \alpha] &=1 - \int_{\|X\|_F < \alpha} \frac{e^{-\beta^2 f(X)}}{\tilde{Z}} \\
&\leq 1 - \int_{\|X\|_F < \alpha} \frac{e^{-\beta^2 f(X)}}{Z(1+g(\alpha))}\\ 
&=\frac{g(\alpha)}{1+g(\alpha)} 
\end{align*} 
From this, we can immediately get the tv distance bound in the Lemma: 
\begin{align*} 
\mbox{TV}(p(\cdot | b), \tilde{p}) &= 1/2 \int_{X \in \Real^{d \times d}} |p(X| b) - \tilde{p}(X)| \\ 
&= 1/2 \left(\int_{X:\|X\|_F \leq \alpha} |p(X| b) - \tilde{p}(X)| + \Pr_{\tilde{p}}[\|X\|_F \geq \alpha]\right)\\  
&= 1/2 \left(\int_{X:\|X\|_F \leq \alpha} \left|\frac{e^{-\beta^2 f(X)}}{Z}(1 - \frac{1}{1+g(\alpha)})\right| + \Pr_{\tilde{p}}[\|X\|_F \geq \alpha]\right) \\ 
&\leq \frac{g(\alpha)}{1+g(\alpha)} 
\end{align*} 
As $g(\alpha) \to 0$ when $\alpha \to \infty$, the claim follows. 
\end{proof}